\documentclass[reqno]{amsart}
\usepackage[normalem]{ulem}
\usepackage{amsmath,amssymb,color}
\usepackage{amsfonts, amscd, epsfig, amsmath, amssymb,enumerate}
\usepackage{graphicx}
\usepackage{color}
\usepackage{mathrsfs}
\usepackage{mathptmx}       
\usepackage{helvet}         
\usepackage{courier}        

\usepackage[normalem]{ ulem }
\usepackage{soul}

\newtheorem{theorem}{Theorem}[section]
\newtheorem{lemma}[theorem]{Lemma}
\newtheorem{proposition}[theorem]{Proposition}
\newtheorem{corollary}[theorem]{Corollary}
\newtheorem{remark}[theorem]{Remark}
\newtheorem{definition}{Definition}[section]

\usepackage{tikz}
\usetikzlibrary{backgrounds}
\usetikzlibrary{patterns,fadings}
\usetikzlibrary{arrows,decorations.pathmorphing}
\usetikzlibrary{decorations}
\usetikzlibrary{calc}
\usetikzlibrary{shapes.misc}

\definecolor{light-gray}{gray}{0.95}
\usepackage{float}
\usepackage[colorlinks=true,linkcolor=blue,citecolor=magenta]{hyperref}
\def\centerarc[#1](#2)(#3:#4:#5){\draw[#1] ($(#2)+({#5*cos(#3)},{#5*sin(#3)})$) arc (#3:#4:#5);}

\newcommand{\Clock}[5][shift={(0,0)}]%
{\begin{scope}[#1,scale={#2/5},transform shape]
\fill[color=cyan!30] (0,0) circle (5);
\foreach \a in {1,2,...,60}
   \draw[black] ($(6*\a:4.5)$) -- ($(6*\a:5)$);
\foreach \a in {5,10,...,60} \draw[black,fill,line width={#2*.4pt}]
      ($(6*\a:3.5)$) circle (2pt) -- ($(6*\a:5)$);
\foreach \a in {1,2,...,12}
   \node[text=black,fill=blue!50]
        at ($(90-\a*30:4.3)$) {\Large \bf \a};
\draw[->,>=stealth',line width={#2*1.3pt}]
   (0,0) -- ($(450-30*#3-0.5*#4-0.0083*#5:3.25)$);
\draw[line width={#2*1.1pt},gray]
   (0,0) -- ($(450-6*#4-0.1*#5:4.2)$);
\draw[->,>=stealth',line width={#2*.7pt}]
   (0,0) -- ($(450-6*#4-0.1*#5:4.45)$);
\draw[>->,>=stealth',line width={#2*.6pt},color=gray!80!blue]
   ($(270-6*#5:1)$) -- ($(450-6*#5:4.4)$);
\draw[black,line width={#2*.2pt},fill=gray] (0,0) circle (.2);
\draw[black,fill] (0,0) circle (.1);
\draw[black,line width={#2*.6pt}] (0,0) circle (5);
\end{scope}}

\numberwithin{equation}{section}
\numberwithin{figure}{section}
\newcommand{\mc}[1]{{\mathcal #1}}

\newcommand{\bb}[1]{{\mathbb #1}}

\renewcommand{\epsilon}{\varepsilon}

\newcommand{\R}{\mathbb R}

\newcommand{\Z}{\mathbb Z}
\newcommand{\N}{\mathbb N}
\renewcommand{\P}{\mathbb P}
\newcommand{\T}{\mathbb T}
\newcommand{\E}{\mathbb E}
\newcommand{\Q}{\mathbb Q}

\renewcommand{\bar}{\overline}
\renewcommand{\tilde}{\widetilde}
\renewcommand{\hat}{\widehat}

\newcommand{\abs}[1]{\;\left\vert\;#1 \;\right\vert\;}
\newcommand{\abss}[1]{ \vert \,#1 \,\vert}

\newcommand{\normm}[1]{{\left\vert\kern-0.1ex\left\vert\kern-0.1ex\left\vert\; #1 \; \right\vert\kern-0.1ex\right\vert\kern-0.1ex\right\vert}}    
\newcommand{\cro}[1]{\left[#1\right]}
\newcommand{\pa}[1]{\left(#1\right)}
\newcommand{\tm}{\widetilde{\mu}_N}
\newcommand{\gene}{\mathcal{L}}
\newcommand{\ccl}[1]{{#1}}

\renewcommand{\leq}{\leqslant}
\renewcommand{\geq}{\geqslant}
\renewcommand{\le}{\leqslant}
\renewcommand{\ge}{\geqslant}

\allowdisplaybreaks 

\title{Hydrodynamic limit for a \ccl{facilitated exclusion process}}
\author{Oriane Blondel}
\address{Univ Lyon, CNRS, Universit\'e Claude Bernard Lyon 1, UMR5208, Institut Camille Jordan, F-69622 Villeurbanne, France}
\email{blondel@math.univ-lyon1.fr}
\author{Cl\'ement Erignoux}
\address{Equipe PARADYSE, Bureau B211
Centre INRIA Lille Nord-Europe
Park Plaza, Parc scientifique de la Haute-Borne, 40 Avenue Halley B\^atiment B, 59650 Villeneuve-d'Ascq
France}
\email{clement.erignoux@inria.fr}
\author{Makiko Sasada}
\address{Graduate School of Mathematical Sciences, University of Tokyo, 3-8-1, Komaba, Meguro-ku, Tokyo, 153--8914, Japan}
\email{sasada@ms.u-tokyo.ac.jp}
\author{Marielle Simon}
\address{Inria, Univ. Lille, CNRS, UMR 8524 - Laboratoire Paul Painlev\'e, F-59000 Lille}
\email{marielle.simon@inria.fr}

\thanks{We thank Patr\'icia Gon\c calves, Claudio Landim, Cristina Toninelli and Augusto Teixeira for helpful discussions. 
O.B.\ and M.S.\ are grateful to the University of Tokyo for its hospitality. O.B.\ acknowledges support from ANR-15-CE40-0020-03 grant LSD and ANR-16-CE93-0003 grant MALIN. M.S.\ thank Labex CEMPI (ANR-11-LABX-0007-01) and acknowledges support from the project EDNHS ANR-14-CE25-0011 of the French National Research Agency (ANR). O.B.~and M.S.~thank INSMI (CNRS) for its support through the PEPS project ``D\'erivation et \'Etude Math\'ematique de l'\'Equation des Milieux Poreux'' (2016). M.S. was supported by JSPS Grant-in-Aid for Scientific Research (B) (Generative Research Fields) No.16KT0021. The research leading to the present results benefited from the financial support of the seventh Framework Program of the European Union (7ePC/2007-2013), grant agreement n\textsuperscript{o}266638. This project has received funding from the European Research Council (ERC) under the European Union's Horizon 2020 research and innovative programme (grant agreement  n\textsuperscript{o}715734).   Part of this work was done during the authors' stay at the Institut Henri Poincare - Centre Emile Borel during the trimester ``Stochastic Dynamics Out  of Equilibrium''. The authors thank this institution for its hospitality and support.}

\begin{document}

\begin{abstract}
We study the hydrodynamic limit for a periodic $1$-dimensional exclusion process with a dynamical constraint, which prevents a particle at site $x$ from jumping to site $x\pm1$ unless site $x\mp1$ is occupied.
This process with degenerate jump rates admits transient states, which it eventually leaves to reach an ergodic component, assuming that the initial macroscopic density is larger than $\frac 12$, or one of its absorbing states if this is not the case. It belongs to the class of conserved lattice
gases (CLG) which have been introduced in the physics literature as systems with active-absorbing phase transition in the presence of a conserved field. We show that, for initial profiles smooth enough and uniformly larger than the critical density $\frac{1}{2}$, the macroscopic density
profile for our dynamics evolves under the diffusive time scaling according to a fast diffusion equation (FDE). The first step in the proof is to show that the system typically reaches an ergodic component in subdiffusive time.
\end{abstract}

\maketitle

\section{Introduction}

The procedure of deriving the partial differential equation (PDE) ruling the macroscopic density profile of a microscopic stochastic dynamics under some suitable 
space-time rescaling is referred to in the probability community as \emph{hydrodynamic limit} (cf. \cite{KL}). 
There is a vast literature on the hydrodynamic limit for exclusion processes, namely systems of particles interacting on a 
lattice, where only one particle per site is allowed. The subject has been receiving much attention for decades since 
such models are   simple enough for rigorous mathematical study but also  rich enough to describe many interesting phenomena. 
In this article, we derive the hydrodynamic limit for a $1$-dimensional exclusion model with dynamical constraints, 
which we refer to as \ccl{\emph{facilitated exclusion process} (\ccl{FEP}) following the previously established terminology} \cite{BBCS, GKR} \footnote{\ccl{This system is also called \emph{restricted} \cite{BM} in the literature. It further relates with} activated random walks (which display a similar active-absorbing transition) rather than facilitated or kinetically constrained spin models, in which the constraint is typically chosen to make the dynamics reversible.}, in which a particle present at site $x$ jumps at rate one to site 
$x\pm 1$ assuming that site $x\mp1 $ is occupied by a particle.

This model was originally introduced as a system with active-absorbing phase transition in the presence of a conserved field \cite{RPV}. 
Such models are referred to as conserved lattice gases (CLG) in the context of the study of non-equilibrium phase transition. 
They usually exhibit a phase transition from an absorbing phase to an active state. In the \ccl{FEP} for instance, if the density is bigger than $\frac{1}{2}$ the system is in the active state with a unique invariant measure, while all the invariant measures are superpositions of atoms on absorbing states if the density is less than $\frac{1}{2}$. The critical particle density is therefore $\frac{1}{2}$. The critical behavior of CLG has been studied numerically and analytically in the physics literature \cite{L,O,BM}, in order to identify the universality classes of models displaying an active-absorbing phase transition. In particular, the \ccl{FEP} is not in the same universality class as directed percolation. Let us mention that, recently, the \ccl{FEP} appeared in its totally asymmetric version in \cite{BBCS}, where the authors compare the behavior of its first particles with what happens in the totally asymmetric exclusion process (TASEP).

Our main result is that, in the active phase, the macroscopic behavior of this microscopic dynamics, under periodic 
boundary conditions, is ruled by a non-linear diffusion equation.

\medskip

The derivation of a non-linear diffusion equation from general reversible exclusion processes was first established by Funaki et al. \cite{FHU} for gradient models, and by Varadhan and Yau \cite{VY} for non-gradient models. In both works, the irreducibility of the process on the hyperplanes with fixed number of particles is one of the essential assumptions. Results are few without this irreducibility. For instance, degenerate dynamics have been studied as microscopic models of the $1$-dimensional porous medium equation (PME) \begin{equation}\label{eq:PME}
\partial_t \rho  =\partial_u (\rho^{m-1} \partial_u \rho) \quad (m\in\bb N, m>1) 
\end{equation}
in \cite{GLT}, \cite{S} (for a generalization of the PME) and more recently in \cite{BCSS}. In these models, the problem comes mostly from the fact that regions with trivial density do not dissolve instantaneously, which calls for some adaptation and new techniques to apply the classical entropy method or relative entropy method. The obstacles caused by the lack of irreducibility that we address in this paper are of a different nature.

First, in our dynamics the state space is divided into transient states, absorbing states and ergodic states. 
Depending on the initial number of particles in the system, some of configurations will lead to the ergodic component 
(we call such states \lq\lq transient good") and the others will be absorbed to an inactive state (resp. \lq\lq transient bad"). 
Although the behavior of the dynamics started from a transient bad configuration is an interesting matter in its own right, in this article, we focus on the transient good case where the ergodic component is ultimately reached.
To apply the entropy method, the existence of transient good states is also troublesome since, although invariant measures are 
supported only on the ergodic components, the process may stay in 
the transient good states for some macroscopic time with positive probability. To guarantee this is not the case, 
we show Theorem \ref{thm:transience}. More precisely: assume that the initial density is initially larger than $\frac 12$  and look at the microscopic system of size $N$  at a macroscopic time $(\log N)^\alpha/N^2$ for some $\alpha>0$; then with high probability it has already reached the ergodic component.
This is the first main novelty of this paper.

Secondly, the Gibbs measures (grand canonical measures) of our process are not product, while they are Bernoulli product measures in \cite{GLT}. In particular, when the density is close to $\frac 12$, these measures exhibit \ccl{spatial} correlations, and adapting the entropy method for these non-product  grand canonical measures requires significant
extra technical work. We note that both papers \cite{FHU, VY} also cover cases with Gibbs measures which are not necessarily product, although for non-degenerate dynamics. Besides exclusion processes, microscopic models with non-product Gibbs measures for which the hydrodynamic limit is rigorously established are very few. 

As the nature of microscopic dynamics is different from \cite{GLT}, the macroscopic equation is also different from the PME, and we obtain the following macroscopic equation 
\[\partial_t \rho =\Delta (-\rho^{-1})= \partial_u (\rho^{-2} \partial_u \rho). \] 
The equation is in the class of fast diffusion equation (FDE), which corresponds to the case with $m <1$ in \eqref{eq:PME}. 
Though the equation has a singularity at $\rho=0$, it is not relevant for us since we always assume that the initial density is 
bigger than $\frac{1}{2}$ at any spatial point. Because of the phase transition described above, for density profiles below $\frac{1}{2}$ the system cannot be governed by this macroscopic equation, since it does not reflect the absorption phenomenon. In fact, for general density profiles, one expects that the macroscopic evolution should be the solution to a Stefan problem, with a free boundary between the active regions (with density higher than $1/2$) and the frozen ones. This property 
remains out of reach. It would require understanding the interplay between frozen regions with density lower than $\frac12$ and active regions with higher density. The Stefan problem has been derived from microscopic dynamics in a few, less degenerate contexts \cite{LV,GQ,Fu}.
We note that the same FDE was derived recently from a non-degenerate zero-range process under a proper high density limit \cite{HJV}.

We now give some remarks on other aspects of our model. As discussed in the next section, the model is reversible and gradient. Moreover, it can be \textit{formally} understood as the specific case of Examples 1 or 2 in \cite[Section 5]{FHU}, specifically the case with parameters $b(1)=-\infty$, $b(k) =0$ for $k \ge 2$ in Example 1, or $\alpha=\infty$ and $\beta=0$ in Example 2 under proper normalization. In particular, the exclusion processes in the class described in Example 1 of \cite{FHU} 
can be mapped to zero-range processes by a simple but non-linear transformation of configurations, which we discuss precisely in Section \ref{sec:Bijection}. This relation has been used in the literature 
(cf. \cite{FS, BM, J}) and plays an essential role in our estimation of the transience time. Note 
that since the corresponding zero-range process is also degenerate, we cannot apply classical results (cf. \cite[Section 5]{KL}) 
to derive its hydrodynamic limit, therefore our main theorem is not proved that way. It is also not straightforward to deduce a hydrodynamic limit for the \ccl{FEP} from the zero-range through the relation mentioned above.

\bigskip

Here follows an outline of the paper. We start in Section \ref{sec:model} by introducing notations, defining the model and stating our main results. Section \ref{sec:Bijection} is devoted to mapping our exclusion process to a zero-range process. In Section \ref{sec:ergodic} we use this transformation and we prove Theorem \ref{thm:transience}, which states that the system reaches its ergodic component in a sub-diffusive time scale. In Section \ref{sec:hydroErgo} we give the rigorous proof of the hydrodynamic limit (Theorem \ref{theo:hydro}) via a non-trivial adaptation of the entropy method, which involves two main ingredients: first, the \emph{equivalence of ensembles}, proved in Section \ref{sec:invariant} (which also introduces the different reference measures we consider), and second, the \emph{Replacement Lemma}, proved in Section \ref{app:replacement}.

\section{Model and results} \label{sec:model}
\subsection{Notations and conventions}
Let us introduce notations and conventions that we use throughout the paper.
\begin{itemize}
\item $N$ is an integer which plays the role of a scaling parameter and will go to infinity. 
\item For any finite set $\Lambda$ we denote by $|\Lambda|$ its cardinality and by $\Lambda^c$ its complement. 
\item We let $\T_N:=\Z/N\Z=\{1,\dots,N\}$ be the discrete torus of size $N$, and  $\T=[0,1)$ be the $1$-dimensional continuous torus.
\item For any $\ell \in \bb N$ we set $B_\ell:=\{-\ell,\dots,\ell\}$ the centered symmetric box of size $2\ell+1$, which can be seen as either a subset of $\T_N$ (if $2\ell+1\leqslant N$), or a subset of $\bb Z$. Similarly, we set $\Lambda_\ell:=\{1,\ldots,\ell\}$.
\item We will consider particle systems on different state spaces $\{0,1\}^E$, with $E$ either the full discrete line $\Z$, or the discrete torus $\T_N$, or a finite box in $\Z$. To avoid any confusion, the elements of $\{0,1\}^E$, which are named \emph{configurations}, will be denoted: by $\xi$ if $E=\Z$, by $\eta$ if $E=\T_N$, and by $\sigma,\varsigma$ if $E$ is a finite box. We will use the letters $\omega,\chi,\zeta,\Upsilon$ to denote configurations in yet different spaces.
\item For any $x\in\T_N$ and configuration $\eta\in\{0,1\}^{\T_N}$, we denote by $\eta(x)\in\{0,1\}$ the occupation variable at site $x$, namely: $\eta(x)=1$ if there is a particle at site $x$, and $0$ otherwise. We treat similarly the configurations in different state spaces.
 \item For any measurable function $f:\{0,1\}^{\T_N} \to \R$, and $x\in\T_N$, we denote by $\tau_x f$ the  function obtained by translation as follows: $\tau_x f(\eta):=f(\tau_x \eta)$, where $(\tau_x\eta)(y) = \eta(x+y),$ for $y\in\bb T_N$. We treat similarly the other state spaces.
\item For any $\Lambda \subset \T_N$ or $\Lambda\subset \bb Z$ and for any  probability measure $\pi$ on $\{0,1\}^{\Lambda}$, we adopt the following notations:
\begin{enumerate}[(i)]
 \item the configuration {$\eta \in \{0,1\}^{\T_N}$ (resp. $\xi \in \{0,1\}^{\Z}$)} restricted to $\Lambda$ is denoted by {$\eta_{|\Lambda}$ \big(resp. $\xi_{|\Lambda}$\big)},
\item if $f:\{0,1\}^\Lambda\to\bb R$ is a measurable function,  $\pi(f)$ denotes the expectation of $f$ w.r.t.~the measure $\pi$,
\item if $A \subset \{0,1\}^{\Lambda}$ is a subset of all possible configurations, $\pi(\eta \in A)$ equivalently means $\pi(\mathbf{1}_{\{\eta \in A\}}):=\pi(A)$.
\end{enumerate}
\item For any sequence $(u_k)_{k\in \N}$, possibly depending on other parameters than the index $k$, we will denote $\mc O_k(u_k)$ (resp. $o_k(u_k)$) an arbitrary sequence $(v_k)_{k\in \N}$ such that there exists a constant $C>0$ (resp. a vanishing sequence $(\varepsilon_k)_{k\in \N}$) -- possibly depending on the other parameters -- such that 
\[v_k\leq C u_k\quad (\mbox{resp.}\quad v_k\leq u_k\varepsilon_k) \quad \forall\; k\in \N.\]
\end{itemize}

\subsection{The microscopic dynamics}
\label{sec:Model}

 The dynamics is as follows: on the periodic domain $\T_N$, we associate with each site a random Poissonian clock. When the clock at site $x$ rings, if there is a particle sitting at this site $x$, that particle jumps to $x-1$ or $x+1$ if some local constraint is satisfied: a particle can jump to the right (resp.\@ left) only if it has a particle to its left (resp.\@ right).

 More precisely, the infinitesimal generator ruling the evolution in time of this Markov process is given by $\mathcal{L}_N$, which acts on  functions $f:\{0,1\}^{\T_N} \to \R$ as 
\begin{equation}
\label{eq:DefLN}
\mathcal{L}_Nf(\eta):=\sum_{x\in\T_N}c_{x,x+1}(\eta)\big(f(\eta^{x,x+1})-f(\eta)\big),
\end{equation}
where the constraint and the exclusion rule are encoded in the rates $c_{x,x+1}$ as 
\begin{equation}\label{e:constraint}
c_{x,x+1}(\eta)= \eta(x-1)\eta(x)(1-\eta(x+1))+\eta(x+2)\eta(x+1)(1-\eta(x)),
\end{equation}
and $\eta^{x,y}$ denotes the configuration obtained from $\eta$ by exchanging the states of sites $x,y$, namely $\eta^{x,y}(x)=\eta(y)$, $\eta^{x,y}(y)=\eta(x)$ and $\eta^{x,y}(z)=\eta(z)$ if $z\neq x,y$.
 Figure \ref{fig:jumps} below shows examples of jumps. Note that the dynamics conserves the total number of particles $\sum_{x\in\T_N}\eta(x)$.

 \begin{figure}[h]
\centering
\includegraphics[width=8cm]{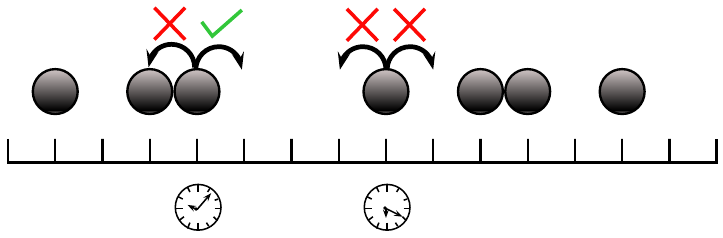}
\caption{Allowed jumps are denoted by {\color{green}$\checkmark$}. Forbidden jumps are denoted by {\color{red}X}.}
\label{fig:jumps}
\end{figure}

This \ccl{facilitated} exclusion dynamics is \emph{degenerate},  \emph{gradient}, and \emph{reversible}, as explained in the following three paragraphs.
\subsubsection{Degenerate dynamics}\label{sec:deg}

Because the jump rates can vanish, the dynamics will be referred to as \emph{degenerate.} For this reason, it is convenient to refine the classification of the configurations into transient/recurrent states as follows (this classification is fully justified in Section~\ref{sec:Bijection}). Let $\mc H_N^k$ be the hyperplane of configurations with $k$ particles, with $k \in\{0,...,N\}$, namely:
\[\mc H_N^k:= \Big\{ \eta \in \{0,1\}^{\T_N} \; :\; \sum_{x\in\T_N} \eta(x)=k\Big\}.\]
Then, \begin{enumerate}
\item if $k \leqslant \frac N 2$, some configurations in $\mc H_N^k$ are \emph{blocked}, in the sense that no particle can jump because the constraint \eqref{e:constraint} is satisfied nowhere (they are \emph{absorbing states} for the dynamics). Those are exactly the configurations in which all particles are isolated. 

The other configurations in $\mc H_N^k$ are not blocked but, starting from them, with probability one the process will arrive at a blocked configuration in a finite number of steps. We call them \emph{transient bad} configurations. See Figure \ref{fig:examples} below.
\item if $k > \frac N 2$, the process will never reach an absorbing state (except in the trivial case $k=N$). In $\mc H_N^k$, there are configurations which are in the ergodic component (the recurrent states for the process, which in this case form an irreducible component); we call them \emph{ergodic} configurations. They are the configurations in which empty sites are isolated. 

Starting from the other configurations, which are called \emph{transient good} configurations, the process enters the ergodic component after a finite number of steps a.s.  See Figure \ref{fig:examples} below.

\end{enumerate}

\begin{figure}[H]
\centering
\includegraphics[width=8cm]{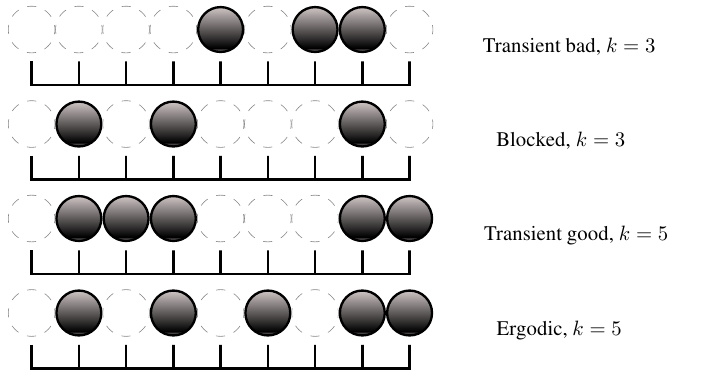}
\bigskip
\caption{Four examples of configurations which belong to the different classes, for $N=9$ sites.}
\label{fig:examples}
\end{figure}
\newcommand{\ERGt}{\widehat{\mathcal{E}}}
\newcommand{\ERG}{\mathcal{E}}

Those observations (and similar ones for configurations in finite boxes) lead to the following definitions.
\begin{definition}
We denote by $\mathcal{E}_N \subset \{0,1\}^{\T_N}$ the set of \emph{ergodic configurations on $\T_N$}, namely
\begin{equation}
\label{eq:DefERGN}
\mathcal{E}_N:=\bigg\{\eta \in \{0,1\}^{\T_N} \; : \; \forall\; x \in \T_N, \;  \big(\eta(x),\eta(x+1)\big) \neq (0,0)\text{ and }\sum_{x\in\T_N}\eta(x)> \frac N 2\bigg\},
\end{equation}
and by $\widehat{\mathcal{E}}_\Lambda$ the set of \emph{ergodic local configurations} (which are actually restrictions of ergodic configurations) on a (finite) connected set $\Lambda\subset\Z$, namely
\begin{equation}
\label{eq:DefERGlt}
\hat{\mathcal{E}}_{\Lambda}:=\bigg\{\sigma \in \{0,1\}^{\Lambda} \; : \; \forall\; (x,x+1) \in \Lambda^2, \;  \big(\sigma(x),\sigma(x+1)\big) \neq (0,0)\bigg\}.
\end{equation}
For $k>\frac N 2$ we also let
\begin{equation}
\label{def:omegaNk}
\Omega_N^k:=\mathcal{H}_N^k\cap\mathcal{E}_N,
\end{equation}
be the set of ergodic configurations on $\T_N$ which contain exactly $k$ particles.
\end{definition}

We conclude this paragraph with the following result:
 
\begin{lemma}\label{lem:irreducible}
For any $N\geq 1$ and any $k>\frac N 2$, $\Omega_N^k$ is an irreducible component for our Markov process.
\end{lemma}
\begin{proof}
Let us change the point of view and move the zeros around instead of the particles. Then it is clear that in a configuration in $\mathcal{E}_N$, a zero can jump as long as it remains at distance at least two from the others, and every allowed jump is reversible. Consequently, it is enough to show that from every configuration in $\Omega_N^k$, one can reach the configuration $\circ\bullet\circ\bullet\cdots\bullet\bullet$ where the $N-k$ zeros start alternating with particles until there are none left. Let $\eta\in\Omega_N^k$ and number its zeros from left to right (the ``left-most'' site being $1\in\T_N$). In that order, pull each of them as much to the left as possible. That way we reach either the desired configuration, or the same shifted one step to the right (in case $\eta(N)=0$): $\bullet\circ\bullet\circ\bullet\cdots\bullet$. In the second case, iterating the process brings us back to the desired configuration.
\end{proof}

\subsubsection{Gradient system}

\label{ssec:current}

 We introduce the \emph{instantaneous currents}, which are defined for any configuration $\eta$ and any site $x$ as \[j_{x,x+1}(\eta)=c_{x,x+1}(\eta)(\eta(x+1)-\eta(x)),\] and satisfy $\mc L_N(\eta(x))=j_{x-1,x}(\eta)-j_{x,x+1}(\eta)$. One can easily check that our model is \emph{gradient}, which is to say that these currents can be written as  discrete gradients\[ j_{x,x+1}(\eta) = \tau_{x} h(\eta)-\tau_{x+1}h(\eta), \qquad \text{for any } x \in \T_N,\] with the function $h$ given by \begin{equation}\label{eq:h} h(\eta):=\eta(-1)\eta(0)+\eta(0)\eta(1)-\eta(-1)\eta(0)\eta(1).\end{equation}
This function $h$ plays a fundamental role in the derivation of the hydrodynamic limit of our process.

\subsubsection{Reversible measures}

The uniform measures on $\Omega_N^k$, denoted below by $\pi_N^k$ (with $k\in \{0,...,N\}$)  are invariant for the Markov process  induced by the infinitesimal generator $\mc L_N$ and  satisfy  the \emph{detailed balance condition}, as detailed in \eqref{eq:detailed}.
When $k/N\rightarrow\rho\in (\frac12,1)$, these measures locally converge to an infinite volume grand canonical measure $\pi_\rho$ on $\{0,1\}^\Z$ {for which an explicit formula can be derived}. The measures $\pi_\rho$ are not product, and all relevant canonical and grand canonical measures will be thoroughly investigated in Section \ref{sec:invariant}.

\subsection{Main results} 

We are now ready to state the main result of this article. Fix an initial smooth profile $\rho_0:\bb T\to (\frac12,1]$, and consider the non-homogeneous product measure on $\{0,1\}^{\bb T_N}$ {fitting $\rho_0$}, defined as 
\begin{equation}
\label{eq:DefmuN}
\mu_N(\eta):=\prod_{k\in \bb T_N}\Big(\rho_0\big( \tfrac k N\big)\eta(k)+\big(1-\rho_0\big(\tfrac k N\big)\big)(1-\eta(k))\Big).
\end{equation}
Let $\{\eta_t \; : \; t\geq 0\}$ denote the Markov process driven by the \emph{accelerated} infinitesimal generator $N^2\mathcal{L}_N$ (cf. \eqref{eq:DefLN}) starting from the initial measure $\mu_N$.
Fix $T>0$ and denote by $\P_{\mu_N}$ the probability measure on the Skorokhod path space $\mc D([0,T],\{0,1\}^{\T_N})$ corresponding to this dynamics. We denote by $\E_{\mu_N}$ the corresponding expectation.  Note that, even though it is not explicit in the notation, $\P$, $\E$ and $\eta_t$ strongly depend  on $N$: through the size of the state space but also through the diffusive time scaling.

\begin{theorem}[Hydrodynamic limit]\label{theo:hydro}

For any $t \in [0,T]$, any $\delta >0$  and any continuous test function $\varphi:\T\to\R$, we have
\begin{equation}\label{eq:limith} \lim_{N\to\infty} \P_{\mu_N}\bigg[\bigg|\frac1N\sum_{x\in\T_N}\varphi\Big(\frac x N\Big)\eta_{t}(x) - \int_{\T}\varphi(u)\rho(t,u)du\bigg|>\delta\bigg]=0\end{equation}
where $\rho(t,u)$ is the unique smooth solution of the hydrodynamic equation 
\begin{equation}\label{eq:hydro}
\partial_t \rho = \Delta\Big(\frac{2\rho-1}{\rho}\Big), \qquad \rho(0,\cdot)=\rho_0(\cdot): \T \to (\tfrac12,1]. 
\end{equation}
\end{theorem}

\begin{remark} \label{rem:solution}
The fact that there is a unique smooth solution to \eqref{eq:hydro}, provided that the initial profile satisfies $\rho_0 >\frac12$, is quite standard in PDE theory, and can be found for instance in \cite[Section 3.1]{vaz}. The equation \eqref{eq:hydro} belongs to the family of \emph{quasilinear parabolic problems} of the form 
\[\partial_t \rho = \partial_u \big( A(\rho,\partial_u\rho) \big),\] where, in our case, $A(p,q):=q/p^2$. The classical theory, which gives smoothness of solutions and maximum principles, cannot be used for equation \eqref{eq:hydro}, because the latter is not uniformly parabolic\footnote{We say that $A$ is uniformly parabolic if there exist constants $0<c_1<c_2<\infty$ such that $c_1 \leqslant \frac{\partial A}{\partial q} (p,q) \leqslant c_2$, uniformly in $(p,q)$.}. However, the classical results can be used if the initial condition is \emph{non-degenerate}. For instance, if $\rho_0$ satisfies $\frac12+\varepsilon \leqslant \rho_0 \leqslant 1$, then one can apply the usual theory with such data, choosing $A(p,q)=q/p^2$ if $\frac12+\varepsilon \leqslant p \leqslant 1$, and extending the function $A$ as a linear function of $p$ and $q$ for $p$ near {$\frac12$ or $1$}, making a smooth connection at the point {$p=\frac12+\varepsilon$}. With this choice, the degeneracy has been eliminated, and therefore there exists a unique classical (smooth) solution, which satisfies the same bounds $\frac12+\varepsilon \leqslant \rho \leqslant 1$. This means that $\rho$ actually never takes values in the critical region, and therefore one gets a classical solution for \eqref{eq:hydro}.
\end{remark}
There are two main difficulties to prove Theorem \ref{theo:hydro}. The first one  lies in the fact that the dynamics is \emph{a priori} not ergodic. We now state our second main result, which will be proved in Section \ref{sec:ergodic}. It states  that the accelerated system {reaches its ergodic component} at a macroscopic time $t_N$ of order $(\log N)^\alpha/N^2$ for some $\alpha >0$. Therefore, for any macroscopic time $t>0$ and for any $N$ large enough such that 
${t_N}<t$, the configuration $\eta_t$ belongs to the ergodic component with very high probability. 
This is the one of the main novelties of this work.
\begin{theorem}[Transience time for the exclusion process with absorption]\label{thm:transience}
Letting  $\ell_N=(\log N)^8$ and ${t_N}=\ell_N^4/N^2$, we have 
\[\lim_{N\to\infty}\P_{\mu_N}\Big(\eta_{{t_N}}\notin \mathcal{E}_N\Big) = 0.\]
\end{theorem}

The second main difficulty to prove Theorem \ref{theo:hydro} comes from the nature of the invariant measures 
of the process: we investigate in Section \ref{sec:invariant} the canonical and grand canonical measures for the 
generator $\mathcal{L}_N$, which only charge the ergodic component $\mathcal{E}_N$ where all empty sites are isolated. 
These measures are therefore not product, and significant work is required to show the various properties that are needed to apply 
the classical \emph{entropy method}.

\medskip

The rest of the article is organised as follows. 
We describe in Section \ref{sec:Bijection} a correspondence between our \ccl{facilitated} exclusion dynamics 
and a particular zero-range dynamics, which will be crucial to prove Theorem \ref{thm:transience} in 
Section \ref{sec:ergodic}. We then give in Section \ref{sec:hydroErgo} a general outline of the classical strategy 
to prove Theorem \ref{theo:hydro}. 
We thoroughly study in Section \ref{sec:invariant} the invariant measures for the \ccl{facilitated} exclusion dynamics,  
and use the properties we derive to prove the Replacement Lemma in Section \ref{app:replacement}.

\section{Correspondence with the zero-range model}
\label{sec:Bijection}
Following \cite{BM}, our exclusion-type dynamics can be mapped to a \emph{zero-range model} in a way which we describe below.

First, we map the initial configuration of the exclusion dynamics (EX) to a zero-range (ZR) configuration on a reduced torus: with $\eta\in \mc H_N^k$, $k<N$, we associate a configuration $\omega\in \N^{\T_{N-k}}$ as follows. Look for the first empty site to the left of or at site $1$ and label it $1$. Moving to the right, label all empty sites in $\eta$ from $1$ to $N-k$. Then define $\omega(i)$ as the number of particles between the $i$--th and $(i+1)$--th (with $N-k+1$ identified with $1$) empty sites in $\eta$ (see Figure \ref{fig:bij}).
For any exclusion configuration $\eta$, let us denote by $\Pi(\eta)$ the corresponding zero-range configuration. Note that this mapping is not one-to-one. For instance, in the example considered in Figure~\ref{fig:bij}, shifting $\eta$ one site to the left does not change $\omega$. 

 \begin{figure}[h]
\centering
\includegraphics[width=12cm]{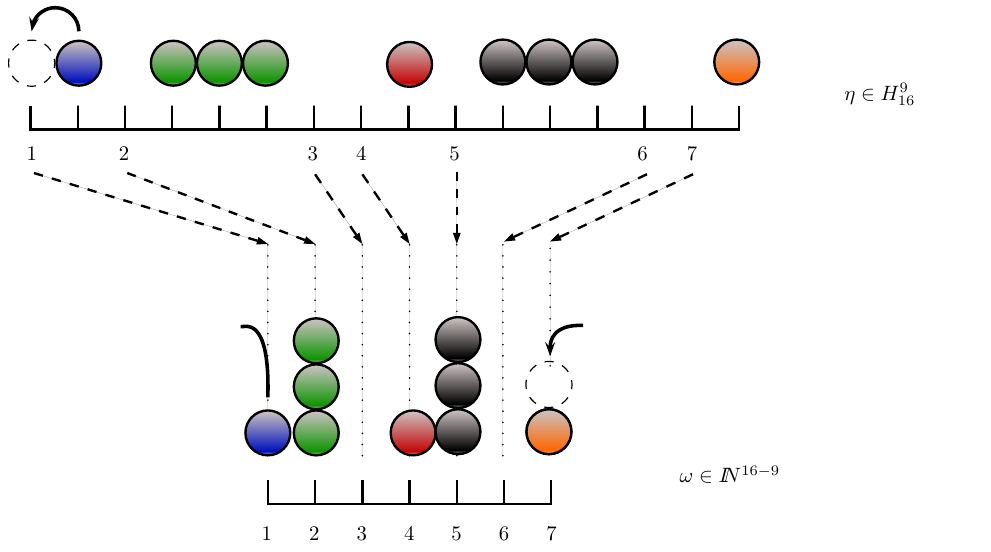}
\caption{\textsc{Top}: configuration in the exclusion dynamics, with $N=16$ sites and $k=9$ particles. \textsc{Bottom}: the associated zero-range configuration.
The arrows depict a jump in the EX process together with its effect on the ZR process; the dashed circles represent the positions of the blue particles after the jump.}
\label{fig:bij}
\end{figure} 

Now we couple the EX process $\{\eta_t \; : \; t\geq 0\}$ (generated by $N^2\mc L_N$) with a ZR process $\{\omega_t \; : \; t\geq 0\}$. 
With $\eta_0$, let us associate $\omega_0:=\Pi(\eta_0)$ in the way described above. Then whenever a particle jumps in the process $\eta_t$, a particle in the corresponding pile in $\omega_t$ jumps in the same direction (as shown in Figure~\ref{fig:bij})\footnote{Note that, in order to make this informal description rigorous, we need to keep track of the correspondence between blocks of particles in the EX process and piles in the ZR process: it is not true that $\omega_t$ is obtained from $\eta_t$ through the same mapping $\Pi$ as $\omega_0$ from $\eta_0$ (this can be checked in the example of Figure~\ref{fig:bij}). One way of understanding the process is to tag the empty site labeled $1$ (\emph{i.e.\@} when a particle jumps there, the label moves to the site that particle previously occupied). Then the mapping from $\eta_t$ to $\omega_t$ is the one described above, with the exception that the empty site with label $1$ need not be the first to the left of the origin.}. 
In particular, a jump of a particle to an empty site is allowed in the EX process iff the corresponding pile in the ZR process has at least two particles. Then, one can easily check that $\{\omega_t \; : \; t\geq 0\}$ is a Markov process with infinitesimal generator $N^2\mc L_{N-k}^{ZR}$, with $\mc L_{N-k}^{ZR}$ acting on functions $f:\N^{\T_{N-k}}\to \bb R$ as follows:
\begin{equation} \label{eq:LZR}
\mathcal{L}_{N-k}^{ZR}f(\omega):=\sum_{x\in\T_{N-k}}\sum_{z=\pm 1}\mathbf{1}_{\{\omega(x) \geqslant 2\}}\big(f(\omega^{x,x+z})-f(\omega)\big),
\end{equation}
where
\[\omega^{x,y}(z):= \begin{cases}
\omega(x)-1; & z=x, \\
\omega(y)+1; & z=y, \\
\omega(z); & z \notin\{x,y\}. \end{cases}\]
Let us classify the possible states of this ZR process as we have already done in Section~\ref{sec:deg} for the EX process. In this setting, it is fairly clear that for $\omega\in \N^{\T_{N-k}}$ with $k$ particles,
\begin{enumerate}
\item if $k\leqslant N-k$ (there are fewer particles than sites), either $\omega(x)\leqslant 1$ for all $x$ (in which case the configuration is blocked), or this situation will be reached after a finite number of jumps a.s.
\item if $k> N-k$, by the pigeonhole principle there will always be at least one particle allowed to jump. The ergodic configurations  are those where $\omega(x)\geqslant 1$ for all $x$, and the transient good the other ones, where there exists a site $x$ such that $\omega(x)=0$.
\end{enumerate}
This translates immediately into the classification we claimed for the EX configurations in Section~\ref{sec:deg}\footnote{Note that, even though the reverse mapping from ZR to EX is only defined up to the position of the empty site with label $1$, since the properties we consider are translation invariant we can safely transfer them from one setting to the other.}. In particular we will use in the proof of Theorem~\ref{thm:transience} the fact that for all $t\geq 0$, $\eta_t$ is in its ergodic component iff $\omega_t$ is in its ergodic component.

\section{Proof of Theorem \ref{thm:transience}: Transience time to reach the ergodic component}
\label{sec:ergodic}
In this section, we prove that after a time of order $(\log N)^\alpha/N^2$, with high probability, the exclusion process has reached the ergodic component. 
This result relies strongly on the correspondence between our \ccl{facilitated} exclusion process and the zero-range dynamics presented in Section \ref{sec:Bijection}. It uses arguments in the spirit of \cite{Andjel}.
In Section \ref{sec:erg1}, we state the main result of this section (Proposition \ref{prop:transience}), which is an estimate of the transience time for the zero-range process assuming that one starts from what we call a \emph{regular} configuration, and use it to prove Theorem \ref{thm:transience}. 
Section \ref{ssec:estimation} is dedicated to proving Proposition \ref{prop:transience}. 
Finally, in Section \ref{sec:probass}, we prove a technical lemma that states that, starting the exclusion process from the smooth product measure \eqref{eq:DefmuN}, the probability for the corresponding zero-range configuration to be regular is close to $1$.

\newcommand{\Qzr}{\Q_{\omega_0^{\ell}}^{\ell}}
\newcommand{\Lbar}{\overline{\Lambda}_\ell}
\subsection{Ergodic component of the zero-range dynamics}
\label{sec:erg1}
Let us fix an integer $\ell \geq 0$, and set $\Lbar{:=\Lambda_\ell\cup\{0, \ell+1\}}=\{0,\dots,\ell+1\}$. For any given integer $K>\ell$, any zero-range configuration $\omega\in\N^{\T_K}$ on $\T_K$, and any $x\in \T_K$, we define the zero-range configuration $\bar{\omega}^{\ell,x}\in \N^{\Lbar}$ as
\begin{equation}
\label{eq:DefOlxtilde}
\bar{\omega}^{\ell,x}(y):=\begin{cases}\omega(x+y)&\mbox{ if }1\leq y \leq \ell,\\
                        0 &\mbox{ if }y=0, \; \ell+1.
                       \end{cases}
\end{equation}
which is the configuration in the box of size $\ell$ to the right of $x$ in $\T_K$, with empty sites at the boundaries. We also let 
\begin{equation}\label{eq:Defnl}n_\ell(\omega):=\sum_{y=1}^{\ell}\omega(y)\end{equation}
the initial number of particles in $\Lambda_\ell$, and let 
\begin{equation}\label{eq:DefSl}Z_{\ell}(\omega):=\sum_{y=1}^{\ell}y\omega(y)\end{equation}
the cumulated distance of the particles in $\Lbar$ to the first site in the configuration $\omega$.

 Fix $\delta\in (0,1)$.  Throughout the proof, we will when convenient omit the integer part $\lfloor.\rfloor$, and for example simply write $(1+\delta)\ell$ instead of $\lfloor(1+\delta)\ell\rfloor$.
 
 For any configuration $\omega\in\N^{\T_K}$ and $x \in\T_K$, we denote by $\omega^{\ell,x}$ the configuration obtained from $\bar{\omega}^{\ell,x}$ by keeping at most the $(1+\delta)\ell$ left-most particles and destroying all the other particles. More precisely, let us denote by $1\leq x_1\leq x_2\leq\dots\leq x_{n_\ell(\bar{\omega}^{\ell,x})}\leq \ell$ the positions of the particles in $\bar{\omega}^{\ell,x}$, we then have 
 \[\bar{\omega}^{\ell,x}(y)=\sum_{k=1}^{n_\ell(\bar{\omega}^{\ell,x})}{\bf 1}_{\{x_k=y\}},\] 
 and we let 
\begin{equation} 
\label{eq:DefOlx}
\omega^{\ell,x}(y)=\sum_{k=1}^{\min\{(1+\delta)\ell \; , \; n_\ell(\bar{\omega}^{\ell,x})\}}{\bf 1}_{\{x_k=y\}}.
\end{equation}
We denote by
\begin{multline}
\label{eq:DefAl}
A_\ell:=A_\ell(\delta)=\Bigg\{\omega\in \N^{\Lbar}\; \colon \; \omega(0)=\omega(\ell+1)=0, \\
n_\ell(\omega)= (1+\delta)\ell \quad\mbox{ and }\quad Z_\ell(\omega)\leq \Big(1+\frac \delta 2\Big)(1+\delta)\frac{\ell(\ell+1)}{2}\Bigg\},
\end{multline}
the set of configurations on $\Lbar$ for which  $Z_\ell$ is not much bigger than it would be if the $(1+\delta)\ell$ particles were spread evenly across $\Lambda_\ell$.
We also set, for any $ K\geq\ell$
\begin{equation}
\label{eq:DefBlK}
B_{K,\ell}:=B_{K,\ell}(\delta)=\Big\{\omega\in \N^{\T_K}\; \colon\;\omega^{\ell,x}\in A_\ell \quad \forall\; x\in \T_K \Big\}.
 \end{equation}
We  call \emph{${\delta-}$regular configurations} the elements of $B_{K,\ell}{(\delta)}.$ In other words, since in ${\omega}^{\ell,x}$ we delete all but $(1+\delta)\ell$ particles, $B_{K,\ell}$ is the set of zero-range configurations on $\T_K$ such that in each box $\{x+1,\dots, x+\ell\}$, there are at least $(1+\delta)\ell$ particles, and those particles are not placed abnormally to the right of $\{x+1,\dots,x+\ell\}$.

\medskip

For a given initial configuration $\omega_0\in \N^{\T_K}$, we denote by $\Q^{ZR}_{K,\omega_0}$ the distribution of the process $\{ \omega_t \; : \; t\geq 0\}$ on the space of trajectories $\mc D([0,T],\N^{\T_K})$, starting from $\omega_0$ and driven by the generator $N^2\gene^{ZR}_{K}$ introduced in \eqref{eq:LZR}. 
We start by stating that, assuming that the initial configuration $\omega_0$ is ${\delta-}$regular (i.e.~is in $B_{K,\ell}{(\delta)}$), then the probability that the zero-range process starting from $\omega_0$  has not reached the ergodic component at a time of order $\ell^4$  is exponentially small.

\begin{proposition}[Transience time for the zero-range process started from $B_{K,\ell}$]
\label{prop:transience} For any integer $\ell \geq 0$, we define $T_\ell=\ell^4$. 
For any $\varepsilon<\frac12$, and any constant $\delta>0$,  there exists $\bar \ell=\bar\ell(\varepsilon,\delta)$, such that for any integer $K>\ell\geq \bar \ell$, and any $N \in \bb N$
\[\sup_{\omega_0\in B_{K,\ell}(\delta)}\Q^{ZR}_{K,\omega_0}\Big(\exists\; x\in \T_K, \; \omega_{T_\ell/N^2}(x)=0\Big)\leq K \exp\big(-\ell^{\varepsilon}\big).\]
\end{proposition} 

We now state that, while pairing the exclusion process and the zero-range process, if $\eta_0$ is started close to a smooth product measure, then with high probability the associated zero-range configuration $ \omega_0$ is in $B_{K,\ell}$ for some well chosen $K, \ell$. 
Recall that $\{\eta_t\; : \; t \geq 0\}$ denotes the Markov process driven by the infinitesimal generator $N^2\mathcal{L}_N$, whose law is denoted by $\bb P_{\mu_N}$, where $\mu_N$ is the initial measure for the process defined in \eqref{eq:DefmuN}. 
Further recall that for any exclusion configuration $\eta$, $\Pi(\eta)$ denotes the corresponding zero-range configuration. We also denote by $K(\eta)=\sum_{x\in\T_N}(1-\eta(x))$ the number of empty sites in $\eta$, i.e.~the size of the associated zero-range configuration $\Pi(\eta)$ (as explained in Section \ref{sec:Bijection}).

\begin{lemma}
\label{lemma:probass} 
Assume that the initial density profile $\rho_0$ is not identically equal to $1$. Then, letting $\ell_N=(\log N)^{8},$ there exists a constant $\delta>0$ such that
\begin{equation*}
\label{eq:EXtoZR} 
\liminf_{N\to\infty}\mu_N\Big(K(\eta)>\ell_N \quad \mathrm{and}\quad \Pi(\eta)\in B_{K(\eta),\ell_N}(\delta) \Big)=1.
\end{equation*}
\end{lemma}
We prove this last lemma in Section \ref{sec:probass}. Before proving both Proposition \ref{prop:transience} and Lemma \ref{lemma:probass}, however, we conclude the proof of Theorem \ref{thm:transience}.

\begin{proof}[Proof of Theorem \ref{thm:transience}]
This theorem is a consequence of Proposition \ref{prop:transience} and Lemma \ref{lemma:probass}. First recall (from Section~\ref{sec:Bijection}) that
\[\eta_t\notin \mathcal{E}_N \; \Leftrightarrow \;  \exists\; x\in \T_{K(\eta_0)},\; \Pi(\eta_{0})_t(x)=0.\]
Letting $\ell_N=(\log N)^8$, and $t_N=\ell_N^4/N^2$ as in Theorem \ref{thm:transience}, we now write according to Lemma \ref{lemma:probass}, 
\begin{align*}
&\bb P_{\mu_N}\Big(\eta_{t_N}\notin \mathcal{E}_N\Big)\\
=&\;\bb P_{\mu_N}\Big( \exists\; x\in \T_{ K(\eta_{0})},\; {\Pi(\eta_0)_{t_N}}(x)=0\Big)\\
=&\;\bb P_{\mu_N}\Big( \exists\; x\in \T_{ K(\eta_{0})},\; {\Pi(\eta_0)_{t_N}}(x)=0\; \Big| \;K(\eta_0)>\ell_N \; \mbox{ and }\; \Pi(\eta_0)\in B_{K(\eta_0),\ell_N}\Big)+o_N(1).
\end{align*} 
According to Proposition \ref{prop:transience}, for any $\varepsilon<\frac12$, and any $\ell_N\geq \bar \ell(\varepsilon,\delta)$, the probability above is less than
\[ \E_{\mu_N}\Big[K(\eta_0)\; \Big| \;K(\eta_0)>\ell_N \; \mbox{ and } \; \Pi(\eta_0)\in B_{K(\eta_0),\ell_N}\Big]\exp(-\ell_N^{\varepsilon }).\]
Since $K(\eta)$ is the number of empty sites in $\eta$, it is less than $N$, and we now obtain for any $\varepsilon <\frac12$, and any $N$ large enough such that $\ell_N>\bar \ell(\varepsilon,\delta)$,
\[\bb P_{\mu_N}\Big(\eta_{t_N}\notin \mathcal{E}_N\Big)\leq N\exp(-\ell_N^{\varepsilon})+o_N(1).\]
Since $\ell_N =(\log N)^8$, fixing $\varepsilon\in(\frac18,\frac12)$ concludes the proof.
\end{proof}

\subsection{Estimation of the transience time} \label{ssec:estimation}

We now prove Proposition \ref{prop:transience}.  Throughout this subsection, we assume that  $\delta>0$ is a fixed constant and $\ell$ is a fixed positive integer. Moreover, in order to avoid heavy notations, we work throughout this section  with \emph{unrescaled processes}. Namely, let $\{\tilde \omega_t \; : \; t \geq 0\}$ be the Markov process started from $\omega_0$ and generated by $\mc L_K^{ZR}$. A version of this process is given by $\tilde \omega_t = \omega_{t/N^2}$, and therefore, in order to prove Proposition \ref{prop:transience}, we are reduced to prove that: 
\[\sup_{\omega_0\in B_{K,\ell}(\delta)}\Q^{ZR}_{K,\omega_0}\Big(\exists\; x\in \T_K, \;  \tilde\omega_{T_\ell}(x)=0\Big)\leq K \exp\big(-\ell^{\varepsilon}\big).\]
For the sake of clarity, the proof will be divided into several lemmas which we prove at the end of this section. The principle of the proof is the following: fix $x\in\T_K$. By waiting a long time (of order $\ell^4$), we make sure that with high probability each particle initially present in $\{x+1,\dots,x+\ell\}$ will either 
\begin{itemize}
\item have encountered an empty site in the box $\{x+1,\dots,x+\ell\}$, and gotten stuck there forever,
\item or have left the box $\{x+1,\dots,x+\ell\}$ at some point.
\end{itemize}
At most $\ell$ particles fall into the first case, therefore, if the box initially contains at least  $(1+\delta)\ell $ particles, $\delta\ell$ particles or more will have left the box $\Lambda^{\ell,x}:=\{x+1,\dots,x+\ell\}$ at least once before $T_\ell=\ell^4$. In order for the site $x$ to remain empty up until $T_\ell$, all of these particles must  have left by the other boundary $ x+\ell$, and therefore, particles in $\Lambda^{\ell,x}$ would have performed an abnormally large number of steps to the right. The probability of the last case occurring decays faster than $\exp\pa{-\ell^{\varepsilon}}$ for any $\varepsilon<\frac12$.

\bigskip

Our purpose is now to make this argument rigorous.
Let us first write that 
\begin{equation}
\label{eq:QZRK}
\Q^{ZR}_{K,\omega_0}\pa{\exists\; x\in \T_K, \; \tilde\omega_{T_\ell}(x)=0}\leq K \max_{x\in\T_K}\Q^{ZR}_{K,\omega_0}\pa{\tilde\omega_{T_\ell}(x)=0}.
\end{equation}
We are going to prove that the probability on the right hand side is, for any $x$, for any $\varepsilon<\frac12$ and any $\ell$ large enough, smaller than $\exp(-\ell^\varepsilon)$.

\medskip

Recall that we introduced the domain $\overline{\Lambda}_\ell=\{0,\dots,\ell+1\}$, and recall definitions \eqref{eq:DefAl} and \eqref{eq:DefBlK} of $A_\ell$ and $B_{K,\ell}$, respectively. 
Let us fix $x\in \T_K$ and $\omega_0\in B_{K,\ell}$.  We introduce an auxiliary process $\{(\chi_t(y))_{y\in\Lbar}\; : \; t\geq 0\}$ (which will depend on $\ell$ and $x$).  This process is started from $\omega_0^{\ell,x}$ (defined above in \eqref{eq:DefOlx}), and therefore, since $\omega_0\in B_{K,\ell}$, we know that $\chi_0=\omega_0^{\ell,x}\in A_\ell$ contains exactly $(1+\delta)\ell$ particles located in $\Lambda_\ell$.
The process $\chi_t$ is driven by the (stuck) zero-range generator of $\tilde\omega_t$ restricted to $\overline{\Lambda}_\ell$, with the exception that particles reaching either $y=0$ or $y=\ell+1$ never leave, namely the generator $\mc L_\ell^{SZR}$ defined as
\begin{align*}\mathcal{L}_{\ell}^{SZR}f(\chi):=\mathbf{1}_{\{\chi(\ell) \geqslant 2\}}\big(f(\chi^{\ell, \ell+1})-f(\chi)\big)&+\mathbf{1}_{\{\chi(1) \geqslant 2\}}\big(f(\chi^{1,0})-f(\chi)\big) \vphantom{\bigg(}\\
&+\sum_{y,z\in\Lambda_\ell \atop |y-z|=1}\mathbf{1}_{\{\chi(y) \geqslant 2\}}\big(f(\chi^{y,z})-f(\chi)\big),\end{align*} for any $f:\bb N^{\bar\Lambda_\ell}\to\bb R$. 
Both sites $y=0$ and $y=\ell+1$ are initially empty and play the role of cemetery states where particles become stuck.

Let $\Q^{SZR}_{\chi_0}$ be the distribution of the process $\chi_t$ started from $\chi_0=\omega_0^{\ell,x}$ and driven by $\mathcal{L}_{\ell}^{SZR}$.
We denote  by 
\begin{equation}\label{eq:DefTxi}T_\chi=\inf\big\{t\geq 0 \; \colon \; \chi_t(y)\leq 1, \ \forall \; y\in \Lambda_\ell\big\}\end{equation} 
the time at which all the particles have either  left $\Lambda_\ell$ or been stuck in a hole. Then, by monotonicity of the event $\{\tilde\omega_{T_\ell}(x)=0\}$ w.r.t.~the configurations, we have for any $x\in \T_K$
\begin{equation*}
\Q_{K,\omega_0}^{ZR}\Big(\tilde\omega_{T_\ell}(x)=0\Big)\leq\Q_{\chi_0}^{SZR}\pa{\chi_{T_\ell}(0)=0}.
\end{equation*} 
Since $\chi_0=\omega_0^{\ell,x}\in A_\ell$, 
 Lemma \ref{lemma:Transxi} stated below yields, for any $\varepsilon <\frac12$ and any $\ell\geq \bar \ell(\varepsilon, \delta)$, that
\[\Q_{K,\omega_0}^{ZR}\Big(\tilde\omega_{T_\ell}(x)=0\Big)\leq \exp\pa{-\ell^{\varepsilon}}\]
as wanted. Equation \eqref{eq:QZRK} then concludes the proof.

\begin{lemma}
\label{lemma:Transxi}
For any $\varepsilon<\frac12$, there exists $\bar \ell :=\bar \ell(\varepsilon,\delta)$, such that for any $\ell\geq \bar \ell$,
\begin{equation}
\label{eq:Transxi}
\sup_{\chi_0\in A_\ell}\Q^{SZR}_{\chi_0}\Big(\chi_{T_\ell}(0)=0\Big)\leq \exp(-\ell^{\varepsilon}).\end{equation}
\end{lemma}
\begin{proof}[Proof of Lemma \ref{lemma:Transxi}] This result immediately follows from Lemmas \ref{lemma:Txi1} and \ref{lemma:ExitProba} below.
\begin{lemma}\label{lemma:Txi1}
For any $\varepsilon<1$, there exists $\bar \ell :=\bar \ell(\varepsilon,\delta)$ such that for any $ \ell\geq\bar \ell$,
\[ \sup_{\chi_0\in A_\ell}\Q^{SZR}_{\chi_0} \Big(T_\chi> T_\ell\Big)\leq \exp\pa{-\ell^{\varepsilon}}.\]
\end{lemma}
\begin{lemma}

\label{lemma:ExitProba} 
For any $\varepsilon<\frac12$, there exists $\bar \ell :=\bar \ell(\varepsilon,\delta)$ such that for any $ \ell\geq\bar \ell$,
\[\sup_{\chi_0\in A_\ell}\Q^{SZR}_{\chi_0} \pa{\chi_{T_\chi}(0)=0}\leq \exp\pa{-\ell^{\varepsilon}}.\]
\end{lemma}

In the next two paragraphs we prove Lemmas \ref{lemma:Txi1} and \ref{lemma:ExitProba}.\end{proof}

\subsubsection{Proof of Lemma \ref{lemma:Txi1}}
To prove Lemma \ref{lemma:Txi1}, we couple our zero-range process $\chi_t$ on $\Lbar$ with two auxiliary processes: $\{(\zeta_t(y))_{y\in \overline{\Lambda}_\ell}\; : \; t\geq 0\}$ and $\{(\Upsilon_t(y))_{y\in \overline{\Lambda}_\ell}\; : \; t\geq 0\}$, both
starting from the same initial state as $\chi_t$. Let us describe the dynamics of these processes. 

\medskip

The auxiliary process $\zeta_t$ is also a zero-range process, but with rate 1: no particle gets stuck inside $\Lambda_\ell$, instead each particle moves freely until it reaches sites $y=0$ or $y=\ell+1$ where it remains stuck. Its generator $\mc L_\ell^{FZR}$ is given by
\[\mathcal{L}_{\ell}^{FZR}f(\zeta):=f(\zeta^{\ell,\ell+1})-f(\zeta)+f(\zeta^{1,0})-f(\zeta)
+\sum_{y,z\in\Lambda_\ell \atop |y-z|=1}\big(f(\zeta^{y,z})-f(\zeta)\big).\]
We denote by  $\Q^{FZR}_{\zeta_0}$ the distribution of the process $\zeta_t$ started from $\zeta_0=\chi_0$ and driven by the generator $\mathcal{L}_{\ell}^{FZR}$ above and we denote by
\[T_\zeta=\inf\big\{t\geq 0\; \colon \; \zeta_t(y)=0, \ \forall\; y\in \Lambda_\ell\big\}\] 
the time at which all the particles have left $\Lambda_\ell$.

\medskip

 Recall that, since $\chi_0\in A_\ell$, it contains $(1+\delta)\ell$ particles in $\Lambda_\ell$. In the second auxiliary process $\Upsilon_t$ the particles in $\Lambda_\ell$ perform independent simple random walks with jumps occurring at rate $1/((1+\delta)\ell)$, and the particles get stuck forever once they have left $\Lambda_\ell$. Its generator $\mc L_\ell^{IRW}$ is given by
\begin{align*}\mathcal{L}_{\ell}^{IRW}f(\Upsilon):=\frac{\Upsilon(\ell)}{(1+\delta)\ell}\big(f(\Upsilon^{\ell,\ell+1})-f(\Upsilon)\big)&+\frac{\Upsilon(1)}{(1+\delta)\ell}\big(f(\Upsilon^{1,0})-f(\Upsilon)\big)\\
&+\sum_{y,z\in\Lambda_\ell \atop |y-z|=1}\frac{\Upsilon(y)}{(1+\delta)\ell}\big(f(\Upsilon^{y,z})-f(\Upsilon)\big).\end{align*}
We denote by $\Q^{IRW}_{\Upsilon_0}$ the distribution of the process $\Upsilon_t$ started from $\Upsilon_0=\chi_0$ and driven by the generator $\mathcal{L}_{\ell}^{IRW}$ above, and we denote by
\[T_\Upsilon=\inf\big\{t\geq 0\; \colon  \; \Upsilon_t(y)=0,\  \forall\; y\in \Lambda_\ell\big\}\] 
the time at which all the particles  have left $\Lambda_\ell$. 

\medskip

Given these two new processes, we state the following three lemmas, which put together prove Lemma \ref{lemma:Txi1}.
\begin{lemma}
\label{lemma:Tsigma}
For any $\varepsilon<1$, there exists $\bar \ell :=\bar \ell(\varepsilon,\delta)$, such that for any $ \ell\geq\bar \ell$,
\[ \sup_{\Upsilon_0\in A_\ell}\Q^{SZR}_{\Upsilon_0} \Big(T_\Upsilon> T_\ell\Big)\leq \exp\pa{-\ell^{\varepsilon}}.\]
\end{lemma}
\begin{lemma}\label{lemma:Tzeta}For any $\ell\geq 1$, and assuming that $\zeta_0=\Upsilon_0 \in A_\ell$,
\[\Q^{FZR}_{\zeta_0}\Big(T_\zeta> T_\ell\Big)\leq\Q^{IRW}_{\Upsilon_0}\Big(T_\Upsilon> T_\ell\Big).\]
\end{lemma}
\begin{lemma}\label{lemma:Txi}For any $\ell\geq 1$, and assuming that $\chi_0=\zeta_0 \in \N^{\bar\Lambda_\ell}$,
\[\Q^{SZR}_{\chi_0}\Big(T_\chi> T_\ell\Big)\leq\Q^{FZR}_{\zeta_0}\Big(T_\zeta> T_\ell\Big).\]
\end{lemma}
The proofs of Lemmas \ref{lemma:Tzeta} and \ref{lemma:Txi} are based on coupling arguments. Lemma \ref{lemma:Tsigma} uses classic estimates on simple random walks, which are stated in Appendix \ref{app:exit}.
Lemma \ref{lemma:Txi1} follows immediately from these three lemmas.

\subsubsection{Proof of Lemma \ref{lemma:Tsigma}}
Consider a single particle performing a symmetric random walk, jumping left and right at the same rate $1/((1+\delta)\ell)$. Thanks to the estimate \eqref{eq:ExitTimeCont} given in Appendix \ref{app:exit}, 
regardless of its initial position in $\Lambda_\ell$ the probability for this particle to remain in $\Lambda_\ell$  until time $T_\ell=\ell^4$ is bounded, for any $\ell>\bar\ell(\delta)$ by $e^{-C \ell}$ for some positive constant $C=C(\delta)$. Since $\Upsilon_0 \in A_\ell$, there are initially $(1+\delta)\ell$ such particles located in $\Lambda_\ell$.
Moreover, since in $\Upsilon$ these particles move independently, by definition of $T_{\Upsilon}$, it is straightforward to obtain as wanted, for any $\Upsilon_0\in A_\ell$,
\[\Q^{IRW}_{\Upsilon_0}\Big(T_\Upsilon > T_\ell\Big)\leq 1-\big(1-e^{-C \ell}\big)^{(1+\delta)\ell} \sim (1+\delta)\ell e^{-{\color{blue}C\ell}}.\]
Given $\varepsilon <1$, choosing $\bar\ell$ large enough such that $\big(1-\big(1-e^{-C \bar\ell}\big)^{(1+\delta)\bar\ell}\big)e^{\bar\ell^\varepsilon}\leq 1$ proves the lemma.

\subsubsection{Proof of Lemma \ref{lemma:Tzeta}}
Fix an initial configuration $\zeta_0=\Upsilon_0 \in A_\ell$. The difference between the processes $\Upsilon$ and $\zeta$ is that in the second one the jump rate of a given particle depends on the current configuration. However,  since initially the total number of particles in $\Lambda_\ell$ is $(1+\delta)\ell$,  there are certainly never more than $(1+\delta)\ell$ particles on a given site of $\Lambda_\ell$, and therefore the particles in $\zeta$ are faster than those in $\Upsilon$. For completeness we give a full construction of a coupling between the two processes.

The particles' trajectories will be the same for both $\zeta$ and $\Upsilon$, but the order and times of the jumps will change. 
Let us start from a configuration $\zeta_0=\Upsilon_0$ with $(1+\delta)\ell$ particles in $\Lambda_\ell$. 
For any $1\leq p\leq (1+\delta)\ell$, denote by $x_p$ the initial position of the $p$-th particle in $\zeta_0$. 
We define the random trajectory $X_p(0)=x_p,\dots, X_p(k_p)\in\{0,\ell+1\}$ of the $p$-th particle until it leaves $\Lambda_\ell$, where $k_p$ is the random number of jumps necessary to do so. In other words, we define $X_p(0)=x_p$ and for each $k\geq 1 $, we let 
\[X_p(k)=\begin{cases}
X_p(k-1)+1&\mbox{ with prob. }\frac12,\\
X_p(k-1)-1&\mbox{ with prob. }\frac12,\\
\end{cases}\]
and we define $k_p=\inf\big\{k\geq 1\; : \; X_p(k)\in \{0,\ell+1\}\big\}$.
We finally denote by \[ \bar X=\big(X_p(k)\big)_{1\leq p\leq (1+\delta)\ell, \; 0\leq k\leq k_p }{ \quad \mbox{ and } \quad\bar k =(k_p)_{1\leq p\leq (1+\delta)\ell} }\] the family of these trajectories and their lengths. 

\medskip

Given ${\bar k, \;}\bar X$, let us now define the order and times of the jumps in each process. Let
\[K=\sum_{p= 1}^{(1+\delta)\ell}k_p\] be
the total number of jumps to perform in the configuration in order to reach a frozen state, which is the same for $\zeta$ and $\Upsilon$ since they start from the same configuration. We are going to define update times  $t^\Upsilon_0=0<t^\Upsilon_1<\dots<t^\Upsilon_K=T_\Upsilon$ and $t^\zeta_0=0<t^\zeta_1<\dots<t^\zeta_K=T_\zeta$ for both processes, and build $\zeta$ and $\Upsilon$ in such a way that for any $i\leq K$, 
\[t_i^\zeta\leq t_i^\Upsilon.\]
At time $t^\Upsilon_i$, (resp. $t^\zeta_i$) we are going to choose a particle $p^\Upsilon_i$ (resp. $p^\zeta_i$)  which we will move in $\Upsilon$ (resp. $\zeta$) according to its predetermined trajectory $X_{p^\Upsilon_i}$ (resp. $X_{p^\zeta_i}$). 

\medskip

Assume that both $\zeta$ and $\Upsilon$ are defined respectively  until time $t_i^\zeta$ and $t_i^\Upsilon$. We denote by $ n_i^\Upsilon \in \{1,\dots, (1+\delta)\ell\}$ the number of \emph{particles} still in $\Lambda_\ell$ at time $t_i^\Upsilon$ and by $n_i^\zeta \in \{1,\dots,\ell\}$ the number of \emph{sites} in $\Lambda_\ell$ still occupied by at least one particle in $\zeta$ at time $t_i^\zeta$. We then sample two exponential times $\tau_{i+1}^\zeta$ and $\tau_{i+1}^\Upsilon$, with respective parameters $n_i^\zeta$ and $n_i^\Upsilon/((1+\delta)\ell)$. For any $i< K$, in $\Lambda_\ell$ there is at least one occupied site in $\zeta$, therefore $n_i^\zeta\geq 1$. Furthermore, there are in $\Lambda_\ell$ at most  $(1+\delta)\ell$  particles in $\Upsilon$, so that $n_i^\Upsilon/((1+\delta)\ell)\leq 1$. It is therefore possible to couple  $\tau_{i+1}^\zeta$ and $\tau_{i+1}^\Upsilon$ in such a way that 
\begin{equation}\label{eq:steptime}\tau_{i+1}^\zeta\leq \tau_{i+1}^\Upsilon.\end{equation}
We then let 
\[t^\Upsilon_{i+1}=t^\Upsilon_{i}+\tau_{i+1}^\Upsilon\leq t^\zeta_{i+1}=t^\zeta_{i}+\tau_{i+1}^\zeta.\]
We leave $\Upsilon$ (resp. $\zeta$) unchanged in the time segment $(t^\Upsilon_{i},t^\Upsilon_{i+1})$ $\big($resp. $(t^\zeta_{i},t^\zeta_{i+1})\big)$. In order to close the construction, we now need to build $\zeta_{t^\zeta_{i+1}}$ and $\Upsilon_{t^\Upsilon_{i+1}}$. To do so, we need to choose the particle to move, and then the chosen particle will move according to the pre-chosen trajectories $\bar X$, as we explain{ed above}.
\medskip

To choose the index of the particle to move in $\zeta$, we choose one site uniformly among the occupied sites in $\zeta_{t^\zeta_i}$, and then choose the particle to update uniformly among the particles present at this site. We then update the position of the particles according to the trajectories stored in $\bar X$. To choose the particle to move in $\Upsilon$, we simply choose the particle uniformly among the particles in $\Lambda_\ell$ in $\Upsilon_{t_i^\Upsilon}$, and update the position of this particle according to the fixed trajectories $\bar X$.   After the final time $t_K^\zeta=T_\zeta$ (resp. $t_K^{\Upsilon}=T_\Upsilon$), the process $ \zeta$  (resp. $\Upsilon$) remains frozen.

It is not difficult to see that under this construction, $\zeta$ and $\Upsilon$ are respectively distributed according to $\Q^{FZR}_{\zeta_0}$ and $\Q^{IRW}_{\Upsilon_0}$. Furthermore, by construction, the update times for each step satisfy \eqref{eq:steptime}, therefore
\[t_K^\zeta=T_\zeta\leq t_K^\Upsilon=T_\Upsilon,\]
thus concluding the proof of Lemma \ref{lemma:Tzeta}.

\subsubsection{Proof of Lemma \ref{lemma:Txi}}
Let us couple $\chi$ and $\zeta$ through the following graphical construction. Attach to every site in $\Lambda_\ell$ independently a Poisson process of parameter $1$ and think of it as a sequence of clock rings signalling possible update times for the process. Decorate the clock rings with independent $\mathrm{Ber}(\frac12)$ variables (also independent from the Poisson processes). In a given configuration of the process $\chi$, we call \emph{excess particle} any particle which is not alone on a site. The process $\chi$ (resp.\ $\zeta$) can be constructed by saying that when a clock rings at site $x$, one excess particle at $x$ (resp.\ one particle present at $x$) jumps right or left depending on the corresponding Bernoulli variable. Now in order to couple the two processes in a way to have $T_\chi\leq T_\zeta$, we turn $\zeta$ into a \emph{two-color process}. In the initial configuration, in $\zeta$, one particle per site is coloured red (corresponding to the particle{s} that will remain stuck in $\chi$) and the others blue. We build the processes in such a way that the following property is preserved: 
\begin{equation}\label{coupling-property}
\textit{The number of excess particles in $\chi$ is equal to the number of blue particles in $\zeta$.}
\end{equation} When a clock rings in a configuration satisfying this condition, three cases arise.
\begin{itemize}
\item If the ring occurs on a site with no particle in either process, nothing happens.
\item If the ring occurs on a site with only red particles in $\zeta$, one of them jumps according to the Bernoulli variable and nothing happens in $\chi$ (where there is either no particle or a single -- and therefore blocked -- one).
\item If the ring occurs on a site with excess particles in $\chi$, one of them jumps according to the Bernoulli variable. In $\zeta$, one of the blue particles makes the same jump. If the excess particle in $\chi$ becomes stuck after the jump, then color the corresponding particle in $\zeta$ red.
\end{itemize}
It is immediate to check that property \eqref{coupling-property} is preserved through any of these transitions. Also, $\chi$ has the desired law, as does the color-blind version of $\zeta$. Moreover, with this description, $T_\chi$ is the time at which all particles in $\zeta$ have either turned red or left the box, which clearly happens before they all leave the box. This ends the proof of Lemma \ref{lemma:Txi}.

Now it only remains to prove Lemma \ref{lemma:ExitProba}.

\subsubsection{Proof of Lemma \ref{lemma:ExitProba}}
For any $ \chi_0\in A_\ell$, we are going to estimate
\[\Q^{SZR}_{\chi_0} \pa{\chi_{T_\chi}(0)=0}.\]
The core idea to estimate the probability above is that in order for each particle that did not get stuck to leave $\Lambda_\ell$ at $x=\ell+1$ rather than at $x=0$, an unusual number of particle jumps to the right must have occurred, which happens with probability decaying exponentially fast in $\ell^{1/2}$.

For any realization of $\chi$, we denote by $L=L(\chi)$ the random number of jumps occurring before $T_{\chi}$, i.e. the number of particle jumps needed to reach a frozen state:
\[L(\chi)=\#\big\{{0\leq}t\leq T_{\chi}\; : \;  \chi_{t^-}\neq \chi_t\big\}.\]
As before, since $\chi_0\in A_\ell$, there are $  (1+\delta)\ell $ particles initially in $\Lambda_\ell$, and  each one of those typically requires $\mc O_{\ell}(\ell^2)$ jumps to either leave or get stuck. For any $\varepsilon<\frac12$, we claim that there must exist $\bar \ell(\varepsilon, \delta)$ such that for any $\ell\geq \bar \ell$
\begin{equation}
\label{eq:MajLxi}
\Q^{SZR}_{\chi_0}\Big(L(\chi)> \ell^{3+1/2}\Big)\leq \exp\pa{-\ell^{\varepsilon}}.
\end{equation}
To prove \eqref{eq:MajLxi}, we use the same coupling as in the proof of Lemma \ref{lemma:Txi}, where the process $\chi$ is compared with the process $\zeta$ where particles do not get stuck in empty sites. Then, denote by
\[L(\zeta)=\#\big\{t\leq T_{\zeta} \; \colon \; \zeta_{t^-}\neq \zeta_t\big\},\]
the random number of particle jumps which are necessary for each particle in $\zeta$ to leave $\Lambda_\ell$. Using the coupling described in the proof of Lemma \ref{lemma:Txi}, we obtain that \[L(\chi)\leq L(\zeta), \]
and thus we need to estimate $\Q^{FZR}_{\zeta_0}\big(L(\zeta)> \ell^{3+1/2}\big)$. 

Since in $\zeta$, the number $L_i$ of jumps which are necessary for the $i$-th particle to leave $\Lambda_\ell$ does not depend on the other trajectories (even though the time to leave does), we can now write (using the fact that  $\zeta_0=\chi_0$ is in $ A_\ell$)
\begin{align*}\Q^{FZR}_{\zeta_0}\bigg(L(\zeta)=\sum_{i=1}^{(1+\delta)\ell}L_i> \ell^{3+1/2}\bigg)&\leq\Q^{FZR}_{\zeta_0}\bigg( \exists\; i\in\big\{1,...,(1+\delta)\ell\big\}, \; L_i> \frac{\ell^{{2}+1/2}}{(1+\delta)}\bigg)\\
&\leq (1+\delta)\ell{\underset{i\in\{1,...,(1+\delta)\ell\}}{\sup}} \; \Q^{FZR}_{\zeta_0}\bigg(L_{i}> \frac{\ell^{2+1/2}}{(1+\delta)}\bigg).
\end{align*}
As a consequence of equation \eqref{eq:ExitTimeDisc} given in Lemma \ref{lemma:ExitTimeRW}, regardless of the starting point of each particle, there exists a universal constant $C>0$ such that for any $\ell$, and any $1\leq i\leq (1+\delta)\ell$,
\[\Q^{FZR}_{\zeta_0}\bigg(L_i> \frac{\ell^{2+1/2}}{(1+\delta)}\bigg)\leq \exp(-C\ell ^{1/2}),\]
thus proving \eqref{eq:MajLxi} for any $\varepsilon< \frac12$. Given \eqref{eq:MajLxi}, in order to prove Lemma \ref{lemma:ExitProba}, we now only  need to consider the realizations of $\chi$ for which $L(\chi)\leq  \ell^{3+1/2}$, and therefore  Lemma \ref{lemma:ExitProba} is a  consequence of Lemma \ref{lemma:Exitxi} below.

\begin{lemma}
\label{lemma:Exitxi}
For any $\varepsilon<\frac12$, there exists $\bar\ell(\varepsilon,\delta)$, such that for any $\ell\geq \bar \ell$ 
\begin{equation}
\label{eq:Exitxi2}
\sup_{\chi_0\in A_\ell} \Q^{SZR}_{\chi_0}\pa{\chi_{T_{\chi}}(0)=0\;\Big| \;L(\chi)\leq \ell^{3+1/2} }\leq \exp(-\ell^\varepsilon).\end{equation}
\end{lemma}

 \subsubsection{Proof of Lemma \ref{lemma:Exitxi}}
To prove this last result, we introduce a discrete time random walker $(X_k)_{k\leq L(\chi)}$ evolving on $\Z$,  starting from $X_0=0$ and jumping each time a particle of $\chi$ present in $\Lambda_\ell$ jumps, in the same direction as the particle. In other words, if the $k$--th particle jump in $\chi$ is to the left (resp.\ to the right), then $X_{k}=X_{k-1}-1$ (resp.\ $X_{k}=X_{k-1}+1$).

Let us define  
\[{Z(t,\chi)}:=\sum_{y=0}^{\ell+1} y\; \chi_t(y),\]
and let $0<\tau_1<...<\tau_{L(\chi)}=T_{\chi}$ be the successive jump times in $\chi$. 
Then, an elementary computation yields that 

\begin{equation}\label{eq:StXk}{Z(\tau_k,\chi)-Z(0,\chi)}=X_k,\end{equation}
because each jump to the right (resp.\ left) increases (resp.\ decreases) $Z(t,\chi)$ by $1$.  All the jumps occurring in $\chi$ being symmetric, the law of $(X_k)_{k\leq L(\chi)}$ is that of a symmetric random walker in discrete time. To avoid burdensome notations, we now assume that the random walk $X$ is defined for any time $k\in \N$ by completing the random walk after $L(\chi)$ independently from $\chi$.

Assuming that no particle in $\chi$ reaches $x=0$, we know that after $T_\chi$, each particle must have either gotten stuck in $\Lambda_\ell$, or reached the site $x=\ell+1$. Note that in order to minimize ${Z(T_{\chi},\chi)}$ in this case, there must be exactly one particle per site in $\Lambda_\ell$, and all the remaining particles must be at site $\ell+1$. In that case, since the initial number of particles in $\Lbar$ is $(1+\delta)\ell$,
\[ {Z(T_{\chi},\chi)}= \pa{2\delta+1}\frac{\ell(\ell+1)}{2}.\]
We can therefore write  
\begin{equation}
\label{eq:ZTxi}
\chi_{T_{\chi}}(0)=0\quad \Longrightarrow \quad {Z(T_{\chi},\chi)}\geq \pa{2\delta+1}\frac{\ell(\ell+1)}{2}.
\end{equation}
Recall that $\chi_0\in A_\ell$ and therefore by definition
\begin{equation}
\label{eq:Z0}
{Z(0,\chi)} \leq \Big(1+\frac{\delta}{2}\Big)(1+\delta)\frac{\ell(\ell+1)}{2}.
\end{equation}
Thanks to \eqref{eq:ZTxi} and \eqref{eq:Z0}, we obtain that 
\begin{equation*}
\big(\chi_0\in A_\ell \;\; \mbox{and}\;\; \chi_{T_{\chi}}(0)=0\big)\quad \Longrightarrow \quad {Z(T_{\chi},\chi)}-{Z(0,\chi)}\geq (\delta-\delta^2)\frac{\ell(\ell+1)}{4}\geq (\delta-\delta^2)\frac{\ell^2}{4}.
\end{equation*}
Recall that we assumed $0<\delta<1$. Letting $C=C(\delta):=(\delta-\delta^2)/4>0$, we can finally write for any $\chi_0\in A_\ell$ 
\begin{multline}
\label{eq:xi0X}
\Q_{\chi_0}^{SZR}\pa{\chi_{T_{\chi}}(0)=0\;\Big| \;L(\chi)\leq \ell^{3+1/2}}\\
\leq \Q_{\chi_0}^{SZR}\pa{{Z(T_{\chi},\chi)}-{Z(0,\chi)}\geq C\ell^2\;\Big| \;L(\chi)\leq \ell^{3+1/2}}. 
\end{multline}
Now, recall that by equation \eqref{eq:StXk}, ${Z(T_{\chi},\chi)}-{Z(0,\chi)}=X_{L(\chi)}$, and let $\P^0$ denote the distribution of a discrete time symmetric random walk on $\Z$ started from $0$. We get
\begin{align*}
\Q_{\chi_0}^{SZR}\bigg({Z(T_{\chi},\chi)}-{Z(0,\chi)}\geq C\ell^2\;\Big| \;L(\chi)\leq \ell^{3+1/2}\bigg)
\leq \;\P^0\bigg(\sup_{0\leq k\leq \ell^{3+1/2}}X_{k}\geq C\ell^2\bigg).\end{align*}
Since $X_k$ is a discrete time symmetric random walk, we use Doob's inequality on the positive sub-martingale $e^{\lambda X_k}$ for $\lambda=C\ell^{-\frac32}$. This yields 
\[\P^0\bigg(\sup_{0\leq k\leq \ell^{3+1/2}}X_{k}\geq C \ell^2\bigg)\leq \frac{\exp\cro{\ell^{3+1/2}\pa{\frac{C^2\ell^{-3}}{2}+\mc O_\ell\pa{\ell^{-6}}}}}{\exp\pa{C^2\sqrt{\ell}}}.\] 
Therefore, for any $\varepsilon < \frac12$, there exists an integer $\bar \ell(\varepsilon, \delta)$ such that for any $\ell\geq \bar \ell$
\[\P^0\bigg(\sup_{0\leq k\leq \ell^{3+1/2}}X_{k}\geq C \ell^2\bigg)\leq \exp\pa{-\ell^{\varepsilon}},\]
which concludes the proof of Lemma \ref{lemma:Exitxi}.

\subsection{Probability of the regular configurations}
\label{sec:probass}
In this section, we prove Lemma \ref{lemma:probass}, namely:
\[\lim_{N\to\infty}\mu_N\Big(K(\eta)>\ell_N \quad \mbox{and}\quad \Pi(\eta)\in B_{K(\eta),\ell_N} \Big)=1,\]
where we have already defined $\ell_N=(\log N)^8$. We start by proving that 
\begin{equation}
\label{eq:Kpetit}
\lim_{N\to\infty}\mu_N\big(K(\eta)>\ell_N \big)=1. 
\end{equation}
Recall that $K(\eta)$ is the number of empty sites in the exclusion configuration $\eta$, and that  $\rho_0(\cdot)\in (\frac12,1]$, and $\rho_0$ is smooth. We assume that $\rho_0$ is not identically equal to $1$, in which case the model is trivial. Then, since $\rho_0$ is smooth, there exist $\alpha>0$, $u\in \T$, and $\varepsilon>0$ such that uniformly in $[u, u+\varepsilon]$, we have $ \rho_0<1-\alpha$.  Since under $\mu_N$, the state of each site is chosen independently, we can use a standard large deviation estimate: for any $c<\alpha\varepsilon$, there exists a constant $C(c)$ such that
\[ \mu_N\big(K(\eta)\leq c N\big)\leq e^{-CN},\]
which proves \eqref{eq:Kpetit}.

\medskip

To close the proof of Lemma \ref{lemma:probass}, we now prove that 
\begin{equation}
\label{eq:PBKl}
\lim_{N\to\infty}\mu_N\Big(  {K(\eta)>\ell_N  \; \mbox{ and }\;\Pi(\eta)\notin B_{K(\eta),\ell_N}} \Big)=0.
\end{equation}
Because of the random nature of the number $K(\eta)$ of sites of $\Pi(\eta)$, and of the dependency between $K(\eta)$ and the number of particles in $\Pi(\eta)$, this proof, whose principle is rather simple, becomes quite technical. We first treat the case where the profile $\rho_0$ is homogeneous with density $\bar\rho\in (\frac12,1)$, and then prove that the case where $\rho_0$ is not constant follows, thus concluding the proof of Lemma \ref{lemma:probass}. Since we will need to differentiate the cases depending on the initial profile, we introduce explicitly $\rho_0$ in our notation, and denote by $\mu_{N,\rho_0}$ the initial product measure fitting the macroscopic profile $\rho_0$ (defined in \eqref{eq:DefmuN}).

\begin{lemma}
\label{lem:nlY}
For any $\bar\rho\in (\frac12,\frac23)$, there exists $\delta=\delta(\bar\rho)\in (0,1)$ such that
\begin{equation*}
\lim_{N\to\infty}\mu_{N,\bar \rho}\Big({K(\eta)>\ell_N  \; \mbox{ and }\;\Pi(\eta)\notin B_{K(\eta),\ell_N}}  \Big)=0
\end{equation*}
where $\mu_{N,\bar \rho}$ denotes the measure associated with the constant profile $\rho_0\equiv \bar\rho$.

\end{lemma}

\begin{corollary}
 \label{cor:IncrProfile}
Fix a non-constant smooth function $\rho_0:\T \to (\frac12,1]$, then, there exists $\delta \in (0,1)$ such that
\begin{equation*}
\lim_{N\to\infty}\mu_{N,\rho_0}\Big( {K(\eta)>\ell_N  \; \mbox{ and }\;\Pi(\eta)\notin B_{K(\eta),\ell_N}} \Big)=0.
\end{equation*}
\end{corollary}

We assume for now Lemma \ref{lem:nlY} in order to prove Corollary \ref{cor:IncrProfile}. 
\begin{proof}[Proof of Corollary \ref{cor:IncrProfile}]
Fix $\bar{\rho}\in\Big(\frac12,\frac23\wedge\underset{u\in\T}{\min}\ \rho_0(u)\Big)$.

We claim that in order to prove Corollary \ref{cor:IncrProfile}, it is enough to build on the same probability space an increasing family of configurations $\eta^\lambda$ on $\T_N$, for $\lambda\in [0,1]$, with joint distribution  $\nu_N$, such that $\eta^0$ has law $\mu_{N,\bar \rho}$, $\eta^1$ has law $\mu_{N,\rho_0}$ and
\begin{equation}
\label{eq:Monotony}
\nu_N\pa{{ \Pi(\eta^1)\notin B_{K(\eta^1),\ell_N}\left|\Pi(\eta^{0})\in B_{K(\eta^{0}),\ell_N},\  K(\eta^{0})>\ell_N \mbox{ and } K(\eta^1)>\ell_N\right.}}=0.
\end{equation}
Indeed, since $\eta^0$ is distributed according to $\mu_{N,\bar\rho}$,  according to Lemma \ref{lem:nlY} and equation \eqref{eq:Kpetit} (which holds for any density profile not identically equal to $1$), the probability of the conditioning event goes to $1$ as $N\to\infty$, therefore \eqref{eq:Monotony} proves Corollary \ref{cor:IncrProfile}.

To build the family $\eta^\lambda$, and prove \eqref{eq:Monotony}, consider a family of $N$ i.i.d.~uniform $\mc U([0,1])$ variables $U_1,\dots,U_N$. Define for any $\lambda\in[0,1]$ the function 
\[\rho^\lambda=\bar \rho+\lambda(\rho_0-\bar\rho),\]
which is an interpolation between $\bar\rho$ and $\rho_0$. We then write for any $x\in \T_N$
\[\eta^\lambda(x)={\bf 1}_{\big\{U_x\leq \rho^\lambda(x/N)\big\}}.\]
Then, as wanted, $\eta^0\sim \mu_{N,\bar \rho}$ and $\eta^1\sim\mu_{N,\rho_0}$. We now prove that \eqref{eq:Monotony} holds.
 
 \medskip
As $\lambda$ goes from $0$ to $1$, particles are added in $\eta^0$ until it becomes $\eta^1$. Fix an exclusion configuration $\eta$ with an empty site at $x$, and consider the configuration $\eta+\mathbf{1}_x$ where a particle has been added at site $x$ (namely, $\mathbf{1}_x$ is the configuration such that $\mathbf{1}_x(x)=1$ and $\mathbf{1}_x(y)=0$ for any $y\neq x$). Since $x$ is empty in $\eta$, it is associated with a site $\widetilde{x}$ in $\Pi(\eta)$. 
Then, the associated zero-range configuration $\Pi(\eta+\mathbf{1}_x)$ is obtained by deleting the site $\widetilde{x}$ in $\Pi(\eta)$, and putting all the particles that were present at $\widetilde{x}$ in $\Pi(\eta)$ on the previous site $\widetilde{x}-1$. 
Another particle is then added on the site $\widetilde{x}-1$ to represent the empty site that is now a particle. We  follow this construction, throughout the transition from $\eta^0$ to $\eta^1$. Assuming that $\Pi(\eta^{0})\in B_{K(\eta^{0}),\ell_N}$, then in $\Pi(\eta_0)$, all the boxes of size $\ell_N$ already contain $(1+\delta)\ell_N$ particles. 
Then, as $\lambda$ goes from $0$ to $1$, we add particles to $\eta^0$, and therefore delete sites in $\Pi(\eta^0)$, while placing all the particles on the deleted site on the left neighbor of the deleted site. 
In each box $\{x+1,\dots ,x+\ell_N \}$ in the smaller torus, the cumulated distance between the $(1+\delta)\ell_N$ left-most particles to $x$ therefore diminishes as we go from $\Pi(\eta^0)$ to $\Pi(\eta^1)$. This construction proves equation \eqref{eq:Monotony}, and concludes the proof of the corollary.
\end{proof}

\begin{proof}[Proof of Lemma \ref{lem:nlY}]
In this proof, we write $\mu_N$ for $\mu_{N,\bar{\rho}}$. One technical issue is that even though $\eta$ is now translation invariant under $\mu_N$, $\Pi(\eta)$ is not. To solve this, we map the exclusion configuration $\eta$ to a zero-range configuration on $\Lambda_\ell$ starting from an arbitrary point in $\T_N$. 

Fix $x\in\T_N$ and a positive integer $\ell$ such that $(2+\delta)\ell<N$. As represented in Figure \ref{fig:RegConfig}, let $\eta_{(x,\ell)}$ be the exclusion configuration on $\{0,1\}^{\T_N}$ obtained from $\eta$ by keeping at most the first $(1+\delta)\ell$ particles in $\{x+1,x+2,\ldots,x\}$ (in cyclic order). We define $\Pi_{(x,\ell)}(\eta)\in\N^{{\bar\Lambda_\ell}}$ as the zero-range configuration of size $\ell$ in which
\begin{itemize}
\item $\Pi_{(x,\ell)}(\eta)(0)=\Pi_{(x,\ell)}(\eta)(\ell+1)=0$,
\item $\Pi_{(x,\ell)}(\eta)(y)$ is the number of particles between the $y$--th and $(y+1)$--th zero in $\eta_{(x,\ell)}$, starting as if\footnote{Due to the condition $(2+\delta)\ell<N$, $\Pi_{(x,\ell)}(\eta)$ actually does not depend on $\eta(x)$.} the first zero were at $x$.
\end{itemize} 
Note that, if $K(\eta)>\ell$, for any $ y\in\T_{K(\eta)}$, there exists $x \in \T_N$ such that
\begin{equation}\label{eq:thereexistsx} \Pi(\eta)^{\ell,y}=\Pi_{(x,\ell)}(\eta) \quad \in\bb N^{\Lambda_\ell}
\end{equation}
(recall that for any zero-range configuration $\omega$, the configuration $\omega^{\ell,y}$ has been defined in \eqref{eq:DefOlx}).  
With \eqref{eq:thereexistsx}, we can write, for $N$ large enough to have $(2+\delta)\ell_N<N$
\begin{align*}
\mu_N\Big(K(\eta)>\ell_N\text{ and }\Pi(\eta)\notin &B_{K(\eta),\ell_N}\Big)\\
&\leq \mu_N\left(K(\eta)>\ell_N\text{ and }\exists\; y\in\T_{K(\eta)},\ \Pi(\eta)^{\ell_N,y}\notin{A}_{\ell_N}\right)\\
&\leq \mu_N\left(K(\eta){>}\ell_N\text{ and }\exists\; x\in\T_N,\ \Pi_{(x,\ell_N)}(\eta)\notin{A}_{\ell_N}\right)\\
&\leq N\max_{x\in\T_N}\mu_N\left(\Pi_{(x,\ell_N)}(\eta)\notin{A}_{\ell_N}\right)\\
&=N\mu_N\left(\Pi_{(1,\ell_N)}(\eta)\notin{A}_{\ell_N}\right),
\end{align*}
because $\mu_N$ is translation invariant.

We now argue that we can estimate the last probability via standard properties of i.i.d.\@ sequences of geometric random variables. 

\begin{figure}[h]
\centering
\includegraphics[width=10cm]{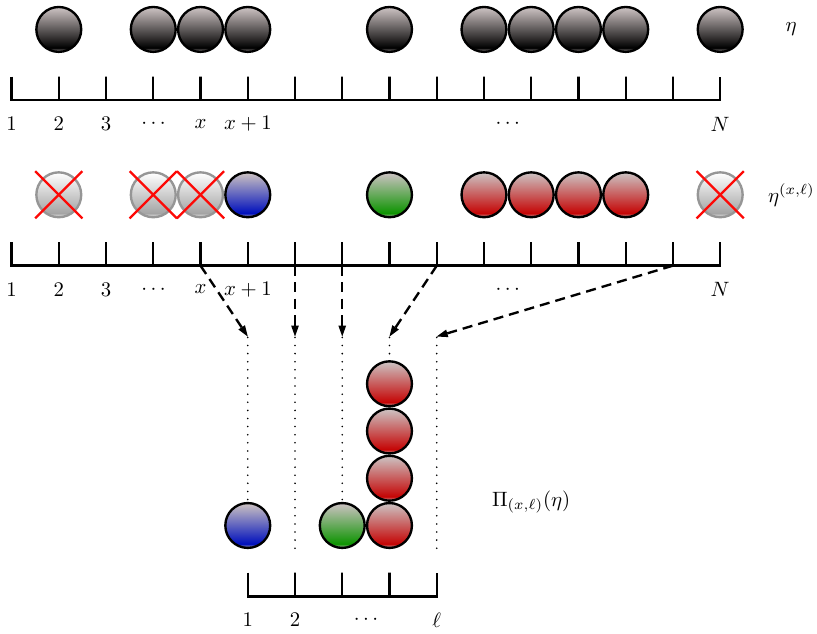}
\caption{Representation of $\eta$, $\eta_{(x,\ell)}$, $\Pi_{(x,\ell)}(\eta)$, for $\ell=5$ and $(1+\delta)\ell=6$. Starting from $\eta$, we obtain $\eta_{(x,\ell)}$ by keeping the first $(1+\delta)\ell$ particles to the right of $x$ ($x$ excluded). We then obtain $\Pi_{(x,\ell)}(\eta)$ by taking the zero-range configuration associated with the first $\ell$ empty sites in $\eta_{(x,\ell)}$, where $x$ is considered to be the first empty site.}
\label{fig:RegConfig}
\end{figure}

\medskip

Let us uncover one by one the variables given by the states of $\eta$ on sites $1,2,\ldots, (2+\delta)\ell_N$. Since $(2+\delta)\ell_N<N$, they are i.i.d. Then $\{n_{\ell_N}\left(\Pi_{(1,\ell_N)}(\eta)\right)<(1+\delta)\ell_N\}$ is the event that in this process we see $\ell_N+1$ zeros before seeing $(1+\delta)\ell_N$ particles and therefore
\begin{align}
\mu_N\Big(n_{\ell_N}\left(\Pi_{(1,\ell_N)}(\eta)\right)<(1+\delta)\ell_N\Big)
&=P\left(\sum_{k=1}^{\ell_N}G_k<(1+\delta)\ell_N\right),\label{eq:geom1}
\end{align}
where the $G_k$ are i.i.d.\@ geometric random variables of parameter $\bar{\rho}$, namely for $k\in\N$, $P(G_1=k)=(1-\bar{\rho})\bar{\rho}^{k}$.

For the same reason, $\Pi_{(1,\ell_N)}$ is dominated by a configuration given by i.i.d.\@ $\mathrm{Geom}(\bar\rho)$ variables on $\Lambda_{\ell_N}$. Consequently,
\begin{multline}
\mu_N{\bigg[}Z_{\ell_N}\left(\Pi_{(1,\ell_N)}(\eta)\right)> \Big(1+\frac\delta 2\Big)(1+\delta)\frac{\ell_N(\ell_N+1)}{2}{\bigg]}\\
\leq P\cro{\sum_{k=1}^{\ell_N}kG_k> \Big(1+\frac\delta 2\Big)(1+\delta)\frac{\ell_N(\ell_N+1)}{2}}.\label{eq:geom2}
\end{multline}
Now it only remains to show that the right hand side in \eqref{eq:geom1} and \eqref{eq:geom2} is $o_N(1/N)$ for a well-chosen $\delta(\bar{\rho})$. We choose $\delta\in (0,1)$ such that
\begin{equation}
1+\delta<\frac{\bar{\rho}}{1-\bar{\rho}}<(1+\delta)\Big(1+\frac\delta 2\Big).
\end{equation}
This is possible because $\bar{\rho}\in(\frac12,\frac23)$. By standard arguments, there exists $C=C(\delta,\bar{\rho})>0$ such that for $N$ large enough
\begin{equation}
P\left(\sum_{k=1}^{\ell_N}G_k<(1+\delta)\ell_N\right)\leq e^{-C\ell_N}.
\end{equation}
To estimate \eqref{eq:geom2}, we crop the variables $G_k$ so that they remain bounded, and then apply the Azuma-Hoeffding inequality to $\sum_{k=1}^{\ell_N}kG_k$.
For any integer $p_N$ to be determined later on, and for any $k=1,\ldots,\ell_N$, we define
\[\widehat{G_k}=G_k\wedge p_N.\]
For any given $k=1,\ldots,\ell_N$, we can write 
\[P\big(\widehat{G_k}\neq G_k\big)=\sum_{j=p_N+1}^{\infty}(1-\bar{\rho})\bar{\rho}^j\leq \bar{\rho}^{p_N+1}.\]
Thus, letting  $p_N=\ell_N^{1/4}$, we obtain by a simple union bound
\begin{equation}\label{eq:cropxi}P\big(\exists\ k\leq\ell_N,\ \widehat{G_k}\neq G_k\big)\leq \ell_N\bar{\rho}^{p_N+1}=\mc O_N\Big(\ell_Ne^{-C \ell_N^{1/4}}\Big).\end{equation}
Let us define the martingale\[M_K=\sum_{k=1}^K k \big(\widehat{G_k}-E[\widehat{G_k}]\big),\quad \text{ for } K=1,\ldots,\ell_N,\]
with $M_0=0$ and $M_K=M_{\ell_N}$ for $K\geq \ell_N$. Note that $E[\widehat{G_k}]=\frac{\bar{\rho}}{1-\bar{\rho}}+\mc O_N(p_N\bar{\rho}^{p_N})$.
For any $\alpha>\frac{\bar{\rho}}{1-\bar{\rho}}$,
\begin{multline*}
P\left(\sum_{k=1}^{\ell_N}kG_k>\alpha\frac{\ell_N(\ell_N+1)}{2}\right)\leq P\big(\exists\ k\leq\ell_N,\  \widehat{G_k}\neq G_k\big)\\
 +P\left(M_{\ell_N}> \big(\alpha-E[\widehat{G_1}]\big)\frac{\ell_N(\ell_N+1)}{2}\right).
\end{multline*}
Recall that we defined $\ell_N=(\log N)^8$. To prove that the right hand side of \eqref{eq:geom2} is $o_N(1/N)$, it is therefore sufficient to prove that for $\beta>0$ there exists $C(\beta)>0$ such that
\begin{equation*}
P \bigg(M_{\ell_N}> \frac{\beta \ell_N(\ell_N+1)}{2}\bigg)=\mc O_N\big(e^{-C\ell_N^{1/2}}\big).\end{equation*}
By definition, the martingale $(M_K)_K$ has increments bounded a.s. as follows: 
\[\mid M_{K}-M_{K-1}\mid \leq K p_N.\]
Therefore we can apply the Azuma-Hoeffding inequality to obtain that \[P\bigg(M_{\ell_N}>\frac{\beta \ell_N(\ell_N+1)}{2}\bigg)\leq \exp \pa{-\frac{\beta^2 \ell_N^2(\ell_N+1)^2}{8\sum_{K=1}^{\ell_N}K^2p_N^2}}=\mc O_N\pa{ e^{-\frac{3}{8}\beta^2\sqrt{\ell_N}}},\]
which concludes the proof.
\end{proof}

 \section{Proof of Theorem \ref{theo:hydro}: Hydrodynamic limits} 
\label{sec:hydroErgo}
Our proof of the hydrodynamic limit relies on the classical \emph{entropy method} (which is  explained in details for instance in \cite[Chapter 5]{KL}), but requires adaptations to 
solve the ergodicity issue, and to account for the non-product invariant measures. 
In Section \ref{sec:sketch}, we describe the main steps of the proof in the form of two distinct results: \begin{itemize} \item the first one states that after reaching the ergodic component, the macroscopic density profile of the system is very close to the 
initial density profile $\rho_0$. This result is proved in Section  \ref{sec:macrotransience};
\item the second result states that starting from the ergodic component, the hydrodynamic limit holds, and is proved in Section \ref{sec:hydroErgo1}. 
The latter requires the classical Replacement Lemma, which, because of the shape of our canonical and grand canonical measures, requires significant work. 
The proof of the Replacement Lemma is therefore postponed to Section \ref{app:replacement}, after we have proved the relevant properties for the canonical measures in Section \ref{sec:invariant}.
 \end{itemize}

\subsection{Sketch of the proof}\label{sec:sketch}

Recall that we defined in Theorem \ref{thm:transience} the time
$t_N:=\ccl{T_{N}}/N^2$ with $\ccl{T_{N}=(\log N)^{32}}$. 
We denote by \[\tau_N =\inf \big\{t>0\; : \; \eta_t\in \mathcal{E}_N\big\},\]
the hitting time of the ergodic component for the process rescaled diffusively in time. Recall that $\mathcal{E}_N$ has been defined in \eqref{eq:DefERGN}. 
We proved in Theorem \ref{thm:transience} that 
\[\limsup_{N\to\infty}\bb P_{\mu_N}(\tau_N>t_N)=0.\] 
Let $\tm$ denote the distribution of $\eta_{t_N}$ at time $t_N$, conditioned to have already reached the ergodic component, namely
\begin{equation}
\label{eq:Deftm}
\tm(\eta')=\bb P_{\mu_N}\Big(\eta_{t_N}=\eta'\mid \tau_N\leq t_N\Big), \qquad \eta'\in\mc E_N.
\end{equation}

\begin{proof}[Proof of Theorem \ref{theo:hydro}]
	Fix a time $t>t_N$ and $\delta>0$. We bound from above the probability in \eqref{eq:limith} by
	\begin{equation*} 
	\P_{\mu_N}\bigg(\left.\bigg|\frac1N\sum_{x\in\T_N}\varphi\Big(\frac x N\Big)\eta_{t}(x) - \int_{\T}\varphi(u)\rho(t,u)du\bigg|>\delta \;\; \right|\;\; \tau_N\leq t_N\bigg)+\bb P_{\mu_N}\pa{\tau_N>t_N}.
	\end{equation*}
	 As already pointed out, the second term above vanishes as a consequence of  Theorem \ref{thm:transience}. The first term can be rewritten for any $t>t_N$, by the Markov property, as
	\begin{align}
	\P_{\tm}\bigg(\bigg|\frac1N\sum_{x\in\T_N}\varphi&\Big(\frac x N\Big)\eta_{t-t_N}(x) - \int_{\T}\varphi(u)\rho({t},u)du\bigg|>\delta\bigg)\label{eq:t1}\\
	\leq\;&\; \P_{\tm}\bigg(\bigg|\frac1N\sum_{x\in\T_N}\varphi\Big(\frac x N\Big)\big(\eta_{t}(x)-\eta_{t-t_N}(x)\big)\bigg|>\frac \delta 3\bigg) \label{eq:t2}\\
	&+ \P_{\tm}\bigg(\bigg|\frac1N\sum_{x\in\T_N}\varphi\Big(\frac x N\Big)\eta_{t}(x) - \int_{\T}\varphi(u)\rho(t,u)du\bigg|>\frac \delta 3\bigg) \label{eq:t3}.
	\end{align}
	The first term in the right hand side \eqref{eq:t2} vanishes as a consequence of the Markov inequality and Lemma \ref{lem:MacroTransience} stated below, by taking for $\nu_N$ the distribution under $\bb P_{\tm}$ of $\eta_{t-t_N}$. The second term \eqref{eq:t3} vanishes as well according to Proposition \ref{lem:HydroErgo} stated below.

Fix $t>0$, then for some $N$ large enough, we have $t>t_N$, so that the bound \eqref{eq:t1} indeed holds. This proves the vanishing limit \eqref{eq:limith} for any positive time $t$. Since \ccl{it is} trivially true for $t=0$ as well, this proves Theorem \ref{theo:hydro}.
\end{proof}

\begin{lemma}[Empirical measure at the transience time]
\label{lem:MacroTransience}
For any  smooth test function $\varphi:\T \to \R$, and any initial measure $\nu_N$ on  $\{0,1\}^{\T_N}$, we have 
\[ \lim_{N\to\infty} \bb E_{\nu_N}\bigg[\bigg|\frac1N\sum_{x\in\T_N}\varphi\Big(\frac x N\Big)\big(\eta_{t_N}(x)-\eta_0(x)\big)\bigg|\bigg]=0.\]
\end{lemma}

\begin{proposition}[Hydrodynamic limit starting from the ergodic component]
	\label{lem:HydroErgo}
	For any $t \in [0,T]$, any $\delta >0$  and any smooth test function $\varphi:\T\to\R$, we have
	\begin{equation}\label{eq:limithErgo} \lim_{N\to\infty} \P_{\tm}\bigg(\bigg|\frac1N\sum_{x\in\T_N}\varphi\Big(\frac x N\Big)\eta_{t}(x) - \int_{\T}\varphi(u)\rho(t,u)du\bigg|>\delta\bigg)=0\end{equation}
	where $\rho(t,u)$ is the unique solution of the hydrodynamic equation \eqref{eq:hydro}.
\end{proposition}

Before proving these two results, we give the following corollary of Lemma~\ref{lem:MacroTransience}, which states that at time $t_N$, the macroscopic density profile has not changed. This will be useful to prove Proposition \ref{lem:HydroErgo} in Section \ref{sec:hydroErgo1}.
\begin{corollary}[Macroscopic profile after the transience time]
	\label{cor:MacroTransience}
	For any smooth function $\varphi:\T\to \R$, 
	\[ \lim_{N\to\infty} \bb E_{\mu_N}\bigg[\bigg|\frac1N\sum_{x\in\T_N}\varphi\Big(\frac x N\Big)\eta_{t_N}(x)-\int_{\T}\varphi(u)\rho_0(u)du\bigg|\bigg]=0.\]
\end{corollary}
This corollary is an immediate consequence of the definition of $\mu_N$, together with the law of large numbers, and Lemma \ref{lem:MacroTransience}.

\medskip

In the last two paragraphs of this section we prove Lemma \ref{lem:MacroTransience} and Proposition \ref{lem:HydroErgo}.

\subsection{Proof of Lemma \ref{lem:MacroTransience}}
\label{sec:macrotransience}
	This proof is pretty elementary, since in a time $t_N=o_N(N)$, no particle in the system has time to travel on a microscopic distance much larger than ${T_N=(\log N)^{32}}$, and in particular no particle travels on a macroscopic (of order $N$) distance. Let us denote by $k=k(\eta)=\sum_{x\in \bb T_N} \eta_0(x)$ the random number of particles initially in $\bb T_N$. Let us also denote by $x_1(t),\ldots,x_k(t)\in \T_N$ the \emph{microscopic} positions of the particles at time $t$. 
	The main result we need to prove is the following: for some well chosen $\beta>0$
	\begin{equation}
	\label{eq:TraveledDistance}
	\bb P_{\nu_N}\pa{\sup_{0\leq i\leq k}\abs{x_i(0)-x_i(t_N)}\geq T_N^{\beta}}\underset{N\to \infty}{\longrightarrow}0.
	\end{equation} 
Before briefly proving this bound, let us show that Lemma \ref{lem:MacroTransience} follows from it. By reorganizing the terms, we can rewrite the expectation in Lemma \ref{lem:MacroTransience} as
\begin{align*}
\bb  E_{\nu_N}&\bigg[\bigg|\frac1N\sum_{i=1}^k\Big(\varphi\Big(\frac{x_i(t_N)}{N}\Big)-\varphi\Big(\frac{x_i(0)}{N}\Big)\Big)\bigg|\bigg]\\
\leq&\sup_{\substack{x,y\in \bb T_N\\ \abss{x-y}\leq {T}_N^{\beta}}}\abs{\varphi\pa{\frac xN}-\varphi\pa{\frac yN}}+2\|\varphi\|_{\infty}{\bb P_{\nu_N}}\bigg(\sup_{1\leq i\leq k}\abs{x_i(t_N)-x_i(0)}\geq {T}_N^{\beta}\bigg) 
.\end{align*} 
In the bound above, the first term vanishes because $\varphi$ is a smooth function and because ${T}^{\beta}_N/N=o_N(1)$, whereas the second term vanishes according to  \eqref{eq:TraveledDistance}.

We therefore only need to prove \eqref{eq:TraveledDistance} to conclude the proof of Lemma \ref{lem:MacroTransience}. Since there are at most $N$ particles in the system, the probability in \eqref{eq:TraveledDistance} is less than 
\[N\sum_{j\leq N}\bb P_{\nu_N}(k(\eta)=j)\sup_{i\leq j}\bb P_{\nu_N}\pa{|x_i(0)-x_i(t_N)|\geq {T}_N^{\beta}\mid k(\eta)=j}.\]
Since the particles jump to a neighboring site  at rate at most $N^2$, the probability above is less than the probability that a random walker jumping to its neighbors at rate $N^2$ jumped $X\geq{T}_N^{\beta}$ times in a time $t_N$, so that we only have to prove that 
\begin{equation}
\label{eq:probaLD}
N \bb P\big(X\geq {T}_N^{\beta}\big) \xrightarrow[N\to\infty]{}0.
\end{equation}
This number of jumps $X$ is distributed according to a Poisson law with mean ${T}_N$, so that by a standard argument of large deviations, one obtains that from some constant $C$,
\[\bb P\big(X\geq {T}_N^{\beta}\big)\leq e^{-C{T}_N^{\beta-1}}. \]
We now only need to choose $\beta$ {large enough} to obtain \eqref{eq:probaLD} and then conclude the proof of Lemma \ref{lem:MacroTransience}. \qed

\subsection{Proof of Proposition \ref{lem:HydroErgo}}
\label{sec:hydroErgo1}
In this section, we prove Proposition \ref{lem:HydroErgo} by adapting the classical \emph{entropy method} (cf. \cite[Chapter 5]{KL}). Let us consider the process $\{(\eta_t(x))_{x\in\T_N} \; : \; t \in [0,T]\}$ started from the ergodic component under the measure  $\tm$ defined in \eqref{eq:Deftm}.
For $s\in[0,T]$, we denote by $<\cdot, m_s^N>$ the integral on $\{0,1\}^{\bb T_N}$ w.r.t.~the empirical measure at time $s$ denoted by $m^N_s$, namely
\[<\psi,m^N_s>=\frac{1}{N}\sum_{x\in \bb T_N}\psi\Big(\frac x N\Big)\eta_s(x), \qquad \text{with } \psi:\T \to \R.\]
Fix a smooth function \[\begin{array}{ccccc}\varphi:&  [0,T]\times \bb T &\to &\bb R\\
&(s,u)& \mapsto &\varphi_s(u).\end{array}\]
 The evolution of the process between time $0$ and time $T$ is described by the following Dynkin's formula:
\begin{align*} <\varphi_T,m_T^N>=<\varphi_0,m_0^N>& +\int_0^T ds\; \frac{1}{N}\sum_{x\in \bb T _N}\varphi_s\Big(\frac x N\Big) N^2\mathcal{L}_N\eta_s(x) \\
& +\int_0^T ds <\partial_s\varphi_s,m_s^N> + M_T^N(\varphi),\end{align*} 
where $(M_T^N(\varphi))$ is a $\P_{\tm}$--martingale which vanishes\footnote{To see this, one can easily compute its quadratic variation, and shows that its expectation vanishes. See \cite[Chapter 4, Page 76]{KL} for instance.} in $\mathbf{L}^2$, as $N\to \infty$. Furthermore, by summation by parts, we can also rewrite (thanks to the gradient property of our model, see Section \ref{ssec:current}) the first integrand above as
\begin{equation}\label{eq:ipp}\frac{1}{N}\sum_{x\in \bb T _N}\varphi_s\Big(\frac x N\Big)N^2 \mathcal{L}_N\eta(x)=\frac{1}{N}\sum_{x\in \bb T _N}\Delta_x^N\varphi_s\;\tau_x h(\eta),\end{equation}
where $h$ was defined in \eqref{eq:h} and $\Delta^N\varphi$ denotes the discrete approximation of the continuous Laplacian defined as usual as
\[\Delta_x^N\varphi:=N^2\big(\varphi\big(\tfrac {x+1} N\big)+\varphi\big(\tfrac {x-1} N\big)-2\varphi\big(\tfrac x N\big)\big).\]
The next step consists in replacing in the right hand side of \eqref{eq:ipp}  the local function $\tau_x h$ by its 
{expectation with respect to the invariant measure parametrized by the empirical density in a mesoscopic box around $x$}. 
For that purpose, let us define the empirical density of the configuration $\eta$ in the box $B_\ell$ (of size $2\ell+1$ symmetric around $0$) as
\[\rho^\ell(0)(\eta):=\frac{1}{2\ell+1}\sum_{y \in B_\ell} \eta(y),\]
and around $x$ as $\rho^\ell(x)(\eta):=\rho^\ell(0)(\tau_x\eta)$.
In what follows we will also use a simplified notation when looking at {the} macroscopic time $s$: \[ \rho_s^{\ell}(x):=\rho^\ell(x)(\eta_s).\]
 The replacement in \eqref{eq:ipp} is then done thanks to the \emph{Replacement Lemma}, which in our case is stated as follows:
\begin{lemma}[Replacement Lemma]
\label{lem:RL}
Recall that $h$ has been defined in \eqref{eq:h}, and denote $F(\rho)=\max\{0,(2\rho-1)/\rho\}$. We have for any $T>0$
\begin{equation}\label{eq:repl} \limsup_{\varepsilon \to 0}\limsup_{N\to\infty} \bb E_{\tm}\bigg[\int_{0}^Tds\frac{1}{N}\sum_{x\in \bb T_N}\tau_x\big|V_{\varepsilon N}(h,\eta_{s}) \big|\; \bigg]=0,\end{equation}
where, for any $k \in\N$, 
\begin{equation}
\label{eq:DefV}
V_{k}(h,\eta):=\frac{1}{2{k}+1}\sum_{{y\in B_{k}}}\tau_y h(\eta) - F\big(\rho^{k}(0)(\eta)\big).
\end{equation}
\end{lemma}
The standard proof of this lemma (see for example \cite[Lemma 5.7 p.97 and Chapter 5, Section 3]{KL}) has to be adapted very carefully to our case. In particular, the grand canonical measures are different in the periodic and non-periodic setting, so that the former cannot be easily restricted to finite volume. Since the proof is quite long and technical, we postpone it to Section \ref{app:replacement}.

Thanks to Lemma \ref{lem:RL}, since $\varphi$ is smooth,
and since the martingale $(M_t^N(\varphi))$ vanishes in probability, we can write for any $\delta>0$ 
\begin{multline}\label{eq:estim1}
\limsup_{\varepsilon \to 0}\limsup_{N\to\infty}\bb P_{\tm}\bigg(\bigg|<\varphi_T,m_T^N>-<\varphi_0,m_0^N>\\
-\int_0^T ds\Big[ <\partial_s\varphi_s,m_s^N>+\frac{1}{N}\sum_{x\in \bb T _N}\Delta_x^N\varphi_s \; F\big(\rho_s^{\varepsilon N}(x)\big)\Big]\bigg|>\delta\bigg) =0.
\end{multline}
The quantity inside absolute values is in fact a function of the empirical measure process $\{m_s^N\;; \;s\in[0,T]\}$: for any measure $m$ on $\bb T$ and any $u \in \T$, we set 
\[\widetilde{F}^{{\varepsilon }}_u(m)=F\big(<\delta_u^{{\varepsilon }}, m>\big),\]
where $\delta_u^{{\varepsilon }}$ is defined as:
\[\delta_u^{{\varepsilon }}(v)=\frac{1}{2{\varepsilon }}{\bf 1}_{\{u-{\varepsilon }\leq v \leq u+{\varepsilon }\}}, \qquad \text{for any } v \in \T.\] 
Note that, when $u=x/N$, since $<\delta_u^{{\varepsilon }}, m^N_s>=\rho^{\varepsilon N}_{s}(x)+o_N(1)$, from \eqref{eq:idh} we have (uniformly in $s\in[0,T]$):
\[\widetilde{F}^{{\varepsilon }}_{x/N}({m^N_{s}})=\max\left\{0,\frac{2\rho^{\varepsilon N}_{s}(x)-1}{\rho^{\varepsilon N}_{s}(x)}\right\}+o_N(1).\]
Let us denote by $Q_N$ the distribution of the empirical measure process $\{m_s^N\;; \;s\in[0,T]\}$ when $\{\eta_s\; ; \; s\in[0,T]\}$ is distributed according to $\P_{\tm}$.
We can now rewrite the previous identity \eqref{eq:estim1} as
\begin{multline}\label{eq:estim2}
\limsup_{\varepsilon \to 0}\limsup_{N\to\infty}Q_N\bigg(\bigg|<\varphi_T,\ccl{m_T}>-<\varphi_0,\ccl{m_0}>\\
-\int_0^T ds\Big[ <\partial_s\varphi_s,\ccl{m_s}>+\frac{1}{N}\sum_{x\in \bb T _N}\Delta_x^N\varphi_s \; \widetilde{F}^{\varepsilon}_{x/N}(\ccl{m_s})\Big] \bigg|>\delta\bigg)=0.
\end{multline}
It is straightforward, using the classical tools, to prove that the sequence $(Q_N)$ is relatively compact for the weak topology so that for any of its limit points $Q^*$, we can write 
\begin{multline*}
\limsup_{\varepsilon \to 0}Q^*\bigg(\bigg|<\varphi_T,m_T>-<\varphi_0,m_0>\\
-\int_0^T ds\Big[ <\partial_s\varphi_s,m_s>+\int_{\bb T} du\Delta\varphi_s(u) \widetilde{F}^{\varepsilon}_u(m_s)\Big] \bigg|>\delta\bigg)=0.
\end{multline*}
Since at most one particle is allowed at any site, one can prove that any limit point $Q^*$ is concentrated on measures $m$ which admit at any time $s$ a density $\rho_s$ w.r.t.~the Lebesgue measure on $\bb T$, so that we can let $\varepsilon$ go to $0$ in the last equality, to obtain thanks to Corollary \ref{cor:MacroTransience} the following: for any $\delta>0$,
\begin{multline*}
Q^*\bigg(\bigg|\int_{\bb T}\varphi_T(u)\rho_T(u)du-\int_{\bb T}\varphi_0(u)\rho_0(u)du\\
-\int_0^T ds\int_{\bb T}du\bigg[\partial_s\varphi_s(u)\rho_s(u)+\Delta\varphi_s(u) \max\left\{0,\frac{2\rho_s(u)-1}{\rho_s(u)}\right\}\bigg] \bigg|>\delta\bigg)=0.
\end{multline*}
Since a maximum principle holds for the hydrodynamic equation \eqref{eq:hydro}, letting $\delta$ go to $0$ proves that $Q^*$ is concentrated on trajectories whose density w.r.t.~the Lebesgue measure is a weak solution of  \eqref{eq:hydro}. The uniqueness of such solutions concludes the proof.

\section{Invariant measures and equivalence of ensembles} \label{sec:invariant}

\subsection{Canonical measures for the exclusion dynamics in periodic setting}
\label{sec:MCM}

\begin{definition}
For $k>N/2$, $\pi_N^k$ is the uniform measure on $\Omega_N^k$ (recall \eqref{def:omegaNk}).
\end{definition}
We know from \cite{GKR} that $\pi_N^k$ is the invariant measure for the exclusion process on $\T_N$ with $k$ particles. 
It  is translation invariant and satisfies the  \emph{detailed balance condition}:  for any $x\in\T_N$ and $\eta\in\Omega_N^k$,
\begin{align}
\pi_N^k(\eta)\eta(x-1)&\eta(x)(1-\eta(x+1)) \notag \\ & =\pi_N^k(\eta)\eta(x-1)\eta(x)(1-\eta(x+1))\eta(x+2)\notag \\
&=\pi_N^k(\eta^{x,x+1})\eta^{x,x+1}(x-1)(1-\eta^{x,x+1}(x))\eta^{x,x+1}(x+1)\eta(x+2)\notag\\
&=\pi_N^k(\eta^{x,x+1})(1-\eta^{x,x+1}(x))\eta^{x,x+1}(x+1)\ccl{\eta^{x,x+1}}(x+2),\label{eq:detailed}
\end{align}
where we used the fact that the holes are isolated on $\Omega_N^k$ in the first and last equalities. Let us now characterize the marginals of $\pi^k_N$. 
\begin{lemma}\label{lem:ConfCount}
Let $ k \in \{1,\ldots, N-1\}$ and $m=N-k$. We have the identity
\begin{equation*}
\label{eq:cardOkn}
\big|\Omega_N^k\big|=\binom{k}{m}+\binom{k-1}{m-1}=\frac{N}{k}\binom{k}{N-k}
\end{equation*}
Furthermore, fix $\ell \leqslant N$ and a local ergodic configuration $\sigma\in \ERGt_{{\Lambda_\ell}}$. We define \begin{itemize}
\item [\textbullet] $p:=\sum_{x\in{ \Lambda_\ell}}\sigma(x) \in \{0,\ldots,{\ell}\}$ its number of particles,
\smallskip
\item [\textbullet] $z:={\ell}-p$ its number of holes,
\end{itemize}
and, assume $p\leqslant k$ and $z\leqslant m$. Then we have
\begin{equation*}
\label{pikN}
\Big|\Big\lbrace \eta\in\Omega_N^k\ : \ \eta_{|{\Lambda_\ell}}=\sigma\Big\rbrace\Big| = \binom{k-p+{\sigma(1)}+\sigma(\ell)-1}{m-z}. 
\end{equation*}
where by convention $\binom{n}{n'}=0$ if $n'>n$.
\end{lemma}

\begin{proof}[Proof of Lemma \ref{lem:ConfCount}]
The proof is contained in \cite{GKR}. We reproduce it here for completeness.

Split the configurations in $\Omega_N^k$ into those with a particle at $1$ and those with a hole at $1$. The number of the first kind is given by the number of ways one can insert the $m$ holes between two particles. Since there are $k$ particles which are situated on the torus, there are $k$ possible positions for the holes and we must choose $m$ of those. When there is a hole at $0$, it remains to insert $m-1$ holes in the remaining $k-1$ positions. 

To compute the cardinality of the second set, we consider that after fixing the configuration  $\sigma$ on ${\Lambda_\ell}$, it remains to insert $m-z$ holes. The number of available positions is given by the remaining number of particles $k-p$ minus one if ${\sigma(1)}=\sigma(\ell)=0$ (in that case the constraint imposes particles at $\ell+1$ and $0=N$), and similarly for the other cases.
\end{proof}

\subsection{Grand canonical measures}
\label{sec:GCM}

\subsubsection{In infinite volume} 

\begin{definition}\label{def:pirho}
Define the grand canonical measure $\pi_\rho$ as the translation invariant measure on $\{0,1\}^{\Z}$ such that
\begin{itemize}
\item  for $\rho>\frac12$, $\ell\geq 1$, 
\begin{equation}\label{pirho2}
\pi_\rho\Big( \xi_{|\Lambda_\ell}={\sigma}\Big)={\bf{1}}_{\{\sigma\in \ERGt_{\Lambda_\ell}\}}
(1-\rho)\left(\frac{1-\rho}{\rho}\right)^{\ell-1-p}\left(\frac{2\rho-1}{\rho}\right)^{2p-\ell+1-{\sigma(1)-\sigma(\ell)}},
\end{equation}
where $p=\sum_{x\in \Lambda_\ell}\sigma(x)\in \{0,\dots,\ell\}$ is the number of particles in $\sigma$.
\medskip
\item  for $\rho\leq \frac12$,
\begin{equation}
\pi_\rho=\frac12\delta_{\circ\bullet}+\frac12\delta_{\bullet\circ},
\end{equation}
where $\circ\bullet$ (resp. $\bullet\circ$) is the configuration in which there is a particle at $x$ iff $x$ is odd (resp. even), and $\delta_\eta$ is the usual Dirac measure concentrated on the configuration $\eta$\footnote{\ccl{We define in this fashion the grand canonical measure for $\rho< \frac12$, in order to be able to write the equivalence of ensembles stated in Proposition \ref{prop:ensembles} and Corollary \ref{cor:ensembles} ( i.e. to be able to write approximations of the canonical measures $\pi^k_N $ by the grand canonical measures $\pi_{k/N}$ even in the pathological cases where $N$ is odd and $k=(N-1)/2<N/2$).}}.
\end{itemize}
\end{definition}

Note that, by periodizing the configurations, we can see the measures $\pi_N^k$ as measures on $\lbrace 0,1\rbrace^\Z$. In that case, since the state space is compact, the sequence $(\pi_N^k)_{N}$ is tight. If $k/N\rightarrow\rho>\frac12$, \ccl{using Lemma~\ref{lem:ConfCount}, one can check that there is a unique limit point, and that this limit point is} the measure $\pi_\rho$ defined in Definition \ref{def:pirho}.

The formula \eqref{pirho2} can be rewritten as
\begin{equation}
\label{pirho}
\pi_\rho\Big(\xi_{|{\Lambda_\ell}}={\sigma}\Big)={\bf{1}}_{\{\sigma\in \ERGt_{\Lambda_\ell}\}}\kappa\;\alpha^p\;\beta^{\ell}\;\gamma^{{\sigma(1)+\sigma(\ell)}},
\end{equation}
where the four constants $\kappa$, $\alpha$, $\beta$, $\gamma$ all depend on $\rho\in (\frac12,1)$, and are given by 
\begin{equation}
\kappa(\rho):=2\rho-1, \quad \alpha(\rho):=\frac{(2\rho-1)^2}{\rho(1-\rho)},\quad  \beta(\rho):=\frac{1-\rho}{2\rho-1} \quad \mbox{and}\quad \gamma(\rho):=\frac{\rho}{2\rho-1}\label{eq:constants2}.
\end{equation}
Whenever no confusion arises, we will omit their dependence on $\rho$. 

\begin{remark}\label{rem:h2}
Fix $\rho>\frac12$. Note that $\pi_\rho( \xi(0))=\rho$. Moreover, if one assumes $\ell\geqslant 1$ and takes $\sigma=(0,1,\dots,1)$, by applying the above formula one finds that
\begin{equation}
\pi_\rho\Big(\xi_{|{\Lambda_\ell}}=(0,1,\dots,1)\Big)=(1-\rho)\left(\frac{2\rho-1}{\rho}\right)^{{\ell-2}},
\end{equation}
which, since $\frac{1-\rho}{\rho}=1-\frac{2\rho-1}{\rho}$, is nothing but \[\pi_\rho(\xi(0)=0)\times\P\Big(\mathrm{Geom}\Big(\frac{1-\rho}{\rho}\Big)\geqslant {\ell-2}\Big).\] In fact, we can understand $\pi_\rho$ as the measure which puts a hole in some position with probability $1-\rho$, then puts one plus a random geometric number of parameter $\frac{1-\rho}{\rho}$ particles on its right, then a hole, then starts again and so on. 
\end{remark}

The properties of the (non-product) measure $\pi_\rho$ will be further investigated in later sections. In particular, we prove in Section \ref{sec:LLN} that for any density $\rho\in (\frac12,1)$, under $\pi_\rho$, the correlations between two boxes at a distance of order $\ell$ decay as $\exp(-C \ell)$. We also prove the \emph{equivalence of ensembles}, stated in Corollary \ref{cor:ensembles}, which is a key ingredient to prove the Replacement Lemma, namely Lemma \ref{lem:RL}.

\begin{remark}\label{rem:h1}
As a consequence of equation \eqref{pirho2}, for any $\rho\in (\frac12,1)$, there exist two constants $C=C(\rho)$ and $a=a(\rho)<1$, such that for any configuration $\sigma\in \{0,1\}^{\Lambda_\ell}$, $\pi_\rho\big( \xi_{|\Lambda_\ell}=\sigma\big)\leq C a^\ell$.
\end{remark}

\begin{remark}\label{rem:h3}
Recall that we introduced the local function $h$ in \eqref{eq:h} as \[h(\eta)=\eta(-1)\eta(0)+\eta(0)\eta(1)-\eta(-1)\eta(0)\eta(1),\]
and that we introduced in the Replacement Lemma \ref{lem:RL} the function $F(\rho)=\ccl{\max}\{0,(2\rho-1)/\rho\}.$ Then, one easily checks that for any $\rho\in [\frac12,1]$
\begin{equation}\label{eq:idh}
\pi_\rho(h)=F(\rho),
\end{equation}
which is, as expected, the quantity appearing in the Laplacian in \eqref{eq:hydro}.
\end{remark}

For $\rho>\frac12$, we now introduce a {periodic} variant of the grand canonical measure $\pi_\rho$,which is defined on $\T_N$, is translation invariant, and is locally close in the limit $N\to\infty$ to $\pi_\rho$.

\subsubsection{In the periodic setting} 
\label{sec:defnuN}

Let us fix $\rho>\frac12$. For any periodic configuration $\eta \in \{0,1\}^{\T_N}$, let us first define
\[\widetilde{\nu}_{\rho,N}(\eta):= \frac1N \sum_{x\in\T_N}\pi_\rho\Big(\xi_{|\{1,\ldots,N\}} = \tau_x\eta\Big). \]  
{Note that, since to define the measure above we break the periodicity of the configuration $\eta$, one can find a non-ergodic periodic configuration $\eta\notin \mathcal{E}_N$ such that $\widetilde{\nu}_{\rho,N}(\eta)>0$, because, for some $x$, the non-periodic configuration $\sigma^x:= \tau_x\eta$ belongs to $\widehat{ \mathcal{E}}_{\Lambda_N}$.  We therefore need to restrict}
 $\widetilde\nu_{\rho,N}$ to the periodic ergodic component $\mathcal{E}_N$ as follows. For completeness, we also give a definition of the periodic grand canonical measure for densities below $\frac12$.
\begin{definition}\label{def:nurhoN}
 For $\rho>\frac12$, $\eta\in\T_N$, 
\begin{equation}
\label{eq:aux} \nu_{\rho,N}(\eta):= \frac{\mathbf{1}_{\{\eta \in \mathcal{E}_N\}}}{\widetilde\nu_{\rho,N}(\mathcal{E}_N)} \widetilde\nu_{\rho,N}(\eta).
\end{equation}
For $\rho\leq \frac12$: if $N$ is even, $\nu_{\rho,N}$ is the uniform measure on the set of the two alternate configurations and if $N$ is odd, $\nu_{\rho,N}$ is the uniform measure on the set (of cardinal $N$) of configurations with $(N-1)/2$ isolated holes.
 \end{definition}

We let the reader check the following straightforward properties:
\begin{itemize}
\item [--] the probability measure $\nu_{\rho,N}$ is translation invariant;
\item [--] the probability measure $\nu_{\rho,N}$ is invariant for the dynamics generated by $\mathcal{L}_N$, since it satisfies the detailed balance condition (see \eqref{eq:detailed});
\item [--] for $\rho>\frac12$, we prove in Lemma \ref{lem:nuERG} below that $\widetilde\nu_{\rho,N}(\mathcal{E}_N)=\rho(2-\rho)+\mc O_N(e^{-CN})$. In particular, according to Remark \ref{rem:h1}, there exist two constants $C$ and $ 0<a<1$  depending only on $\rho $ such that for any $N\geq N_0(\rho)$ and any configuration $ \eta\in \{0,1\}^{\T_N}$ 
\begin{equation}
\label{eq:nuNmax}
\nu_{\rho,N}(\eta)\leq C a^N; 
\end{equation}
\item [--] for $\eta,\eta'\in\mathcal{E}_N$ two configurations with the same number of particles, we have $\nu_{\rho,N}(\eta)=\nu_{\rho,N}(\eta')$. In particular, for any $\rho\in(\frac12,1)$, any integer $k$ satisfying $\frac N 2<k<N$, we have $\nu_{\rho,N}(\,\cdot\, \mid \eta\in\Omega_N^k)=\pi^k_N(\,\cdot\,)$ .
\end{itemize}

The rest of this section is dedicated to stating and proving several properties of the grand canonical measures (GCM) defined above. In Section \ref{sec:LLN}, we prove that for any non-degenerate $\rho\in (\frac12,1]$, the correlations under the grand canonical measures $\pi_\rho$ of two boxes at distance $\ell$ decay as $\exp(-C\ell)$ regardless of the size of the boxes. In Section \ref{sec:meserg}, we show  that the measure of the ergodic set $\widetilde\nu_{\rho,N}(\mathcal{E}_N)$ converges exponentially fast to $\rho(2-\rho)$.
In Section \ref{sec:limnuN}, we show that as $N$ diverges, the periodic GCM $\nu_{\rho,N}$ locally resembles the infinite volume GCM $\pi_\rho$. Finally, the main result of this section is the so-called \emph{equivalence of ensembles}, stated in Corollary \ref{cor:ensembles} and proved in Section \ref{sec:ensembles}. These results are quite technical, we therefore elected to present a rather detailed proof in order to be as clear as possible.

\subsection{Decay of correlations for the GCM $\pi_{\rho}$} 
\label{sec:LLN}
We start by investigating the correlations of the infinite volume measure $\pi_\rho$.
Let us define \[P_\ell:=\pi_{\rho}\Big(\xi(0)=\xi(\ell)=1\Big).\]
{We first prove that the two-points correlations under the invariant measure $\pi_{\rho}$ decay exponentially, which in particular yields a law of large numbers for the measure $\pi_\rho$.}
\begin{lemma}
	\label{lem:correlations}
	For any $\rho \in (\frac12,1)$, there exist two constants $C=C(\rho)>0$ and $q=q(\rho)>0$ such that 
	\[\abs{P_\ell-\rho^2}\leq q e^{-C\ell}.\]
\end{lemma}
\begin{proof}[Proof of Lemma \ref{lem:correlations}]
To prove this result, we use equation \eqref{pirho}. Let us recall the notation  $\Lambda_\ell=\{1,\dots,\ell\}$ and rewrite 
\begin{align*}
P_\ell&=\sum_{\sigma \in \{0,1\}^{\Lambda_{\ell-1}}}\pi_\rho\Big(\xi(0)=1,\;\xi_{\mid \Lambda_{\ell-1}}=\sigma,\;\xi(\ell)=1\Big)\\
&=\sum_{\sigma \in  \{0,1\}^{\Lambda_{\ell-2}}}\pi_\rho\Big(\xi(0)=1,\;\xi_{\mid \Lambda_{\ell-2}}=\sigma,\;\xi(\ell-1)=1,\;\xi(\ell)=1\Big)\\
&\quad +\sum_{\sigma \in  \{0,1\}^{\Lambda_{\ell-2}}}\pi_\rho\Big(\xi(0)=1,\;\xi_{\mid \Lambda_{\ell-2}}=\sigma,\;\xi(\ell-1)=0,\;\xi(\ell)=1\Big)\\
&=\alpha(\rho)\beta(\rho)\sum_{\sigma \in  \{0,1\}^{\Lambda_{\ell-2}}}\pi_\rho\Big(\xi(0)=1,\;\xi_{\mid \Lambda_{\ell-2}}=\sigma,\;\xi(\ell-1)=1\Big)\\
&\quad +\alpha(\rho)\beta(\rho)\gamma(\rho) \sum_{\sigma \in  \{0,1\}^{\Lambda_{\ell-2}}}\pi_\rho\Big(\xi(0)=1,\;\xi_{\mid \Lambda_{\ell-2}}=\sigma,\;\xi(\ell-1)=0\Big)\\
&=\frac{2\rho-1}{\rho}\cro{P_{\ell-1}+\frac{\rho}{2\rho-1}(\rho-P_{\ell-1})}.\end{align*}
This in turn yields \[P_\ell=\rho^2+(P_1-\rho^2)\Big(1-\frac{1}{\rho}\Big)^{\ell-1},\]
with $P_1=\pi_\rho\big(\xi(0){\xi}(1)\big)=2\rho-1$.
This proves Lemma \ref{lem:correlations}, since $\rho\in(\frac12,1)$.
\end{proof}
We now use this lemma to state that the correlations under $\pi_{\rho}$ of two boxes at distance $\ell$ from one another decays as $e^{-C\ell}$. 
\begin{corollary}[Correlation decay]
	\label{cor:correlations}
	Fix $k\geq 1$, and shorten 
	\[
	A=A(k):=\{-k+1,\dots,0\}=\tau_{-k}\Lambda_k\quad \mbox{ and }\quad  {B=B(\ell,k)}=\{\ell+1, \dots,\ell+k\}=\tau_{\ell}\Lambda_k.          
	\]
For any two configurations $\sigma_1$, $\sigma_2$ in $\{0,1\}^{A}$,  $\{0,1\}^{{B}}$, any $\rho \in (\frac12,1)$, there exists $C:=C(\rho)>0$ such that 
	\[\frac{\pi_{\rho}\big(\xi_{\mid A}=\sigma_1, \; \xi_{\mid {B}}=\sigma_2\big)}{\pi_{\rho}\big(\xi_{\mid A}=\sigma_1\big)\pi_{\rho}\big(\xi_{\mid{B}}=\sigma_2\big)}=1+ \mc O_{\ell}( e^{-C\ell}),\]
where the $\mc O_{\ell}( e^{-C\ell})$ depends on $\rho$, but can be bounded uniformly in $k$ and $ \sigma_1$, $\sigma_2$.
\end{corollary}
\begin{proof}[Proof of Corollary \ref{cor:correlations}]
To prove this result, first notice that for any two neighboring sets $A_1=\{a,\dots, b\}$, $A_2=\{b+1, \dots,c \}$, and any given two configurations $\sigma_1$, $\sigma_2$ on these sets, 
we can write thanks to the explicit formula \eqref{pirho}
\[\pi_{\rho}\Big(\xi_{\mid A_1}=\sigma_1,\; \xi_{\mid A_2}=\sigma_2\Big)=\frac{{\bf 1}_{\{\sigma_1(b)+\sigma_2(b+1)\geq 1\}}}{\kappa\gamma^{\sigma_1(b)+\sigma_2(b+1)}}\pi_{\rho}\Big(\xi_{\mid A_1}=\sigma_1\Big)\pi_{\rho}\Big(\xi_{\mid A_2}=\sigma_2\Big).\]
Note in particular that if the configuration $(\sigma_1, \sigma_2)$ is not ergodic, then either $\sigma_1$ is not, or $\sigma_2$ is not, or $\sigma_1(b)+\sigma_2(b+1)=0$, so that the right hand side above vanishes. 

Getting back to our problem, summing over all the possible configurations $\sigma$ in $\Lambda_\ell$, and using the identity above (since $A,\Lambda_\ell$ and $B$ are neighboring sets), we can write for any $\sigma_1$, $\sigma_2$ respectively in 
$\{0,1\}^A$, $\{0,1\}^B$,
\begin{multline*}
\pi_{\rho}\Big(\xi_{\mid A}=\sigma_1, \; \xi_{\mid {B}}=\sigma_2\Big)=\pi_{\rho}\Big(\xi_{\mid A}=\sigma_1\Big)\pi_{\rho}\Big(\xi_{\mid {B}}=\sigma_2\Big)\\
\times \sum_{\sigma\in \{0,1\}^{\Lambda_\ell}}\frac{{\bf 1}_{\{\sigma_1(0)+\sigma(1)\geq 1\}}{\bf 1}_{\{\sigma(\ell)+\sigma_2(\ell+1)\geq 1\}}}{\kappa^2\gamma^{\sigma_1(0)+\sigma(1)+\sigma(\ell)+\sigma_2(\ell+1)}}\pi_{\rho}\Big(\xi_{\mid \Lambda_\ell}=\sigma\Big).
\end{multline*}
Thanks to this identity, proving Corollary \ref{cor:correlations} is a matter of verification. 

Start for example by assuming that $\sigma_1(0)=\sigma_2(\ell+1)=1,$ both indicator functions in the sum of the right hand side are always $1$, and since $\kappa\gamma= \rho$, we can rewrite the last identity as 
\begin{equation*}
\pi_{\rho}\Big(\xi_{\mid A}=\sigma_1, \; \xi_{\mid{B}}=\sigma_2\Big)=\frac{\pi_{\rho}\Big(\xi_{\mid {A}}=\sigma_1\Big)\pi_{\rho}\Big(\xi_{\mid {B}}=\sigma_2\Big)}{\rho^2}\pi_{\rho}\Big(\gamma^{-\xi(1)-\xi(\ell)}\Big).
\end{equation*}
Elementary computations and Lemma \ref{lem:correlations} yield $\pi_{\rho}\big(\gamma^{-\xi(1)-\xi(\ell)}\big)=\rho^2+\mc O_{\ell}(e^{-C\ell})$, thus concluding this case.
Since it is a little different, we now consider the case $\sigma_1(0)=\sigma_2(\ell+1)=0.$ In this case, for the indicator functions not to vanish, we must have a particle at each extremity of $\Lambda_\ell$, so that 
\begin{align*}
\pi_{\rho}\Big(\xi_{\mid A}=\sigma_1, \; \xi_{\mid {B}}=\sigma_2\Big)& =\pi_{\rho}\Big(\xi_{\mid {A}}=\sigma_1\Big)\pi_{\rho}\Big(\xi_{\mid {B}}=\sigma_2\Big)\\
& \qquad\qquad \times \sum_{\sigma\in \{0,1\}^{\Lambda_\ell\setminus\{1,\ell\}}}\frac{1}{\kappa^2\gamma^{2}}\pi_{\rho}\Big(\xi_{\mid \Lambda_\ell}=(1,\sigma,1)\Big)\\
& =\pi_{\rho}\Big(\xi_{\mid {A}}=\sigma_1\Big)\pi_{\rho}\Big(\xi_{\mid {B}}=\sigma_2\Big)\frac{\pi_{\rho}\big(\xi(1)=\xi(\ell)=1\big)}{\rho^2}.
\end{align*}
Therefore, Lemma \ref{lem:correlations} proves this case as well. The last case $\sigma_1(0)\neq \sigma_2(\ell+1)$ is treated in the same way.
\end{proof}

\subsection{Measure of the ergodic set}
{We now prove that the measure of the ergodic set $\widetilde \nu_{\rho,N}(\mathcal{E}_N)$ is exponentially close to $\rho({2}-\rho)$.}
\label{sec:meserg}
\begin{lemma}
\label{lem:nuERG}
Fix $\rho \in (\frac12,1]$. Then there exists $C=C(\rho)>0$ such that
\begin{equation}\label{eq:nuERGexpl}
\widetilde \nu_{\rho,N}(\mathcal{E}_N) =\pi_\rho\Big((\xi(1),\xi(N))\neq (0,0)\Big)+\mc O_N(e^{-CN})=\rho(2-\rho)+\mc O_N(e^{-CN}).
\end{equation}
Note that $\rho(2-\rho)>\frac34$.

\end{lemma}

\begin{proof}
For any $1\leq x\leq N$ and any configuration $\xi$ in $\Lambda_N$, denote by $\xi^x$ the \emph{periodic} configuration on $\T_N$ obtained by injecting $\Lambda_N$ in $\T_N$, and then translating by $x\in\T_N$. By definition,
\begin{align*}\widetilde\nu_{\rho,N}(\mathcal{E}_N)&=\sum_{\eta\in\mathcal{E}_N}\frac{1}{N}\sum_{x\in\T_N}\pi_\rho\Bigl(\xi_{|\{1,\ldots,N\}}=\tau_x\eta\Bigr)\\
&=\frac{1}{N}\sum_{x\in\T_N}\pi_\rho\Bigl({\xi^x_{|\{1,\ldots,N\}}}\in\mathcal{E}_N\Bigr)\\
&=\pi_\rho\Bigl({(\xi(1),\xi(N))\neq (0,0)}, \sum_{y=1}^N\xi(y)>\frac N 2\Bigr)\\
&=\pi_\rho\Big({(\xi(1),\xi(N))\neq (0,0)}\Big) \\ 
& \qquad -2\pi_\rho\Bigl(\forall\; i\in\{1,\ldots,N\},\ \xi(i)=0\text{ iff }i\text{ even and } \xi(1)\xi(N)\neq 0\Bigr),
\end{align*}
where the last equality follows from the fact that a.s.\@ under $\pi_\rho$ holes are isolated. What this says is that the only way for $\xi_{|\{1,\ldots,N\}}$ not to be in $\mathcal{E}_N$ is to be one of the two configurations alternating holes and particles (which may only happen if $N$ is even), or to have $\xi(1)=\xi(N)=0$. 
Therefore, using the notations introduced in Sections \ref{sec:GCM} and \ref{sec:LLN} and recalling Remark~\ref{rem:h1}, there exists $a<1$ such that
\begin{align*}
\widetilde\nu_{\rho,N}(\mathcal{E}_N)&=\pi_\rho\Big(\xi(1)(1-\xi(N))+\xi(N)(1-\xi(1))+\xi(1)\xi(N)\Big){+\mc O_N({a^N})}\\
&= P_{N-1}+2{(\rho-P_{N-1})}{+\mc O_N({a^N})}
\end{align*}
The conclusion follows from Lemma \ref{lem:correlations}. 
\end{proof}

\subsection{Local convergence of the periodic GCM $\nu_{\rho,N} $ to the infinite volume GCM $\pi_{\rho}$} 
\label{sec:limnuN}
{In this section, we show that as $N$ diverges, the periodic GCM $\nu_{\rho,N} $ is locally close to the infinite volume GCM $\pi_{\rho}$.
More precisely, we have the following result:}
\begin{lemma}\label{lem:aux2}
Let $f:\{0,1\}^\mathbb{Z} \to \bb R$ be a local function and assume without any loss of generality that its support is included in $B_\ell$, for some fixed $\ell \in \mathbb{N}$. Then for any sequence $(a_N)$ such that $\log N=o_N(a_N)$ and $a_N = o_N(N)$, there exists $C=C(\ell{,\rho})>0$ such that 
\begin{equation}\label{eq:localconvergence} \Big| \nu_{\rho,N}(f) - \pi_\rho(f)\Big| \leqslant {C\|f\|_\infty\frac{a_N}{ N}}.  \end{equation}
Furthermore, if $\rho>\frac12$, for any integer $\ell\geq 1$, and any two configurations $\sigma_1$ and $\sigma_2$ in $\{0,1\}^{B_\ell}$,
\begin{multline}
\label{eq:TBEmesure}
\limsup_{\varepsilon\to 0}\limsup_{N\to\infty}\sup_{\sqrt{N}<|y|<\varepsilon N}\Big|\nu_{\rho,N}\Big(\xi_{\mid B_\ell}=\sigma_1, \; \xi_{\mid \tau_y B_\ell}=\sigma_2\Big)\\
-\pi_{\rho}\Big(\xi_{\mid B_\ell}=\sigma_1\Big)\pi_{\rho}\Big( \xi_{\mid { B_\ell}}=\sigma_2\Big)\Big|=0,
\end{multline}
uniformly in $\sigma_1$, $\sigma_2$.
\end{lemma}\begin{proof}[Proof of Lemma \ref{lem:aux2}]          
{We start by proving \eqref{eq:localconvergence} for $\rho> \frac12$, \eqref{eq:TBEmesure} is proved using the same tools.} Fix a local function $f$ with support included in $B_\ell$, for some fixed $\ell \in \mathbb{N}$. One can write $f$ as a linear combination of functions of the form $f(\xi)=\mathbf{1}_{\{\xi_{|B_\ell}=\sigma\}}$, where $\sigma \in \{0,1\}^{B_\ell}$, {therefore we only need to prove the result for functions of this form.}
For such a function, by definition
\begin{align*}
\nu_{\rho,N}(f) & = \sum_{\substack{\eta \in \mathcal{E}_N\\\eta_{|B_\ell}=\sigma}}\nu_{\rho,N}(\eta) \\
& = \frac{1}{\widetilde\nu_{\rho,N}(\mathcal{E}_N)} \sum_{\substack{\eta \in \mathcal{E}_N\\\eta_{|B_\ell}=\sigma}} \frac{1}{N}\sum_{x=1}^N \pi_\rho\Big( \xi_{\big|\{1,\dots,N\}}=\tau_x\eta\Big). 
\end{align*}
{To proceed, we} use the decoupling property of $\pi_\rho$ (cf. Corollary~\ref{cor:correlations}). We divide the last sum that appears above into two sums: \begin{itemize}
\item [--] one over $A_N:=\big\{a_N,\dots, N-a_N\big\}$, 
\item [--] and the other one over the complement, namely $\{1,\dots,N\}\setminus A_N$. 
\end{itemize} From now on we assume $N$ sufficiently large such that 
$\ell<a_N$. Let us treat the first sum: for any 
$x \in A_N$, 
\begin{align*}
\sum_{\substack{\eta \in \mathcal{E}_N\\\eta_{|B_\ell}=\sigma}} & \pi_\rho\Big( \xi_{\big|\{1,\dots,N\}}=\tau_x\eta\Big)\\ & = \pi_\rho \Big( \xi_{\big|\{x-\ell,\dots,x+\ell\}} = \sigma \quad \text{and}\quad (\xi(1),\xi(N))\neq (0,0)\Big) +\mc O_N(e^{-CN})\\
& = \pi_\rho\Big(\xi_{\big|\{x-\ell,\dots,x+\ell\}} = \sigma\Big) \pi_\rho\Big((\xi(1),\xi(N))\neq (0,0)\Big) + \mc O_N(e^{-CN}+e^{-Ca_N})\\
& = \pi_\rho\Big( \xi_{\big|\{x-\ell,\dots,x+\ell\}} = \sigma\Big) \widetilde\nu_{\rho,N}(\mathcal{E}_N)+\mc O_N(e^{-CN}+e^{-Ca_N}), 
\end{align*}
where $C=C(\rho,\ell)>0$. {In the identity above, the term $ \mc O_N(e^{-CN})$ comes from the fact that alternate configurations like $(1,0,\dots,1,0)$ and $(0,1,\dots,0,1)$ have positive probability under $\pi_{\rho}$ but do not appear in $\mathcal{E}_N$, which excludes configurations with density $\frac12$ or less.} The first equality follows from the isolation of holes in $\pi_\rho$ (recall the proof of Lemma~\ref{lem:nuERG}) and the last equality from {Corollary \ref{cor:correlations}}. Indeed, for any $x\in A_N$, the points belonging to $\{x-\ell,\dots x+\ell\}$ are at distance at least $a_N-\ell-1$ from both points $1$ and $N$. Therefore, one can write the first sum over $A_N$ as (changing the constants in the exponentials)
\begin{align}
\frac{1}{N\widetilde\nu_{\rho,N}(\mathcal{E}_N)}&\sum_{x\in A_N}\sum_{\substack{\eta \in \mathcal{E}_N\\\eta_{|B_\ell}=\sigma}}  \pi_\rho\Big( \xi_{\big|\{1,\dots,N\}}=\tau_x\eta\Big)\notag\\
& = \frac{1}{N}\sum_{x\in A_N} \pi_\rho\Big( \xi_{\big|\{x-\ell,\dots,x+\ell\}}=\sigma\Big) + \mc O_N\big(e^{-CN}+e^{-Ca_N}\big)\notag\\
&= \pi_\rho(f)+\mc O_N\Big(e^{-CN}+e^{-Ca_N}+\frac{a_N}{N}\Big).\label{eq:transl2}
\end{align}
Note that we used the fact that $\widetilde\nu_{\rho,N}(\mathcal{E}_N)$ is bounded away from $0$ (from Lemma~\ref{lem:nuERG}) and the translation invariance of $\pi_\rho$. 

Let us now turn to the second sum over $\{1,\dots,N\}\setminus A_N$: for any $x$, we can bound
\begin{align*}
\sum_{\substack{\eta \in \mathcal{E}_N\\\eta_{|B_\ell}=\sigma}}  \pi_\rho\Big(\xi_{\big|\{1,\dots,N\}}=\tau_x\eta\Big)
& \leqslant \sum_{\eta \in \mathcal{E}_N} \pi_\rho\Big(\xi_{\big|\{1,\dots,N\}}=\tau_x\eta\Big) = \widetilde\nu_{\rho,N}(\mathcal{E}_N).
\end{align*}
Consequently, the second sum can be easily bounded as follows:
\[
\frac{1}{N\widetilde\nu_{\rho,N}(\mathcal{E}_N)}\sum_{x\in\{1,\dots,N\}\setminus A_N}\sum_{\substack{\eta \in \mathcal{E}_N\\\eta_{|B_\ell}=\sigma}}  \pi_\rho\Big( \xi_{\big|\{1,\dots,N\}}=\tau_x\eta\Big) \leqslant \frac{2a_N}{N}.
\]
This proves \eqref{eq:localconvergence}.

We simply sketch the proof of the second identity \eqref{eq:TBEmesure}, which relies on the same steps, together with the decoupling property for $\pi_\rho$ (cf. Corollary \ref{cor:correlations}). For any given $\varepsilon>0$ (which will tend to 0), we can set $a_N=3\varepsilon N$, and repeat the same proof as before. We obtain for any fixed configurations $ \sigma^1$, and $\sigma^2$ on $B_\ell$, and for any $ 2\ell+2<|y|<\varepsilon N$, that
\begin{equation*}
\nu_{\rho,N}\Big(\xi_{\mid B_\ell}=\sigma^1,\;  \xi_{\mid \tau_y B_\ell}=\sigma^2\Big)=\pi_{\rho}\Big(\xi_{\mid B_\ell}=\sigma^1,\; \xi_{\mid \tau_y B_\ell}=\sigma^2\Big)+o_{\varepsilon}(1)+o_N(1).
\end{equation*}
We can then use the decay of correlations proved for $\pi_{\rho}$ in  Corollary \ref{cor:correlations} and translation invariance of $\pi_\rho$, to obtain that for any fixed configurations $\sigma^1$, $\sigma^2$ on $B_\ell$,
\begin{multline*}
\limsup_{\varepsilon\to 0}\limsup_{N\to\infty}\sup_{\sqrt{N}<|y|<\varepsilon N}\Big|\nu_{\rho,N}\Big(\xi_{\mid B_\ell}=\sigma^1, \; \xi_{\mid \tau_y B_\ell}=\sigma^2\Big)\\
-\pi_{\rho}\Big(\xi_{\mid B_\ell}=\sigma^1\Big)\pi_{\rho}\Big( \xi_{\mid {B_\ell}}=\sigma^2\Big)\Big|=0,
\end{multline*}
uniformly in $\sigma^1$, $\sigma^2$ as wanted.

To conclude the proof of Lemma~\ref{lem:aux2}, recall that for $\rho\leq \frac12$, we defined $\nu_{\rho,N}$ as the uniform measure on the set of the two alternate configurations if $N$ is even; and on the set of configurations with $(N-1)/2$ isolated holes if $N$ is odd. With this observation, \eqref{eq:localconvergence} is a matter of simple verification.
\end{proof}

\subsection{Measures on a finite box}
We now define the projections of the measures $\pi_\rho,\nu_{\rho,N}$ on a finite box $B_\ell$.

Recall the explicit expression \eqref{pirho} for the GCM $\pi_\rho$, which straightforwardly translates to $B_\ell$ instead of $\Lambda_\ell$, by translation invariance . 
For any $\ell \in \bb N$, let us define the probability measure $\hat\pi_{\rho}^{\ell}$ on $\ERGt_{B_\ell}$ as

\begin{equation}
\label{eq:hatpiell} 
\hat\pi_{\rho}^{\ell}({\sigma}):=\pi_\rho\Big(\xi_{|B_\ell}={\sigma}\Big)=\kappa(\rho)\;\alpha(\rho)^p\;\beta(\rho)^{2\ell+1}\;\gamma(\rho)^{\sigma(-\ell)+\sigma(\ell)},
\end{equation}
where $p=\sum_{x\in B_\ell}\sigma(x)$. Then, for any integer $j \in \{{\ell},\dots,2\ell+1\}$ we denote by $\hat{\pi}_{\rho}^{\ell,j}$ the measure $\hat\pi_{\rho}^{\ell}$ conditioned on having $j$ particles. Namely, if we denote by $\hat\Omega_{\ell}^j$ the hyperplane \[\hat\Omega_{\ell}^j:=\Big\{\sigma \in \ERGt_{B_\ell}\; : \; \sum_{x\in B_\ell}\sigma(x)=j\Big\},\] then $\hat{\pi}_{\rho}^{\ell,j}$ is the probability measure on $\hat\Omega_{\ell}^j$ such that 
\begin{equation}\label{eq:hyp}\hat{\pi}_{\rho}^{\ell,j}(\sigma) = \frac{\hat\pi_{\rho}^{\ell}(\sigma)}{\hat\pi_{\rho}^{\ell}(\hat\Omega_{\ell}^j)}, \qquad \text{for any } \sigma \in\hat\Omega_{\ell}^j.\end{equation}
Similarly, we define the probability measure ${\hat\nu_{\rho,N}^\ell}$ on $\ERGt_{B_\ell}$ as
\begin{equation}
\label{eq:cms}
\hat\nu_{\rho,N}^{\ell}(\sigma)=\nu_{\rho,N}\big(\eta_{|B_\ell}=\sigma\big).
\end{equation}
Note in particular that according to Lemma \ref{lem:aux2}, for any sequence $(a_N)$ such that $\log N =o_N(a_N)$ and $a_N=o_N(N)$, one can write
\begin{equation}\label{eq:localconvfinitebox}
\sup_{\sigma\in\hat{\mathcal{E}}_{B_\ell}}\left|\hat\nu_{\rho,N}^\ell(\sigma)-\hat\pi_{\rho}^{\ell}(\sigma)\right|\leq \frac{C(\ell,\rho)a_N}{N}.
\end{equation}

\subsection{Equivalence of ensembles}
\label{sec:ensembles}
We now state and prove the main result for this section, namely the equivalence of ensembles, which states that far enough from the extremities $-\ell$ and $\ell$, the canonical measure $\hat{\pi}_{\rho}^{\ell,j}$ defined in the previous paragraph is locally close to the GCM $\pi_{j/(2\ell+1)}$ with parameter the empirical density $j/(2\ell +1)$.

Let $k \in \bb N$ be fixed, define $B_k(x)=\{x-k,\dots,x+k\}$ (in particular $B_k=B_k(0)$) and
\[\rho_\ell=\rho_\ell(j):=\frac{j}{2\ell+1},\]
the density in $B_\ell$ under $\hat{\pi}_\rho^{\ell,j}$.
Furthermore, let us introduce 
\[E_\ell =\{{\ell},\dots, 2\ell+1\}, \quad \mbox{ and }\quad E_\ell(\delta)=\big\{\lceil{\ell}(1+\delta)\rceil,\dots, \lfloor(1-\delta)(2\ell+1)\rfloor\big\},\]
which are respectively the set of possible particle numbers $j$ in $B_\ell$, and the same set, excluding the densities $\rho_\ell(j)$ which are at distance less than $\delta$ to the critical densities $\frac12$ and $1$. Note that $E_\ell=E_\ell(0)$ and $E_\ell(\delta) \subset E_\ell$ for any $\delta >0$.
\begin{proposition} 
\label{prop:ensembles}
 Fix a configuration $\sigma\in \{0,1\}^{B_{k}}$. Then, for any $\rho \in (\frac12,1)$, any $\delta_1, \delta_2>0$, there exists a constant $C_0:=C_0(k,\sigma,\delta_1,\delta_2)$ such that for any  $\ell>0$,
\begin{equation}
\label{eq:ensembles1}
\max_{\substack{j\in E_\ell(\delta_1)\\
x\in B_{(1-\delta_2)\ell-k}}}\Big|\hat{\pi}_{\rho}^{\ell,j}\big({\varsigma}_{|B_k(x)}=\sigma\big)-\pi_{\rho_\ell(j)}\big({\xi}_{|B_k}=\sigma\big)\Big| \leq \frac{C_0}{\ell^{1/4}}.
\end{equation}
Furthermore, for some constant $C_1:=C_1(k, \delta_2)$,
\begin{equation}
\label{eq:ensembles2}
\limsup_{\ell\to\infty}\max_{\substack{j\in E_\ell\setminus E_\ell(\delta_1)\\
x\in B_{(1-\delta_2)\ell-k}}}\Big|\hat{\pi}_{\rho}^{\ell,j}\big({\varsigma}_{|B_k(x)}=\sigma\big)-\pi_{\rho_\ell(j)}\big({\xi}_{|B_k}=\sigma\big)\Big| \leq C_1\delta_1.
\end{equation}
\end{proposition}
Note that $B_k(x)$ splits $B_\ell$ in three parts: a left cluster, $B_k(x)$ itself, and a right cluster. The condition $x\in B_{(1-\delta_2)\ell-k}$ ensures that both  left and right clusters are of size at least $\delta_2\ell$. 
Because we need an estimate which is uniform in the number of particles and in the positions at which the box is taken, the proof of this result is quite technical. To prove the Replacement Lemma \ref{lem:RL}, we will use the following corollary of Proposition \ref{prop:ensembles}.

\begin{corollary}[Equivalence of ensembles]
\label{cor:ensembles}
For any local function $f:\{0,1\}^{\bb Z}\to\bb R$, any $\rho \in (\frac12,1]$, and any $\delta>0$, we have
\begin{equation}
\label{eq:ensembles3}
\lim_{\ell\to\infty}\;\;\max_{\substack{j\in E_\ell\\
x\in B_{(1-\delta)\ell}}} \Big|\hat{\pi}_{\rho}^{\ell,j}(\tau_x f)-\pi_{\rho_\ell(j)}(f)\Big|=0.
\end{equation}
\end{corollary}
We start by briefly proving that Proposition \ref{prop:ensembles} yields Corollary \ref{cor:ensembles}.
\begin{proof}[Proof of Corollary \ref{cor:ensembles}]
Fix $\rho\in (\frac12,1)$ and a local function $f$. There exists an integer $k$ such that $f$ only depends on sites in $B_k$. Then, if we choose $ \delta_2={\delta/2}${, for any $\ell>2k/\delta$, we have $B_{{(1-\delta)\ell}}\subset B_{(1-\delta_2)\ell-k}$.} Then, for $\ell$ large enough, we can
 use the triangular inequality and  Proposition \ref{prop:ensembles} to write for any $\delta_1>0$, $x\in B_{(1-\delta)\ell}$, $j\in E_\ell$,

\begin{align*}
\Big|\hat{\pi}_{\rho}^{\ell,j}&(\tau_x f)-\pi_{\rho_\ell(j)}(f)\Big|\\
&\leq 2^{2k+1}\|f\|_{\infty}\max_{\substack{j\in E_{\ell}\\
{\sigma}\in \{0,1\}^{B_k}}}\Big|\hat{\pi}_{\rho}^{\ell,j}\big({\varsigma}_{|B_k(x)}=\sigma\big)-\pi_{\rho_\ell(j)}\big({\xi}_{|B_k}=\sigma\big)\Big|\\
&\leq 2^{2k+1}\|f\|_{\infty}\bigg(\ell^{-1/4}\max_{\sigma\in \{0,1\}^{B_k}}\big\{C_0(k,{\sigma},\delta_1,\delta_2)\big\}+(1+o_\ell(1))\delta_1C_1(k,\delta_2)\bigg).
\end{align*} 
The last estimate has been obtained after writing $E_\ell=E_\ell(\delta_1) \cup (E_\ell \setminus E_\ell(\delta_1))$ and applying both results of Proposition \ref{prop:ensembles}. Since $k$ is fixed, we then let $\ell\to\infty$, then $\delta_1\to 0$ to prove Corollary \ref{cor:ensembles}.
\end{proof}
We now prove Proposition \ref{prop:ensembles}.
\begin{proof}[Proof of Proposition \ref{prop:ensembles}]
We start by the case where the fixed density $\rho_\ell(j)$ is close to the extreme values $\frac12$ or $1$ (i.e. $j \in E_\ell \setminus E_\ell(\delta_1)$).
\medskip
 
\paragraph{\textsc{Proof of \eqref{eq:ensembles2}}} We will only detail the case $\rho_\ell(j)\geq 1 -\delta_1$, since the case $\rho_\ell(j)\leq (1+\delta_1)/2$ is treated in the same way. 
Denote by $ {\bf 1}_{k}$ the constant configuration on $B_k$ with one particle at each site, and no empty site. To prove \eqref{eq:ensembles2}, it is sufficient to show that for some constant $C_2:=C_2(k, \, \delta_2)$, we have for any $j \in E_\ell \setminus E_\ell(\delta_1)$  satisfying  $\rho_\ell(j)\geq 1 -\delta_1$,
\[\pi_{\rho_\ell(j)}\big({\xi}_{|B_k}= {\bf 1}_{k} \big)\geq 1-C_2 \delta_1,\]
and  for any $x \in B_{(1-\delta_2)\ell-k}$,
\begin{equation}
\label{eq:rhogrand}
\hat{\pi}_{\rho}^{\ell,j}\big({\varsigma}_{|B_k(x)}= {\bf 1}_{k}\big)\geq 1-C_2\delta_1,
\end{equation}
and then let $C_1(k, \delta_2):=2C_2(k, \delta_2).$
The first bound is an immediate consequence of the explicit formula \eqref{pirho2} for $\pi_\rho$.   
From now on, for $\varepsilon_1, \varepsilon_2\in \{0,1\}$, we denote by 
\[\hat{\pi}_{\rho,\varepsilon_1, \varepsilon_2}^{\ell,j}(\cdot)=\hat{\pi}_{\rho}^{\ell,j}\Big(\cdot\mid {\varsigma}(-\ell)=\varepsilon_1, {\varsigma}(\ell)=\varepsilon_2\Big),\]
the measure $\hat{\pi}_{\rho}^{\ell,j}$ conditioned to be in the state $\varepsilon_1$ (resp. $\varepsilon_2$) at site $ -\ell$ (resp. $\ell$). Then, to prove \eqref{eq:rhogrand}, it is clearly sufficient to prove for any $\varepsilon_1,\varepsilon_2\in\{0,1\}$
\[\hat{\pi}_{\rho, \varepsilon_1, \varepsilon_2}^{\ell,j}\big({\varsigma}_{|B_k(x)}= {\bf 1}_{k}\big)\geq 1-C_2\delta_1.\]
Furthermore, to prove the latter, it is sufficient to prove that for some constant $C(k,  \delta_2)$, for any $\sigma\not= {\bf 1}_{k}\in\{0,1\}^{B_k}$ 
\[\frac{\hat{\pi}_{\rho, \varepsilon_1, \varepsilon_2}^{\ell,j}\big({\varsigma}_{|B_k(x)}=\sigma\big)}{\hat{\pi}_{\rho, \varepsilon_1, \varepsilon_2}^{\ell,j}\big({\varsigma}_{|B_k(x)}= {\bf 1}_{k}\big)}\leq C(k,  \delta_2)\delta_1.\]
Indeed, summing this bound above over $\sigma\in\{0,1\}^{B_k}\setminus\{{\bf 1}_{k}\}$, yields
\[\hat{\pi}_{\rho, \varepsilon_1, \varepsilon_2}^{\ell,j}\big({\varsigma}_{|B_k(x)}= {\bf 1}_{k}\big)\geq \frac{1}{1+2^{|B_k|}C(k,  \delta_2)\delta_1}\geq 1-C_{2}(k,  \delta_2)\delta_1,\]
as wanted. We therefore state the result we need as a separate lemma, the proof of which requires elements of the proof of  \eqref{eq:ensembles1}.

\begin{lemma}
\label{lem:rhogrand}
There exists a constant $C=C(k, \delta_2)$ such that for any non-full configuration $\sigma\in\{0,1\}^{B_k}\setminus \{{\bf 1}_{k}\}$,  any $j\geq (1-\delta_1)(2\ell+1)$ and any $\varepsilon_1,\varepsilon_2\in\{0,1\}$
 \begin{equation}
 \label{eq:rhograndlem}
\limsup_{\ell\to\infty}\frac{\hat{\pi}_{\rho, \varepsilon_1, \varepsilon_2}^{\ell,j}\big({\varsigma}_{|B_k(x)}=\sigma\big)}{\hat{\pi}_{\rho, \varepsilon_1, \varepsilon_2}^{\ell,j}\big({\varsigma}_{|B_k(x)}= {\bf 1}_{k}\big)}\leq C(k, \delta_2)\delta_1.  
 \end{equation}
The same is true if   $j\leq (1+\delta_1){\ell}$, and with $\{{\varsigma}_{|B_k(x)}= {\bf 1}_{k}\}$ replaced by $\{{\varsigma}_{|B_k(x)}\in {\bf 1/2}_{k}\}$, where ${\bf 1/2}_{k}$ represents the set which contains the two alternate configurations $\circ\bullet$ and $\bullet\circ$ with respectively $k$, $k+1$ particles.
\end{lemma}
In order to introduce the notations and estimates needed to prove this result, we postpone the proof of Lemma \ref{lem:rhogrand} to the end of the section, and for now prove the first estimate \eqref{eq:ensembles1}:

\medskip

\paragraph{\textsc{Proof of \eqref{eq:ensembles1}}} 
Fix some particle number $j\in E_\ell(\delta_1)$, we drop to simplify our notation the dependency in $j$ and simply write  $\rho_\ell\in [\frac12+\delta_1, 1-\delta_1]$ for the density in $B_\ell$.
For any $\ell>\ell_0$ , (where $\ell_0$ only depends  on  $k$) we must have for $\sigma\in\hat{\mc E}_{B_k}$
\begin{equation}\label{eq:pro}\hat{\pi}_{\rho, \varepsilon_1, \varepsilon_2}^{\ell,j}\big({\varsigma}_{|B_k(x)}=\sigma\big)> 0\quad \mbox{ and }\quad \pi_{\rho_\ell}\big({\xi}_{|B_k}=\sigma\big)> 0.\end{equation}
Indeed, the first statement holds because, since $\rho_\ell$ is bounded away from $1$ and $\frac12$, for any $j\in E_\ell(\delta_1)$, one will always be able to find a  configuration ${\varsigma} $ in $\Omega^j_\ell$ such that ${\varsigma} _{|B_k(x)}=\sigma$. The second statement holds because  $\rho_{\ell}$ is neither $ \frac12$ nor $1$.

\medskip

From \eqref{eq:pro} we can write
\begin{align*}\Big|\hat{\pi}_{\rho, \varepsilon_1, \varepsilon_2}^{\ell,j}&\big({\varsigma}_{|B_k(x)}=\sigma\big)-\pi_{\rho_\ell}\big({\xi}_{|B_k}=\sigma\big)\Big|\\
&\leq \frac{\Big|\hat{\pi}_{\rho, \varepsilon_1, \varepsilon_2}^{\ell,j}\big({\varsigma}_{|B_k(x)}=\sigma\big)-\pi_{\rho_\ell}\big({\xi}_{|B_k}=\sigma\big)\Big|}{\hat{\pi}_{\rho, \varepsilon_1, \varepsilon_2}^{\ell,j}\big({\varsigma}_{|B_k(x)}=\sigma\big)\pi_{\rho_\ell}\big({\xi}_{|B_k}=\sigma\big)}\hat{\pi}_{\rho, \varepsilon_1, \varepsilon_2}^{\ell,j}\big({\varsigma}_{|B_k(x)}=\sigma\big)\\
&\leq \Bigg|\frac{1}{\hat{\pi}_{\rho, \varepsilon_1, \varepsilon_2}^{\ell,j}\big({\varsigma}_{|B_k(x)}=\sigma\big)}-\frac{1}{\pi_{\rho_\ell}\big({\xi}_{|B_k}=\sigma\big)}\Bigg|\hat{\pi}_{\rho, \varepsilon_1, \varepsilon_2}^{\ell,j}\big({\varsigma}_{|B_k(x)}=\sigma\big)\\
&= \Bigg|\sum_{\sigma'\in\{0,1\}^{B_k}}\frac{\hat{\pi}_{\rho, \varepsilon_1, \varepsilon_2}^{\ell,j}\big({\varsigma}_{|B_k(x)}=\sigma'\big)}{\hat{\pi}_{\rho, \varepsilon_1, \varepsilon_2}^{\ell,j}\big({\varsigma}_{|B_k(x)}=\sigma\big)}-\frac{\pi_{\rho_\ell}\big({\xi}_{|B_k}=\sigma'\big)}{\pi_{\rho_\ell}\big({\xi}_{|B_k}=\sigma\big)}\Bigg|\hat{\pi}_{\rho, \varepsilon_1, \varepsilon_2}^{\ell,j}\big({\varsigma}_{|B_k(x)}=\sigma\big)\\
&\leq  2^{2k+1}\max_{\sigma'\in\{0,1\}^{B_k}}\Bigg|\frac{\hat{\pi}_{\rho, \varepsilon_1, \varepsilon_2}^{\ell,j}\big({\varsigma}_{|B_k(x)}=\sigma'\big)}{\hat{\pi}_{\rho, \varepsilon_1, \varepsilon_2}^{\ell,j}\big({\varsigma}_{|B_k(x)}=\sigma\big)}-\frac{\pi_{\rho_\ell}\big({\xi}_{|B_k}=\sigma'\big)}{\pi_{\rho_\ell}\big({\xi}_{|B_k}=\sigma\big)}\Bigg|\hat{\pi}_{\rho, \varepsilon_1, \varepsilon_2}^{\ell,j}\big({\varsigma}_{|B_k(x)}=\sigma\big). 
\end{align*}
The interest of this elementary computation is that the quotient of the $\hat{\pi}_{\rho, \varepsilon_1, \varepsilon_2}^{\ell,j}({\varsigma}_{|B_k(x)}=\sigma)$'s is far easier to estimate than the $\hat{\pi}_{\rho, \varepsilon_1, \varepsilon_2}^{\ell,j}({\varsigma}_{|B_k(x)}=\sigma)$'s themselves, because the normalizing constants 
\[\pi_{\rho}\bigg(\sum_{x\in B_\ell}{\xi}(x)=j,\; {\xi}(-\ell)=\varepsilon_1,\;{\xi}(\ell)=\varepsilon_2\bigg)\] 
cancel out.
The first estimate \eqref{eq:ensembles1} is therefore a consequence of Lemma \ref{lem:ensembles3} below, by defining 
\[C_0(k,\sigma,\delta_1, \delta_2)=2^{2k+1}\max_{\sigma'\in\{0,1\}^{B_k}}C_3(k,\sigma,\sigma', \delta_1, \delta_2),\]
thus proving Proposition \ref{prop:ensembles}.
\end{proof}

\begin{lemma}
\label{lem:ensembles3}
Fix $k \in \bb N$ as well as two configurations $\sigma, \sigma'\in \{0,1\}^{B_{k}}$. For any $\rho \in (\frac12,1)$, any $\varepsilon_1,\varepsilon_2 \in \{0,1\}$, and any positive $\delta_1, \;\delta_2$, there exists a constant $C_3:=C_3(k, \sigma,\sigma',\delta_1, \delta_2)$ such that 
\begin{equation}
\label{eq:ensembles4}
\max_{\substack{j\in E_\ell(\delta_1)\\
x\in B_{(1-\delta_2)\ell-k}}}\Bigg|\frac{\hat{\pi}_{\rho, \varepsilon_1, \varepsilon_2}^{\ell,j}\big({\varsigma}_{|B_k(x)}=\sigma'\big)}{\hat{\pi}_{\rho, \varepsilon_1, \varepsilon_2}^{\ell,j}\big({\varsigma}_{|B_k(x)}=\sigma\big)}-\frac{\pi_{\rho_\ell}\big({\xi}_{|B_k}=\sigma'\big)}{\pi_{\rho_\ell}\big({\xi}_{|B_k}=\sigma\big)}\Bigg| \leq \frac{C_3}{\hat{\pi}_{\rho, \varepsilon_1, \varepsilon_2}^{\ell,j}\big({\varsigma}_{|B_k(x)}=\sigma\big)\ell^{1/4}}.
\end{equation}
\end{lemma}
\begin{proof}[Proof of Lemma \ref{lem:ensembles3}]
 Denote by $p(\sigma)$ the number of particles in $\sigma$. Then, thanks to the explicit formula \eqref{pirho2} for $\pi_\rho$, we can rewrite the second term in the left hand side of \eqref{eq:ensembles4} as
\begin{equation}
\label{eq:quotient2} 
\frac{\pi_{\rho_\ell}\big({\xi}_{|B_k}=\sigma'\big)}{\pi_{\rho_\ell}\big({\xi}_{|B_k}=\sigma\big)}=\alpha\pa{\rho_\ell}^{p(\sigma)-p(\sigma')}\gamma\pa{\rho_\ell}^{\sigma'(-k)+\sigma'(k)-\sigma(-k)-\sigma(k)}, 
\end{equation}
where $\alpha$ and $\gamma$ were introduced before in \eqref{eq:constants2}, recall 
\[\alpha(\rho)=\frac{(2\rho-1)^2}{\rho(1-\rho)}\quad \mbox{ and }\quad \gamma(\rho)=\frac{\rho}{2\rho-1}.\]
We now estimate the first part in the left hand side of \eqref{eq:ensembles4}. Once again, we will only write the case $\varepsilon_1=\varepsilon_2=1$, the others can be treated similarly. Recall that $x\in B_{(1-\delta_2)\ell-k}$, and define 
\[\ell_1:=\ell_1(x,\ell,k)=\ell-k+x\geq \delta_2\ell,\quad  \ell_2:=\ell_1(x,\ell,k)=\ell-k-x\geq \delta_2\ell,\] 
which are the respective sizes of the clusters to the left and right of $B_k(x)$ in $B_\ell$.
We denote by $n$ the number of particles in the  cluster $\{-\ell,\dots,-\ell+\ell_1-1\}$ to the left of $B_k(x)$ in $B_\ell$, which yields the decomposition 
\begin{equation}
\label{eq:quotjlrho}
\frac{\hat{\pi}_{\rho,1,1}^{\ell,j}\big({\varsigma}_{\mid B_k(x)}=\sigma'\big)}{\hat{\pi}_{\rho,1,1}^{\ell,j}\big({\varsigma}_{\mid B_k(x)}=\sigma\big)}=\frac{\sum_{n=1}^{j}a_n(\sigma')}{\sum_{n=1}^{j}a_n(\sigma)}. 
\end{equation}
where 
\begin{equation}
\label{eq:Defan}
a_n(\sigma)=a_n(\sigma, j,\ell_1, \ell_2):=\binom{n-1+\sigma(-k)}{\ell_1-n}\binom{j- n-p(\sigma)-1+\sigma(k)}{\ell_2-(j- n-p(\sigma))}.
\end{equation}
The first binomial is the number of ergodic configuration\ccl{s} on the left cluster with $n$ particles and compatible with the left boundary of $\sigma$, whereas the second binomial coefficient is the number of ergodic configurations on the right cluster compatible with the right boundary of $\sigma$.
The quantity $a_n(\sigma)$ is therefore  the number of ergodic configurations on $B_\ell$ with $j$ particles, such that the configuration in $B_k(x)$ is given by $\sigma$, and such that the number of particles in $\{-\ell,\dots,-\ell+\ell_1-1\}$ (resp. $\{\ell-\ell_2+1,\dots, \ell\}$)  is $n$ (resp. $j-p(\sigma)-n$).

Note that in both sums in the right hand side of \eqref{eq:quotjlrho}, some terms vanish (using the convention $\binom{a}{b}=0$ if $b\notin \{0,\dots, a\}$) therefore we do not need to be too specific with the index limits.
Let us shorten 
\begin{equation}
\label{eq:DefSxi}
S(\sigma)=\sum_{n=1}^{j}a_n(\sigma),\quad\mbox{ and }\quad S^*=\sum_{n=1}^{j}a_n^*,
\end{equation}
where
\[a_{n}^*=\binom{n}{\ell_1-n}\binom{j-n}{\ell_2-(j-n)}.\]
Given these notations, we are going to prove that Lemma \ref{lem:ensembles3} is a consequence of Lemma \ref{lem:equiv4} below.
\begin{lemma}
\label{lem:equiv4}
For any configuration $\sigma\in \{0,1\}^{B_{k}}$, and any positive $\delta_1, \;\delta_2$, there exists a constant $C_4:={C_4(k,\delta_1, \delta_2)}$ such that for all $j\in E_\ell(\delta_1)$
\begin{equation*} 
\max_{{x\in B_{(1-\delta_2)\ell-k}}}\abs{S(\sigma)-F(\rho_\ell, \sigma)S^*}\leq \frac{C_4}{\ell^{1/4}}S^*
\end{equation*}
where $F$ is the function
\begin{equation}
\label{eq:deff}
F(\rho,\sigma):=\frac{(2\rho-1)^{2p(\sigma)+2-\sigma(-k)-\sigma(k)}}{\rho^{p(\sigma)+2-\sigma(-k)-\sigma(k)}(1-\rho)^{p(\sigma)}}=\frac{(2\rho-1)^2}{\rho^2}\alpha\pa{\rho}^{p(\sigma)}\gamma\pa{\rho}^{\sigma(-k)+\sigma(k)}.
\end{equation}
\end{lemma}
Before proving this lemma, we show that it is enough to conclude the proof of Lemma \ref{lem:ensembles3}.
For that purpose, write 
\[\frac{\hat{\pi}_{\rho,1,1}^{\ell,j}({\varsigma}_{\mid B_k(x)}=\sigma')}{\hat{\pi}_{\rho,1,1}^{\ell,j}({\varsigma}_{\mid B_k(x)}=\sigma)}=\frac{S(\sigma')}{S(\sigma)}=\frac{F(\rho_\ell, \sigma')}{F(\rho_\ell, \sigma)}+\frac{F(\rho_\ell, \sigma)S^*S(\sigma')-F(\rho_\ell, \sigma')S^*S(\sigma)}{F(\rho_\ell, \sigma)S^*S(\sigma)},\]
so that, thanks to Lemma \ref{lem:equiv4},
\[\abs{\frac{S(\sigma')}{S(\sigma)}-\frac{F(\rho_\ell, \sigma')}{F(\rho_\ell, \sigma)}}\leq \frac{C_5}{\ell^{1/4}}\frac{S(\sigma')+S(\sigma)}{F(\rho_\ell, \sigma)S(\sigma)},\]
where we defined $C_5=C_5(k,\delta_1, \delta_2)=2C_4(k,\delta_1, \delta_2)$.

\medskip

By definition of $F$ and equation \eqref{eq:quotient2}, \[\frac{\pi_{\rho}\big({\xi}_{|B_k}=\sigma'\big)}{\pi_{\rho}\big({\xi}_{|B_k}=\sigma\big)}=\frac{F(\rho,\sigma')}{F(\rho,\sigma)},\] and since $j\in E_\ell(\delta_1)$, $F(\rho_\ell,\sigma)$ is bounded away from 0 by some constant $c_F:=c_F(k,\delta_1 )>0$. The  bound above can therefore be rewritten for any $j\in E_\ell(\delta_1)$
\begin{multline*}
\Bigg|\frac{\hat{\pi}_{\rho,1,1}^{\ell,j}\big({\varsigma}_{\mid B_k(x)}=\sigma'\big)}{\hat{\pi}_{\rho,1,1}^{\ell,j}\big({\varsigma}_{\mid B_k(x)}=\sigma\big)}-\frac{\pi_{\rho_\ell}\big({\xi}_{|B_k}=\sigma'\big)}{\pi_{\rho_\ell}\big({\xi}_{|B_k}=\sigma\big)}\Bigg|\\
\leq \frac{C_5}{c_F\; \ell^{1/4}}\pa{1+\frac{\hat{\pi}_{\rho,1,1}^{\ell,j}\big({\varsigma}_{\mid B_k(x)}=\sigma'\big)}{\hat{\pi}_{\rho,1,1}^{\ell,j}\big({\varsigma}_{\mid B_k(x)}=\sigma\big)}} \leq \frac{C_3}{\ell^{1/4}\;\hat{\pi}_{\rho,1,1}^{\ell,j}\big({\varsigma}_{\mid B_k(x)}=\sigma\big)} ,
\end{multline*}
where $C_3:=C_3(k, \delta_1, \delta_2)=2C_5/c_F$, which proves Lemma \ref{lem:ensembles3}.
\end{proof}
We now prove Lemma \ref{lem:equiv4}.
\begin{proof}[Proof of Lemma \ref{lem:equiv4}]
Let us define
\begin{align*}&n^\sigma=n-1+\sigma(-k),\qquad \qquad  j^\sigma=j-{p(\sigma)}-2+\sigma(k)+\sigma(-k),\\
&\ell_1^\sigma=\ell_1-1+\sigma(-k),\qquad \qquad  \ell_2^\sigma=\ell_2-1+\sigma(k),\end{align*}
in order to rewrite
\[a_{n}(\sigma)=\binom{n^\sigma}{\ell_1^\sigma-n^\sigma}\binom{j^\sigma-n^\sigma}{\ell_2^\sigma-(j^\sigma-n^\sigma)}.\]
Note that the difference between $n^\sigma$ (resp. $j^\sigma$, $\ell_1^\sigma$, $\ell_2^\sigma$) and $n$ (resp. $j$, $\ell_1$, $\ell_2$) is of order $ {k}$.
Define 
\[n_*^\sigma=\frac{j^\sigma\ell_1^\sigma}{\ell_1^\sigma+\ell_2^\sigma},\]
which satisfies 
\begin{equation}
\label{eq:dnn*}
|n_*^\sigma-\ell_1\rho_\ell|\leq {C(k,\delta_2)}.
\end{equation}
Denote by $I_\ell(\sigma)$ the set of $n$'s in $\{0,\dots,j\}$ such that $|n^\sigma-n_*^\sigma|\leq \ell^{3/4}$. 
We can now write by triangular inequality
\begin{equation}
\label{eq:SmfS*}
\abs{S(\sigma)-F(\rho_\ell, \sigma)S^*}\leq\sum_{n\in I_\ell(\sigma)}\abs{\frac{a_n(\sigma)}{a_n^*}-F(\rho_\ell, \sigma)}a_n^*+ \sum_{n\notin I_\ell(\sigma)} {\big(}a_n(\sigma)+F(\rho_\ell, \sigma) a_n^*{\big)}.
\end{equation}
Before carrying on, let us state the following lemma, that we will prove after concluding the proof of Lemma \ref{lem:equiv4}.
\begin{lemma}
\label{lem:estan}
There exists a constant $C_7:={C_7(k,\delta_1,\delta_2)}$ such that
\[ \max_{x\in B_{(1-\delta_2)\ell-k}}\max_{n\in I_\ell(\sigma)}\abs{\frac{a_n(\sigma)}{a_n^*}-F(\rho_\ell, \sigma) }\leq \frac{C_7}{\ell^{1/4}}.\]
Moreover, letting $n_0=\lfloor j\ell_1/(\ell_1+\ell_2)\rfloor$, there also exist two constants  $ C_8:=C_8(k,\delta_1,\delta_2)$ and $C_9:=C_9(k,\delta_2)$ such that
\begin{equation}
\label{eq:eq:petitstermes}
\max_{x\in B_{(1-\delta_2)\ell-k}}\max_{n\notin I_\ell(\sigma)}\frac{\max\{a_{n}(\sigma),a^*_{n}\}}{a^*_{n_0}}\leq C_8 e^{-C_9\sqrt{\ell}}.
\end{equation}
\end{lemma}
According to Lemma \ref{lem:estan} above, \eqref{eq:SmfS*} becomes 
\begin{align*}
\abs{S(\sigma)-F(\rho_\ell, \sigma)S^*}\leq&  \sum_{n\in I_\ell(\sigma)}\frac{C_7}{\ell^{1/4}}a_n^*+|I_\ell(\sigma)^c|\big(1+F(\rho_\ell, \sigma)\big)C_8e^{-C_9\sqrt{\ell}}a_{n_0}^*\\
\leq &\pa{\frac{C_7}{\ell^{1/4}}+(2\ell+1) C'_8e^{-C_9\sqrt{\ell}} }S^*,\end{align*}
where we used that $a^*_{n_0}\leq S^*$ and that $(1+F(\rho_\ell, \sigma))C_8\leq {C'_8(k,\delta_1, \delta_2)}$, which holds since by assumption $\rho_\ell\in [(1+\delta_1)/2,1-\delta_1]$. We therefore obtain as wanted that for some constant ${C_{4}=C_{4}(k,  \delta_1, \delta_2)}$
\[\abs{S(\sigma)-F(\rho_\ell, \sigma)S^*}\leq \frac{{C_{4}}S^*}{\ell^{1/4}},\]
which proves Lemma \ref{lem:equiv4}.
\end{proof}
We now prove Lemma \ref{lem:estan}.
\begin{proof}[Proof of Lemma \ref{lem:estan}]
We start by proving the first identity. Shortening $\hat n =j-n$, and extracting the excess terms yield that 
\begin{equation}
\label{eq:binom1}
\frac{a_n(\sigma)}{a_n^*}=\frac{(2n-\ell_1)^{1-\sigma(-k)}}{n^{1-\sigma(-k)}}\frac{(2\hat n-\ell_2)\cdots(2\hat n-2p+\sigma(k)-\ell_2)}{\hat n\cdots(\hat n-p+\sigma(k))(\ell_2+p-\hat n)\cdots(\ell_2+1-\hat n)}. 
\end{equation}
Denote by $x_n=n/\ell_1$ the density in the left cluster, and $g(x_n)=(j- n)/\ell_2=(j- \ell_1x_n)/\ell_2$, which is close to the density in the right cluster. 
We start by proving that there exists a constant ${C_{10}(k,  \delta_1, \delta_2)}$ such that 
\begin{equation}
\label{eq:densboites}
\max_{n\in I_\ell(\sigma)}\big\{\abs{x_n-\rho_\ell}+\abs{g(x_n)-\rho_\ell}\big\}\leq \frac{{C_{10}}}{\ell^{1/4}}.
\end{equation} 
To prove this bound, one merely has to split by the triangular inequality
\[|x_n-\rho_\ell|\leq \frac{\abs{n-n^\sigma}}{\ell_1}+\frac{\abs{n^\sigma-n_*^\sigma}}{\ell_1}+\abs{\frac{n_*^\sigma}{\ell_1}-\rho_\ell}.\]
Since $n^\sigma=n-1+\sigma(-k)$, and thanks to \eqref{eq:dnn*} the first and third terms in the right hand side above are less than ${C(k, \delta_2)}/\ell$ because $\ell_1\geq \delta_2\ell$. The second term is also less than $\ell^{-1/4}/\delta_2$ by definition of $I_\ell(\sigma)$. 
The estimation of $\abs{g(x_n)-\rho_\ell}$ is obtained analogously, thus proving \eqref{eq:densboites}.

Since $\ell_1, \;\ell_2\geq \delta_2\ell$, \eqref{eq:binom1} and \eqref{eq:densboites} yield after elementary computations that 
\[\max_{n\in I_\ell(\sigma)}\abs{\frac{a_n(\sigma)}{a_n^*}-F(\rho_{\ell}, \sigma)}\leq \frac{{C_7(k,\delta_1, \delta_2)}}{\ell^{1/4}},\]
where $F$ is the function defined in \eqref{eq:deff}, as wanted.

\bigskip

Before proving the second identity \eqref{eq:eq:petitstermes}, we state that $a_n$ is relatively small as soon as the density $n/\ell_1$ in the first cluster is far away from the expected density $\rho_\ell$.
\begin{lemma}
\label{lem:concentration2}
For any $\delta_2>0$, there exists a constant $K_1:=K_1(\delta_2)$ such that for any $\ell\geq 1$,  $\delta_2\ell\leq \ell_1, \ell_2 \leq 2\ell+1$, any $ \ell+1 \leq j\leq 2\ell+1$,
and any integer $n$ such that 
\[\frac{\ell_1}{2}  < \;n \;\leq \; \ell_1, \quad \frac{\ell_2}{2}\;<\;j-n\;\leq \;\ell_2,\]
we have 
\begin{equation}
\label{eq:concentration1}
\binom{n}{\ell_1-n}\binom{j-n}{\ell_2-(j-n)}\leq {K}_* \exp\pa{-K_1\frac{(n-n_*)^2}{\ell}},
\end{equation}
where  $n_*=j\ell_1/(\ell_1+\ell_2)$, and $K_*$ is a large constant depending only on $j$, $\ell_1$, and $\ell_2$. Note that $n_*$ is not a priori an integer. 
Furthermore, there exists a constant $K_2:=K_2({\delta_1,}\delta_2)$ such that, letting $n_0=\lfloor n_*\rfloor$
\begin{equation}
\label{eq:concentration2}
 \binom{n_0}{\ell_1-n_0}\binom{j-n_0}{\ell_2-(j-n_0)}\geq \frac{K_*}{2^8\ell^2}\exp\pa{-\frac{K_2}{\ell}}.
\end{equation}

\end{lemma}
Before proving this lemma, we conclude the proof of \eqref{eq:eq:petitstermes} and Lemma \ref{lem:estan}. Recall that we introduced $n_0=\lfloor j\ell_1/(\ell_1+\ell_2)\rfloor$ the approximate index of the largest term in $S^*$. We are going to prove that for any $n\notin I_\ell(\sigma)$, we have 
\[\frac{\max\{a_n(\sigma),a_n^*\}}{a_{n_0}^*}\leq C_8\exp(-C_9\sqrt{\ell}).\]
This bound is a direct consequence of Lemma \ref{lem:concentration2}, since according to \eqref{eq:concentration1}, we have \footnote{The dependency of $K_1(\delta_2/2)$ comes from the crude bound $\ell^\sigma_1, \ell^\sigma_2\geq \delta_2\ell/2$.
}
\[a_n(\sigma)\leq  K_*(\sigma) \exp\pa{-K_1(\delta_2/2)\frac{(n^{\sigma}-n^{\sigma}_*)^2}{\ell}}\]
and
\[a_n^*\leq  K_* \exp\pa{-K_1(\delta_2)\frac{(n-n_*)^2}{\ell}},\]
and \eqref{eq:concentration2} yields
\[a_{n_0}^*\geq  \frac{K_*}{2^8\ell^2} \exp\pa{-\frac{K_2({\delta_1,}\delta_2)}{\ell}}.\]
According to the proof of Lemma \ref{lem:concentration2} the constants $K_*(\sigma)$ and $K_*$ are respectively given by 
\[K_*(\sigma)=\exp\big(-(\ell_1^\sigma+\ell_2^\sigma)H(x_*^\sigma)\big)\quad \mbox{ and }K_*=\exp\big(-(\ell_1+\ell_2)H(x_*)\big)\]
with $x_*^\sigma=j^\sigma/(\ell_1^\sigma+\ell_2^\sigma)$ and $x_*=j/(\ell_1+\ell_2)$, and the entropy functional $H$ is given below by \eqref{eq:DefH}. In particular, $x_*^\sigma-x_*^\sigma=\mc O_{\ell}(1/\ell)$, and both are bounded away from $\frac12$ and $1$ by some constant depending only on $k$ and $\delta_1 $, therefore  
$|H(x_*^\sigma)-H(x_*^\sigma)|\leq C(k,\delta_1)/\ell$, and since $|\ell_1^\sigma+\ell_2^\sigma- (\ell_1+\ell_2)| \leq 2$, we finally obtain 
$K_*(\sigma)/K_*\leq C(k,\delta_1)$. It is easily checked that for any $n \notin I_\ell(\sigma)$, we have $|n-n_*|{\wedge|n^\sigma-n^\sigma_*|}\geq {c}\ell^{3/4}$ for some constant ${c(k)>0}$ depending only on $k$. We finally obtain 
\begin{multline*}
\frac{\max\{a_n(\sigma),a_n^*\}}{a_{n_0}^*}\\
{\leq 2^8\ell^2 \max\{1, C(k,\delta_1)\}\exp\pa{\frac{K_2({\delta_1,}\delta_2)}{\ell}-\min\{K_1(\delta_2),K_1(\delta_2/2)\}{c^2}\sqrt{\ell}}},
\end{multline*}
which is what we wanted to prove.
\end{proof}

We now prove the technical concentration lemma we just used.
\begin{proof}[Proof of Lemma \ref{lem:concentration2}]
Sondow's bound \cite{Sondow}, which is a quantitative extension of Stirling's formula, can be written for any integer $0\leq k\leq \ell$ as 
\[\frac{1}{4\max(k, \ell-k)}\frac{\ell^\ell}{k^k(\ell-k)^{\ell-k}}\leq \binom{\ell}{k}\leq \frac{\ell^\ell}{k^k(\ell-k)^{\ell-k}},\]
which immediately yields
\[\frac{1}{4\ell}\frac{\ell^{\ell}}{k^k(\ell-k)^{\ell-k}}\leq \binom{\ell}{k}\leq \frac{\ell^\ell}{k^k(\ell-k)^{\ell-k}},\]
Using both parts of this bound yields 
\begin{multline*} \frac{1}{16j^2}\frac{n^n}{(\ell_1-n)^{\ell_1-n}(2n-\ell_1)^{2n-\ell_1}}\frac{(j- n)^{j- n}}{(\ell_2-(j-n))^{\ell_2-(j-n)}(2(j-n)-\ell_2)^{2(j-n)-\ell_2}}\\
\leq \binom{n}{\ell_1-n}\binom{j-n}{\ell_2-(j-n)} \leq \\
 \frac{n^n}{(\ell_1-n)^{\ell_1-n}(2n-\ell_1)^{2n-\ell_1}}\frac{(j- n)^{j- n}}{(\ell_2-(j-n))^{\ell_2-(j-n)}(2(j-n)-\ell_2)^{2(j-n)-\ell_2}}.
\end{multline*}
Denoting $x_n=n/\ell_1$, and $g(x_n)=(j-n)/\ell_2=(j-\ell_1x_n)/\ell_2$,
and introducing the entropy functional  
\begin{equation}
\label{eq:DefH}
H(x)=(1-x)\log\Big(\frac{1-x}{x}\Big)+(2x-1)\log\Big(\frac{2x-1}{x}\Big), 
\end{equation}
the previous bound rewrites
\begin{multline}
 \label{eq:entropy1Phi}
\frac{1}{16j^2}\exp\big(-\ell_1H(x_n)-\ell_2H(g(x_n))\big)\\
\leq \binom{n}{\ell_1-n}\binom{j-n}{\ell_2-(j-n)} \leq \\
\exp\big(-\ell_1H(x_n)-\ell_2H(g(x_n))\big) .
\end{multline}
Let us shorten $\Phi=\ell_1H+\ell_2H\circ g$. We now take a look at the quantity inside the exponential. Since $g'=-\ell_1/\ell_2$,
\[\Phi'(u)=\ell_1\big(H'(u)-H'(g(u))\big)=\ell_1\int_{g(u)}^{u}dvH''(v).\]
Denote by
\[x_*=\frac{j}{\ell_1+\ell_2}=\frac{n_*}{\ell_1}\]
the only fixed point for $g$, so that 
\[\Phi(x)=\Phi(x_*)+\ell_1\int_{x_*}^{x}du\int_{g(u)}^{u}dvH''(v),\]
where \[H''(u)=\frac{4}{2u-1}-\frac{1}{u}+\frac{1}{1-u}\]
is positive and  bounded from below on $(\frac12,1)$ by some fixed constant $c_0$, and uniformly bounded from above on any segment $ [a,b]\subset(\frac12,1)$. Define 
\[C_H(x)=\max\big\{H''(u), \;u\in [x,x_*] \mbox{ or }[x_*,x]\big\}.\]
Since $u-g(u)=(\ell_1+\ell_2)(u-x_*)/\ell_2$,  it is then a matter of simple verification to check that 
\[\Phi(x_*)+c_0\frac{\ell_1(\ell_1+\ell_2)(x-x_*)^2}{2\ell_2}\leq \Phi(x)\leq \Phi(x_*)+C_H(x)\frac{\ell_1(\ell_1+\ell_2)(x-x_*)^2}{2\ell_2},\]
which can be rewritten
\[\Phi(x_*)+c_0\frac{(n-n_*)^2}{2\max(\ell_1,\ell_2)}\leq \Phi(x)\leq \Phi(x_*)+C_H(x)\frac{(n-n_*)^2}{\min(\ell_1,\ell_2)}.\]
We inject this bound in \eqref{eq:entropy1Phi}, to finally obtain that for any $[a,b]\subset (\frac12,1)$ and for any  $n$, 
\begin{multline*} 
\frac{K_*}{16j^2}\exp\pa{-C_H(n/\ell_1)\frac{(n-n_*)^2}{\min\{\ell_1, \ell_2\}}}\\
\leq \binom{n}{\ell_1-n}\binom{j-n}{\ell_2-(j-n)}\leq K_* \exp\pa{-c_0\frac{(n-n_*)^2}{2\max\{\ell_1, \ell_2\}}},
\end{multline*}
where we defined $K_*=\exp(-\Phi(x_*))$. Letting $n_0=\lfloor n_*\rfloor$, and assuming that $\ell_1\geq \delta_2 \ell$, one can verify that $C_H(n_0/\ell_1)\leq C({\delta_1,}\delta_2)$, so that the left part of the last estimation rewrites  
\begin{equation*} 
 \binom{n_0}{\ell_1-n_0}\binom{j-n_0}{\ell_2-(j-n_0)}\geq \frac{K_*}{16j^2}\exp\pa{-\frac{\delta_2C({\delta_1,}\delta_2)}{\ell}},
\end{equation*}
thus concluding the proof.
\end{proof}

To conclude this section, we now give the proof of Lemma \ref{lem:rhogrand}, which we postponed early in the section, and prove that the canonical measure $\hat{\pi}_{\rho,1,1}^{\ell,j}({\varsigma}_{\mid B_k(x)}=\sigma)$ is of order $ \delta_1$ if $j$ is larger than $(1-\delta_1)(2\ell+1)$.

\begin{proof}[Proof of Lemma \ref{lem:rhogrand}]
We will once again only prove \eqref{eq:rhograndlem} in the case where $\varepsilon_1=\varepsilon_2=1$, the other cases being analogous. To do so, denote by ${\bf 1}_k$ the full configuration on $B_k$, using the same notations as in \eqref{eq:quotjlrho}, rewrite the left hand side of \eqref{eq:rhograndlem} as
 \begin{equation}
\label{eq:quotjlrho2}
\frac{\hat{\pi}_{\rho,1,1}^{\ell,j}\big({\varsigma}_{\mid B_k(x)}=\sigma\big)}{\hat{\pi}_{\rho,1,1}^{\ell,j}\big({\varsigma}_{\mid B_k(x)}= {\bf 1}_{k}\big)}=\frac{\sum_{n=1}^{j}a_n(\sigma)}{\sum_{n=1}^{j}a_n({\bf 1}_{k})}. 
\end{equation}
We are going to show that assuming that the configuration $\sigma$ is not full, there exists a constant $ C(k, \delta_2)$ such that for any $1\leq n\leq j$, 
\begin{equation}
\label{eq:anxian1}
a_n(\sigma)\leq C(k, \delta_2)\delta_1a_n({\bf 1}_{k}).
\end{equation}
We first prove that if
$\hat n=j-n< \big(1-4\delta_1/\delta_2\big)\ell_2$
the two sides above vanish, because there can be no associated configuration. Indeed, by assumption, 
\[j\geq (1-\delta_1)(2\ell+1).\]
Furthermore, we can safely assume that $n\leq \ell_1$ since if not, we are trying to put more particles than there are holes in the first box, and the binomial coefficients obviously vanish. Therefore, 
\[\hat n\geq (1-\delta_1)(2\ell+1)-\ell_1=\ell_2{+}(2k+1)-\delta_1(2\ell+1)\geq \ell_2\Big(1-\frac{4\delta_1}{\delta_2}\Big) \]
where we used that by assumption, $\ell_2\geq \delta_2\ell\geq \delta_2(2\ell+1)/4$. 
In particular, for any $n$ such that $\hat n /\ell_2< 1-{4\delta_1/\delta_2}$, \eqref{eq:anxian1} obviously holds since both sides vanish. 

\medskip

We now assume that $\hat n/\ell_2 \geq 1-{4\delta_1/\delta_2}$.
To prove \eqref{eq:anxian1} in this case, by extracting the excess terms, as we did in \eqref{eq:binom1}, we rewrite
\begin{multline*}
a_n(\sigma)=a_n({\bf 1}_{k})\frac{(2n-\ell_1)^{1-\sigma(-k)}}{n^{1-\sigma(-k)}}\\
\times\frac{(\hat n-p(\sigma)+\sigma(k)-1)\cdots(\hat n-2k)(\ell_2+2k+1-\hat n)\cdots(\ell_2+{p(\sigma)}+1-\hat n)}{(2\hat n-2p(\sigma)+\sigma(k)-\ell_2-1)\cdots(2\hat n-4k+1-\ell_2)}.
\end{multline*}
The first part of the quotient is less than one since $(2n-\ell_1)/n\leq 1$. Divide each term in the second part of the right hand side by $\ell_2$, underestimate the powers in the numerator and use that $x_n\leq 1$ to obtain that
\[a_n(\sigma)\leq a_n({\bf 1}_{k})\frac{x_n(1-x_n+\varepsilon_{\ell})}{(2x_n-1{-}\varepsilon_\ell)^{4k-2p(\sigma)}}, \]
where $\varepsilon_{\ell}=4k/\ell_2$, and $x_n=\hat n/\ell_2\geq 1-{4\delta_1/\delta_2}.$
The configuration $\sigma$ contains at least one empty site, so that $2k+1\geq p+ 1$. We now use that  $x_n\geq  1-C\delta_1$, therefore $\varepsilon_\ell$ is small w.r.t.\@ $x_n$ and $2x_n-1$. Recall that $\ell$ goes to infinity before $\delta_1$ goes to $0$. We obtain as wanted
\[a_n(\sigma)\leq a_n({\bf 1}_{k})\big(o_\ell(1)+C(k, \delta_2)\delta_1\big), \]
which concludes the proof.

\end{proof}

\section{Proof of the Replacement Lemma \ref{lem:RL}}
\label{app:replacement}
\subsection{Strategy of the proof}

Recall that $\tilde\mu_N$ has been defined in \eqref{eq:Deftm}. 
Let ${\widetilde{\mu}^N_s:=\widetilde{\mu}_Ne^{sN^2\mathcal{L}_N}}$ be the measure at macroscopic time $s$ of the process started from $\tm$.
{Fix $\rho\in (\frac12,1),$ and l}et us introduce, for any $s\geq 0$
\[f_s^N=\frac{d\widetilde\mu^N_s}{d\nu_{\rho,N}}, \quad \text{ and } \quad \overline{f}_T^N= \frac1T\int_0^T f_s^N ds{,}\]
{where $\nu_{\rho,N}$ is the translation invariant, periodic grand canonical measure defined in Definition~\ref{def:nurhoN}}.
This density $f_s^N$ is well defined: $\nu_{\rho,N}(\eta)=0$ \ccl{only} if $\eta\notin\mathcal{E}_N$, and $\ccl{\tilde{\mu}^N_s}$ has support in $\mathcal{E}_N$ by definition.

Denote $\mc D_N$ the Dirichlet form, defined as
\begin{align*}
\mc D_N(g):&=-\nu_{\rho,N}\big(\sqrt{g} \mathcal{L}_N \sqrt{g}\big)\\
&=\frac{1}{2}\nu_{\rho,N}\bigg(\sum_{x\in\T_N}{\eta(x+2)\eta(x-1)}\big(\sqrt g(\eta^{x,x+1})-\sqrt g(\eta)\big)^2\bigg).\end{align*}
It is straightforward to prove that $\mc D_N$ thus defined is non-negative and convex. To obtain the second identity, we used that $\nu_{\rho,N}$--a.s., any configuration such that $\eta(x)\neq \eta(x+1)$ and such that both $\eta$ and $\eta^{x,x+1}$ are ergodic, we must have $\eta(x+2)=\eta(x-1)= 1$, because two empty sites cannot be neighbors under $\nu_{\rho,N}$.
In Proposition \ref{prop:DirEstimate} below we will prove that the Dirichlet form $\mathcal{D}_N$ evaluated at $\overline{f}_T^N$ satisfies: 
\begin{equation}\mathcal{D}_N(\overline{f}_T^N) \leqslant \frac{C}{N}.\label{eq:diri}\end{equation}
Let us turn now to the proof of Lemma \ref{lem:RL}. We rewrite the expression under the limit in \eqref{eq:repl} as: 
\[\nu_{\rho,N}\bigg(\overline{f}_T^N \times \frac{T}{N}\sum_{x\in\T_N} \tau_x \big| V_{\varepsilon N}(h,\cdot) \big|\bigg).\]
Therefore, in order to prove \eqref{eq:repl} it is enough to prove the following: 
\begin{equation}
\label{eq:repl2}
\limsup_{\varepsilon\to 0}\limsup_{N\to\infty} \sup_{f \; : \; \mathcal{D}_N(f)\leqslant \frac{C}{N}} \nu_{\rho,N}\bigg(f\times \frac{1}{N}\sum_{x\in\T_N}\tau_x\big|V_{\varepsilon N}(h,\cdot)\big|\bigg) = 0,
\end{equation}
where the supremum is taken over all densities $f$ w.r.t.~ $\nu_{\rho,N}$ which satisfy $\mathcal{D}_N(f)\leqslant \frac{C}{N}$. 
Take $\ell \in \mathbb{N}$. 
Inside the absolute values in \eqref{eq:repl2} we {add} and subtract 
\[\frac{1}{2\varepsilon N+1} \sum_{|y|\leqslant \varepsilon N} \bigg(\frac{1}{2{\ell'} +1} \sum_{|z-y|\leqslant {\ell'}} \tau_z h(\eta)  - \pi_{\rho^\ell(y)(\eta)}(h)\bigg),\]
{where we shortened $\ell'=\ell-1$, to make the sum measurable w.r.t. the sites in $\tau_yB_{\ell}.$}
Thanks to the triangular inequality, we are reduced to estimate three terms:
\begin{align}
I_{\ell,\varepsilon N}^1 & := \nu_{\rho,N}\bigg( f\times \frac{1}{N}\sum_{x\in\T_N} \tau_x\bigg| \frac{1}{2\varepsilon N+1} \sum_{|y|\leqslant\varepsilon N} \Big(\tau_y h - \frac{1}{2{\ell'}+1} \sum_{|z-y|\leqslant {\ell'}} \tau_z h\Big)\bigg|\bigg), \notag \\
I_{\ell,\varepsilon N}^2 & := \nu_{\rho,N}\bigg( f\times \frac{1}{N}\sum_{x\in\T_N} \tau_x\bigg|\frac{1}{2{\ell'}+1} \sum_{|y|\leqslant{\ell'}} \tau_y h-\pi_{\rho^\ell(0)}(h)\bigg|\bigg), \label{eq:secondone}\\
I_{\ell,\varepsilon N}^3 & := \nu_{\rho,N}\bigg( f\times \frac{1}{N}\sum_{x\in\T_N} \tau_x\bigg|\frac{1}{2\varepsilon N+1} \sum_{|y|\leqslant\varepsilon N} \pi_{\rho^{\ell}(y)}(h)-\pi_{\rho^{\varepsilon N}(0)}(h)\bigg|\bigg).\label{eq:thirdone}
\end{align}
We are going to show that each one of these terms vanishes, as first $N\to \infty$ then $\varepsilon \to 0$ and then $\ell \to \infty$. 
We treat them separately: the first one can be easily bounded as follows~; there exists $C>0$ such that 
\[I_{\ell,\varepsilon N}^1 \leqslant \frac{C\ell}{\varepsilon N} \nu_{\rho,N}\bigg( f\times\frac{1}{N} \sum_{x\in\T_N} \tau_x h \bigg)\leqslant \frac{C\ell}{\varepsilon N} \xrightarrow[N\to\infty]{} 0,\]
since $h$ is bounded by $1$.
{Let us slightly change our notation, and define 
\[\widetilde{V}_\ell(h,\eta)=\frac{1}{2{\ell'}+1} \sum_{|y|\leqslant{\ell'}} \tau_y h(\eta)-\pi_{\rho^\ell(0)(\eta)}(h),\]
whose only difference with $V_\ell(h,\eta)$ is that the first average is taken on a slightly reduced box to make the sum measurable w.r.t. the sites in $B_\ell$. Then, the second term \eqref{eq:secondone} rewrites with this notation as}
\[I_{\ell,\varepsilon N}^2 = \nu_{\rho,N}\bigg( f\times \frac{1}{N} \sum_{x\in\T_N} \tau_x\big| \widetilde{V}_\ell(h,\cdot)\big|\bigg),\]
and it vanishes thanks to the \emph{one-block estimate} stated and proved in Lemma \ref{lem:one-block} below. We finally bound the third one \eqref{eq:thirdone} as 
\begin{align}
I_{\ell,\varepsilon N}^3 & \leqslant \sup_{|y|\leqslant \varepsilon N}\nu_{\rho,N}\bigg( f\times\frac{1}{N}\sum_{x\in\T_N} \tau_x \Big| \pi_{\rho^{\varepsilon N}(0)}(h) - \pi_{\rho^\ell(y)}(h)\Big| \bigg)\notag\\
& \leqslant {4} \sup_{|y|\leqslant \varepsilon N}\nu_{\rho,N}\bigg( f\times\frac{1}{N}\sum_{x\in\T_N}  \Big|\rho^{\varepsilon N}(x) - \rho^\ell(x+y)\Big| \bigg), \label{eq:repl3}
\end{align}
the last inequality coming from the fact that $|\pi_{\rho_\star}(h)-\pi_\rho(h)|\leqslant {4} |\rho_\star - \rho|$ for any $\rho,\rho_\star \in {[0,1]}$ (after using the identity \eqref{eq:idh}).
The left hand side of \eqref{eq:repl3} vanishes according to the \emph{two-blocks estimate} stated and proved in Lemma \ref{lem:two-blocks} below.

\subsection{One-block estimate}

\begin{lemma}
\label{lem:one-block}
Fix $\rho\in(\frac12,1]$. For any finite constant $C>0$,
\begin{equation}
\label{eq:oneb}
\limsup_{\ell \to\infty}\limsup_{N\to\infty} \sup_{f \; : \; \mathcal{D}_N(f)\leqslant \frac{C}{N}} \nu_{\rho,N}\bigg(f\times \frac{1}{N} \sum_{x\in\T_N} \tau_x\big| {\widetilde{V}}_\ell(h,\cdot)\big|\bigg)=0.
\end{equation}
\end{lemma} 

\begin{proof} Following \cite[Chapter 5, Section 4]{KL}, we divide the proof into several steps listed below. 

\paragraph{\sc Step 1 -- Reduction to microscopic blocks }

Since the measure $\nu_{\rho,N}$ is translation invariant, we can rewrite the {expectation} that appears in \eqref{eq:oneb} as
\[\nu_{\rho,N}\Big( \widetilde{f}\;  \big|{\tilde V}_\ell(h,\cdot)\big|\Big), \qquad \text{where } \widetilde f:=\frac{1}{N}\sum_{x\in\T_N}\tau_xf.\]
Note that ${\widetilde{V}}_\ell(h,\eta)$ depends on $\eta$ only through its values on the box $B_\ell=\{-\ell,\dots,\ell\}\subset \T_N$. 

We now define the conditional expectation of $\widetilde{f}$ with respect to the coordinates on $B_\ell$, namely: 

\[\widetilde f_\ell(\sigma)={\left\{
\begin{array}{ll} \displaystyle
\frac{1}{{\hat\nu_{\rho,N}^\ell}(\sigma)}\sum_{\substack{\eta \in\{0,1\}^{\T_N} \\ \eta_{|B_\ell}=\sigma}}\widetilde f(\eta) \nu_{\rho,N}(\eta)&\text{for any } \sigma \in\ERGt_{B_\ell},\\
0& \text{else.}
\end{array}\right.}\]
\begin{lemma}\label{lem:density}
The function
$\widetilde f_\ell$ is a density with respect to ${\hat\nu_{\rho,N}^\ell}$, and
 \begin{equation}\label{eq:infinitenormdensity}
\sup_{\substack{f \text{density w.r.t.\ } \hat\nu_{\rho,N}^\ell,\\ f_{|\ERGt_{B_\ell}^c}\equiv0}}\|f\|_\infty\leq C(\ell,\rho). 
\end{equation}
 
\end{lemma}

\begin{proof}[Proof of Lemma~\ref{lem:density}]
The first point is clear. For the second, notice that if $f$ is a density w.r.t.\ $\hat\nu_{\rho,N}^\ell$, then
$
\sum_{\sigma\in\ERGt_{B_\ell}}\hat\nu_{\rho,N}^\ell(\sigma)f(\sigma)=1,
$
and therefore $f(\sigma)\leq \hat\nu_{\rho,N}^\ell(\sigma)^{-1}$ for all $\sigma\in\ERGt_{B_\ell}$. By formula \eqref{eq:hatpiell} and estimate \eqref{eq:localconvfinitebox}, for any $\sigma\in\ERGt_{B_\ell}$, we have \[\hat\nu_{\rho,N}^\ell(\sigma)\geq \kappa\min(\alpha^{2\ell+1},1)\beta^{2\ell+1}-\frac{C(\ell,\rho)}{\sqrt{N}}\geq C(\ell,\rho)>0\] for large enough $N$, and Lemma \ref{lem:density} is proved.
\end{proof}
Going back to Lemma \ref{lem:one-block}, we finally rewrite the {expectation} in \eqref{eq:oneb} as 
\begin{equation}
\label{eq:oneb2}
{\hat\nu_{\rho,N}^\ell}\Big(\widetilde f_\ell\; \big|{\widetilde{V}}_\ell(h,\cdot)\big| \Big).
\end{equation}
 As a consequence of Lemma \ref{lem:aux2} and Lemma~\ref{lem:density}, we can replace in the large $N$ limit ${\hat\nu_{\rho,N}^\ell}$ with $\hat\pi_{\rho}^{\ell}$, 
and to prove \eqref{eq:oneb} we are reduced to prove that
\begin{equation}
\label{eq:oneb3}
\limsup_{\ell \to\infty}\limsup_{N\to\infty} \sup_{f \; : \; \mathcal{D}_N(f)\leqslant \frac{C}{N}}\hat\pi_{\rho}^{\ell}\Big(\widetilde f_\ell \; \big|{\widetilde{V}}_\ell(h,\cdot)\big|\Big)= 0.
\end{equation}

\paragraph{\sc Step 2 -- Estimates on the Dirichlet form of $\widetilde f_\ell$ }

Let us first introduce some notations: for any  $y \in\T_N$, and positive function $g$, define
\begin{equation}
\label{eq:DefIy}
I_{y}(g,\eta)=\frac{1}{2} {\eta(y+2)\eta(y-1)} \big(\sqrt g(\eta^{y,y+1})-\sqrt g(\eta)\big)^2,
\end{equation}
and consider two Dirichlet forms restricted to bonds in $B_\ell$ defined for all densities ${f}:\{0,1\}^{B_\ell}\to \mathbb{R}_+$ w.r.t.\ $\hat\nu_{\rho,N}^\ell$ which vanish on $\ERGt_{B_\ell}^c$ by 
\begin{align*}
{\mathcal{D}^{\ell}_N(f)} &: =\sum_{y \; : \; \{y,y+1\} \subset B_\ell} {\hat\nu_{\rho,N}^{\ell+1}}\big(I_y({f},\cdot)\big) \\
\mathcal{D}^\ell({f}) &: =\sum_{y \; : \; \{y,y+1\} \subset B_\ell} \hat\pi_{\rho}^{\ell+1}\big( I_y({f},\cdot)\big).
\end{align*}
From Lemma \ref{lem:aux2} and Lemma~\ref{lem:density}, these two Dirichlet forms are pretty close as $N\to\infty$. Recalling the definition \eqref{eq:diri} we also have, for any density $f:\{0,1\}^{\T_N}\to\mathbb{R}_+$,
\[\mathcal{D}_N(f)=\sum_{x\in\T_N} \nu_{\rho,N}\big(I_x(f,\cdot)\big).\]
The exact same proof as in \cite[Step 3 p.86]{KL} shows that
\begin{equation}\label{eq:claim}\text{if } \quad \mathcal{D}_N(f) \leqslant \frac{C}{N} \quad \text{ then } \quad  {\mathcal D^{\ell}_N}(\widetilde f_\ell) \leqslant \frac{C(\ell)}{N^2},\end{equation}
where $C(\ell)$ is a positive constant. From Lemma \ref{lem:aux2}, this implies
\[\mathcal{D}^\ell(\widetilde f_\ell) \leqslant \frac{C'(\ell)}{\sqrt N},\] where $C'(\ell)$ is another positive constant.
For the sake of completeness we expose quickly the arguments which prove \eqref{eq:claim}:
\begin{itemize}
\item first, by the property of the marginal distribution, we have: 
\[{\hat\nu_{\rho,N}^{\ell+1}}\big( I_y(\widetilde f_\ell,\cdot)\big) {=} \nu_{\rho,N}\big( I_y(\widetilde f_\ell,\cdot)\big), \]
for any $y$  such that ${\{y,y+1\} \subset B_\ell,}$ which implies
\[{\mathcal{D}^{\ell}_N}(\widetilde f_\ell) {=} \sum_{y \; : \;{ \{y,y+1\} }\subset B_\ell} \nu_{\rho,N}\big(I_y(\widetilde f_\ell,\cdot)\big)\; ; \]
\item second, since $\widetilde f$ and $\nu_{\rho,N}$ are translation invariant, 
\[\sum_{y \; : \; {\{y,y+1\} \subset B_\ell}} \nu_{\rho,N}\big(I_y(\widetilde f_\ell,\cdot)\big) = \frac{2\ell}{N}\mathcal{D}_N(\widetilde f_\ell)\; ;\]
\item finally, by convexity of the Dirichlet form,
$\mathcal{D}_N(\widetilde f_\ell) \leqslant \mathcal D_N(f)$.

\end{itemize}
Putting all these statements together proves the claim \eqref{eq:claim}. 
 As a result, to prove \eqref{eq:oneb3} we are reduced to prove that
\begin{equation}
\label{eq:oneb4}
\limsup_{\ell \to\infty}\limsup_{N\to\infty} \sup_{f \in {\mathcal A}_{\ell,N}} \hat\pi_{\rho}^{\ell}\Big(f\; \big|{\widetilde{V}}_\ell(h,\cdot)\big|  \Big)= 0,
\end{equation}
where we denoted 
\[{\mathcal A}_{\ell,N}=\left\{f: \hat{\mathcal{E}}_{B_{\ell}} \to \R_+ \ : \  \hat\nu_{\rho,N}^{\ell}(f)=1, \quad \mathcal{D}^\ell(f)\leq\frac{C'(\ell)}{\sqrt N} \right\}.\]
Note that for any $f \in {\mathcal A}_{\ell,N},$ we have $\|f\|_{\infty} \le C(\ell,\rho)$, where \[C(\ell,\rho)^{-1}=\liminf_{N \to \infty}\inf_{\sigma \in\hat{\mathcal{E}}_{B_{\ell}}}  \hat\nu_{\rho,N}^{\ell}(\sigma) >0.\]

\paragraph{\sc Step 3 -- Limit as $N\to\infty$. } This step is exactly identical to \cite[Chapter 5, Section 4, Step 4]{KL}, and mainly relies on the lower semicontinuity of the Dirichlet form and the continuity of $\hat\pi_{\rho}^{\ell}(f)$ with respect to $f$.
Letting
\[{\mathcal A}^*_{\ell}=\left\{f: \hat{\mathcal{E}}_{B_{\ell}} \to\R_+ \ : \ \hat\pi_{\rho}^{\ell}(f)=1, \quad \mathcal{D}^\ell(f)=0 \right\},\]
thanks to Lemma \ref{lem:aux2}, this reduces the proof of \eqref{eq:oneb4} to the proof of
\begin{equation}
\label{eq:oneb5}
\limsup_{\ell \to\infty} \sup_{f\in{\mathcal A}^*_{\ell}} \hat\pi_{\rho}^{\ell}\Big( f\; \big|{\widetilde{V}}_\ell(h,\cdot)\big|  \Big)= 0.
\end{equation}

\paragraph{\sc Step 4 -- Decomposition along hyperplanes with a fixed number of particles. } Note that a density which has a vanishing Dirichlet form is constant on each hyperplane with a fixed total number of particles (recall Lemma~\ref{lem:irreducible}). 
Therefore, the proof of \eqref{eq:oneb5} reduces to the proof of
\begin{equation}
\label{eq:oneb6}
\limsup_{\ell \to\infty} \sup_{j \in \{{\ell},\dots,2\ell+1\}} \hat\pi_{\rho}^{\ell,j}\Big( \big|{\widetilde{V}}_\ell(h,\cdot)\big|  \Big)= 0.
\end{equation}

\paragraph{\sc Step 5 -- An application of the equivalence of ensembles. }

Recall that we defined ${\ell'}=\ell-1$, and note that
\begin{equation}\label{eq:lastin}\hat\pi_{\rho}^{\ell,j}\Big(\big|{\widetilde{V}}_\ell(h,\cdot)\big|   \Big)=  \hat\pi_{\rho}^{\ell,j}\bigg( \bigg| \frac{1}{2{\ell'}+1} \sum_{{x\in B_{\ell'}}} \tau_xh - \pi_{j/(2\ell+1)}(h)\bigg|\bigg).
\end{equation}
Let us fix $k \in \bb N$ and decompose $B_{\ell'}$ into boxes of length $2k+1$, as follows: {first, fix some $\delta>0$, and  define $q=\lfloor {\ell'}(1-\delta)/(2k+1)\rfloor$}{. We partition}
\[ B_{\ell'} = { \mathcal{A}_\ell\sqcup \bigsqcup_{i={{-q}}}^q \mathcal{B}_k(i)}, \qquad \text{where } \mathcal{B}_k(i):=\Big\{(2k+1)i-k,\dots, (2k+1)i{+k}\Big\},\] so that $\mathcal{B}_k(i)$ is the box of size $2k+1$ centered in the site $(2k+1)i$ {and} $\mathcal{A}_\ell$ its complementary set in $B_{\ell'}$, namely 
\[ \mathcal{A}_\ell = B_{\ell'} \setminus {\bigsqcup_{i=-q}^q} \mathcal{B}_k(i){=\{-\ell',\dots,-(2k+1)q-k-1\}\cup\{(2k+1)q+k+1,\dots,\ell'\}}.\]
By construction, the size of $\mathcal{A}_\ell$ is \[|\mathcal{A}_\ell|=2(\ell'-(2k+1)q-k){\leq 2(\delta\ell'+k+1)\leq 3\delta\ell'}\] for $\ell$ large enough. 
Then, in \eqref{eq:lastin}, the total contribution of the sites $x\in \mathcal{A}_\ell$ is less than ${3}\|h\|_{\infty}\delta$, so that \eqref{eq:lastin} is bounded from above by 
\begin{equation}
\label{eq:lastin2}
{{3}\delta\|h\|_{\infty}+}\frac{(2k+1)}{(2{\ell'}+1)} \sum_{i={-q}}^q  \hat\pi_{\rho}^{\ell,j}\bigg( \bigg| \frac{1}{2k+1} \sum_{x \in \mc B_k(i)} \tau_xh-\pi_{j/(2\ell+1)}(h)\bigg| \bigg).
\end{equation}
Recall that $\delta$ and $k$ are fixed, and that for $i\in \{-q,\dots,q\}$, we have $(2k+1)i\in B_{(1-\delta)\ell}$. We can therefore apply Corollary \ref{cor:ensembles} to the translations $\tau_{(2k+1)i}f$, of the local function 
\[f =\bigg|\frac{1}{2k+1} \sum_{x \in B_k} \tau_xh-\pi_{j/(2\ell+1)}(h)\bigg|,\] 
for any fixed $\delta$, $k$: uniformly in $i\in  \{-q,\dots,q\}$, the {expectation} converges  as $\ell\to \infty$ towards 
\[ \pi_{j/(2\ell+1)}\bigg(\bigg| \frac{1}{2k+1} \sum_{x \in \mc B_k(i)} \tau_xh-\pi_{j/(2\ell+1)}(h)\bigg|\bigg).\]
Therefore, \eqref{eq:lastin2} can be rewritten 
\begin{multline}
\label{eq:lastin3}
o_\ell(1)+{{3}\delta\|h\|_{\infty}+}\frac{(2k+1)}{(2{\ell'}+1)} \sum_{i={-q}}^q \pi_{j/(2\ell+1)}\bigg( \bigg| \frac{1}{2k+1} \sum_{x \in \mc B_k(i)} \tau_xh(\sigma)-\pi_{j/(2\ell+1)}(h)\bigg| \bigg) \\
\underset{\ell\to\infty}{\leq} {3}\delta\|h\|_{\infty}+ \sup_{\alpha\in{[}\frac12,1]}\pi_{\alpha}\bigg(\bigg| \frac{1}{2k+1} \sum_{x \in B_k} \tau_xh-\pi_{\alpha}(h)\bigg| \bigg),
\end{multline}
because $\pi_{\alpha}$ is translation invariant and $(2k+1)(2q+1)/(2{\ell'}+1)\leq 1$ for $\ell$ large enough. We now need to eliminate the low densities for which we do not have exponential decay of the correlations for $\pi_{\alpha}$. 

Fix $\alpha_0>\frac12$, in the limit $\ell\to\infty$ then $\delta\to0$ we can therefore bound \eqref{eq:lastin2} by  
\begin{multline}
\label{eq:lastin4}
\sup_{\alpha\in{[}\frac12,\alpha_0]} \pi_{\alpha}\bigg( \bigg| \frac{1}{2k+1} \sum_{x \in B_k} \tau_xh-\pi_{\alpha}(h)\bigg| \bigg)\\
+\sup_{\alpha\in(\alpha_0,1]} \pi_{\alpha}\bigg( \bigg| \frac{1}{2k+1} \sum_{x \in B_k} \tau_xh-\pi_{\alpha}(h)\bigg|  \bigg),
\end{multline}
Since $h$ is a non-negative function, we can crudely bound the first member of \eqref{eq:lastin4} from the triangular inequality by 
\[\sup_{\alpha\in{[}\frac12,\alpha_0]}2\pi_{\alpha}(h)=\sup_{\alpha\in{[}\frac12,\alpha_0]}2\frac{2\alpha-1}{\alpha}\leq 4(2\alpha_0-1),\]
which vanishes as $\alpha_0\to\frac12$ uniformly in $k$.
To estimate the second member of \eqref{eq:lastin4}, we apply the law of large numbers for the grand canonical measure $\pi_\alpha$. The variables $\{{\sigma}(x)\; : \; x \in \bb Z\}$ are not independent. Therefore, we need here to use the exponential decay of correlations proved in Corollary \ref{cor:correlations}, which is uniform on $(\alpha_0,1]$ for any fixed $\alpha_0>\frac12$.  The strong law of large numbers has been proven in this case for instance in \cite{russell}. As a result, for any fixed $\frac12<\alpha_0\leq 1$, the {second line of} \eqref{eq:lastin4} vanishes as $k\to\infty$. 
We then let $\alpha_0\to\frac12 $ to end the proof of Lemma \ref{lem:one-block}.

\end{proof}

\subsection{Two-blocks estimate}

\begin{lemma}
\label{lem:two-blocks} 
Fix $\rho\in(\frac12,1]$. For every finite constant $C>0$, 
\begin{multline}
\label{eq:twob}
\limsup_{\ell \to\infty}\limsup_{\varepsilon \to 0}\limsup_{N\to\infty} \sup_{f \; : \; \mathcal{D}_N(f)\leqslant \frac{C}{N}}\; \sup_{|y|\leqslant \varepsilon N}\\  \nu_{\rho,N}\bigg(f\times \frac{1}{N}\sum_{x\in\T_N}  \Big|\rho^{\varepsilon N}(x)- \rho^\ell(x+y)\Big|\bigg)=0.
\end{multline}
\end{lemma}

\begin{proof}[Proof of Lemma \ref{lem:two-blocks}]
Most of the difficulties to prove the two-blocks estimate have been treated in the one-block estimate, so that we will merely sketch the proof. 
As in the one-block estimate, define $\widetilde{f}=N^{-1}\sum \tau_x f$. To prove Lemma \ref{lem:two-blocks}, it is enough to prove that 
\begin{equation*}
\limsup_{\ell \to\infty}\limsup_{\varepsilon \to 0}\limsup_{N\to\infty} \sup_{f \; : \; \mathcal{D}_N(f)\leqslant \frac{C}{N}}\; \sup_{|y|\leqslant \varepsilon N}\nu_{\rho,N}\Big(\widetilde{f}  \;\Big|\rho^{\varepsilon N}(0) - \rho^\ell(y)\Big|\Big)=0.
\end{equation*}
Similarly to \cite[Section 5.5]{KL}, by triangular inequality, to prove Lemma \ref{lem:two-blocks}, it is enough to prove that 
\begin{equation}
\label{eq:TBE1}
\limsup_{\ell \to\infty}\limsup_{\varepsilon \to 0}\limsup_{N\to\infty} \sup_{f \; : \; \mathcal{D}_N(f)\leqslant \frac{C}{N}}\; \sup_{\sqrt{N} \leqslant |y|\leqslant 2\varepsilon N}\nu_{\rho,N}\Big(\widetilde{f}  \;\Big|\rho^{\ell}(0) - \rho^\ell(y)\Big|\Big)=0.
\end{equation}
Fix $|y|\geq \sqrt{N}$ and  define
\begin{align*}
{\mathcal{D}^{\ell}_{N,y,1}}(h) &: =\sum_{x \; : \; \{x,x+1\} \subset B_\ell\cup\tau_y B_{\ell}} {\nu_{\rho,N}}\big(I_x(h,\cdot) \big) \\
{\mathcal{D}_{N,y,2}}(h) &: =  {\nu_{\rho,N}}\big(J_{0,y}(h,\cdot) \big)\end{align*}
where ${I_x} $ was defined in \eqref{eq:DefIy} and 
\begin{equation*}
J_{0,y}(g,\eta):=\pa{\eta(-1)\eta(0) \eta(1)+\eta(y-1)\eta(y) \eta(y+1)} \big(\sqrt g(\eta^{0,y})-\sqrt g(\eta)\big)^2.
\end{equation*}
This last element allows jumps from $0$ to $y$ and vice-versa, while ensuring that one never leaves the ergodic component. 
We now condition the density $f$ to $B_\ell \cup \tau_yB_\ell$, by letting, for any pair of configurations ${(\sigma_1, \sigma_2)\in \hat{\mathcal{E}}_{B_\ell}\times\hat{\mathcal{E}}_{\tau_yB_\ell}}$
\[\widetilde f_\ell^y(\sigma_1,\sigma_2):=\frac{1}{{\nu_{\rho,N}}(\sigma_1, \sigma_2)}\sum_{\substack{\eta \in\{0,1\}^{\T_N} \\ \eta_{|B_\ell}=\sigma_1\mbox{ and } \eta_{{|\tau_yB_\ell}}=\sigma_2}}\widetilde f(\eta) \nu_{\rho,N}(\eta).
\]
Following the same steps as in \cite[Section 5.5]{KL}, we obtain that, since $\mathcal{D}_N(f)\leq C/N$,
\begin{equation}
\label{eq:DirTBE} 
\max_{\sqrt{N}\leq |y|\leq 2\varepsilon N}\mathcal{D}^{\ell}_{N,y,1}(\widetilde f_\ell^y)+\mathcal{D}^{\ell}_{N,y,2}(\widetilde f_\ell^y)\leq C(\ell)o_{\varepsilon}(1).
\end{equation}
For any  pair of configurations ${(\sigma_1, \sigma_2)\in \hat{\mathcal{E}}_{B_\ell}\times\hat{\mathcal{E}}_{\tau_yB_\ell}}$, denote $j(\sigma_1, \sigma_2)$ their total number of particles 
\[j(\sigma_1, \sigma_2)=\sum_{x\in B_\ell}\sigma_1(x)+\sigma_2(y+x).\]
We claim that for any $\widetilde f_\ell^y$ satisfying \eqref{eq:DirTBE}, we can write for ${(\sigma_1,\sigma'_1, \sigma_2,\sigma'_2)\in \hat{\mathcal{E}}_{B_\ell}^2\times\hat{\mathcal{E}}_{\tau_yB_\ell}^2}$ such that $j(\sigma_1,\sigma_2)=j(\sigma'_1,\sigma'_2)$,
\begin{equation}
\label{eq:varTBE}
|\widetilde f_\ell^y(\sigma_1,\sigma_2)-\widetilde f_\ell^y(\sigma'_1,\sigma'_2)|\leq C'(\ell)o_{\varepsilon}(1), 
\end{equation}
for some large constant $C'(\ell)$ which does not depend neither on $\sqrt{N}\leq |y|\leq 2\varepsilon N$ nor on the chosen configurations $(\sigma_1,\sigma'_1, \sigma_2,\sigma'_2)$. This bound is proved straightforwardly by the  following three steps: 
\begin{itemize}
\item [--] Choose a sequence of licit jumps either in $B_\ell$, in $\tau_y B_\ell$, or between $0$ and $y$, allowing to reach $(\sigma'_1,\sigma'_2)$ from $(\sigma_1,\sigma_2)$. Uniformly in $(\sigma_1,\sigma_2)$ and $(\sigma'_1,\sigma'_2)$, the number $n:=n(\sigma_1,\sigma_2,\sigma'_1,\sigma'_2)$ of jumps necessary for the construction of this path can be crudely bounded from above by some constant $K_\ell$ independent of the configurations and of $y$. Denote $(\sigma_1,\sigma_2)=(\sigma_1^0,\sigma_2^0), (\sigma_1^1,\sigma_2^1),\dots,(\sigma_1^n,\sigma_2^n)=(\sigma_1',\sigma_2')$ such a path.
\item [--] Then,
\begin{equation}
\label{eq:dir1TBE} 
\big|\widetilde f_\ell^y(\sigma_1,\sigma_2)-\widetilde f_\ell^y(\sigma'_1,\sigma'_2)\big|\leq \sum_{k=0}^{n-1}\big|\widetilde f_\ell^y(\sigma^{k+1}_1,\sigma^{k+1}_2)-\widetilde f_\ell^y(\sigma^k_1,\sigma^k_2)\big|.
\end{equation}
\item [--] Since all the jumps from $(\sigma^k_1,\sigma^k_2)$ to $(\sigma^{k+1}_1,\sigma^{k+1}_2)$ are by assumption allowed for our non-ergodic dynamics, each of the terms in the sum in the right hand side above can be crudely bounded from above by
\begin{multline}
\label{eq:dir2TBE}
 2\cro{\|f_\ell^y\|_{\infty}\pa{ \sqrt{\widetilde f_\ell^y}(\sigma^{k+1}_1,\sigma^{k+1}_2)-\sqrt{\widetilde f_\ell^y}(\sigma^k_1,\sigma^k_2)}^2}^{1/2}\\
 \leq 2\cro{\frac{\|f_\ell^y\|_{\infty}\pa{ \mathcal{D}^{\ell}_{N,y,1}(\widetilde f_\ell^y)+\mathcal{D}^{\ell}_{N,y,2}(\widetilde f_\ell^y)}}{ \nu_{\rho,N}(\eta_{|B_\ell}=\sigma_1^k, \eta_{{|\tau_yB_\ell}}=\sigma_2^k)}}^{1/2}\leq \widetilde{C}(\ell)o_{\varepsilon}(1),
\end{multline}
where the last bound comes from a straightforward adaptation of Lemma \ref{lem:density}. 
\end{itemize}
Together, the two bounds \eqref{eq:dir1TBE} and \eqref{eq:dir2TBE} prove \eqref{eq:varTBE}, by letting $C'_\ell=K_\ell \widetilde{C}_\ell$.

\bigskip

We can now project $\nu_{\rho, N}$ as in the previous section on $B_\ell \cup\tau_y B_\ell$, and obtain that in the limit $N\to\infty$ and then $\varepsilon\to 0$, $f$ is ultimately constant on the sets with fixed number of particles in $ B_\ell\cup\tau_y B_\ell$. Note that the number of particles is not fixed in each of those two boxes, because we allowed jumps from $0$ to $y$. 
Specifically, denote \[\hat{\mathcal{E}}_{\ell,y}^k=\left\{(\sigma_1, \sigma_2)\in \hat{\mathcal{E}}_{B_\ell}\times\hat{\mathcal{E}}_{\tau_yB_\ell},\quad j(\sigma_1, \sigma_2)=k\right\}\]
the set of ergodic configurations on $B_\ell\cup\tau_yB_\ell$ with $k$ particles, and denote  $\nu_{\rho,N}^{y,\ell,k}$ the measure $\nu_{\rho,N}$ conditioned to $\hat{\mathcal{E}}_{\ell,y}^k$, defined for any $(\sigma_1, \sigma_2)\in\hat{\mathcal{E}}_{\ell,y}^k$ by
\[ \nu_{\rho,N}^{y,\ell,k}(\sigma_1, \sigma_2):= \frac{\nu_{\rho,N}(\eta_{|B_\ell}=\sigma_1, \; \eta_{\tau_y|B_\ell}=\sigma_2)}{\nu_{\rho,N}(j(\eta_{|B_\ell},\eta_{\tau_y|B_\ell})=k)}.\]
To prove \eqref{eq:TBE1}, thanks to \eqref{eq:varTBE}, it is therefore enough to prove that
\begin{multline}
\label{eq:TBE2}
\limsup_{\ell \to\infty}\limsup_{\varepsilon \to 0}\limsup_{N\to\infty} \sup_{\sqrt{N} \leqslant |y|\leqslant 2\varepsilon N}  \sup_{k}\;\\  
{\sum_{(\sigma_1,\sigma_2)\in\hat{\mathcal{E}}_{\ell,y}^k}}   \Big|\rho^{\ell}(0)(\sigma_1) - \rho^\ell(y)(\sigma_2)\Big| \nu_{\rho,N}^{y,\ell,k}(\sigma_1, \sigma_2)=0.
\end{multline}
Thanks to \eqref{eq:TBEmesure} and Corollary \ref{cor:correlations}, (recall formula \eqref{pirho} for $\pi_\rho(\xi_{|B_\ell}=\sigma_1)$), we can substitute for any $(\sigma_1,\sigma_2)\in\hat{\mathcal{E}}_{\ell,y}^k$
\begin{multline}
\frac{\pi_\rho(\xi_{|B_\ell}=\sigma_1)\pi_\rho(\xi_{|B_\ell}=\sigma_2)}{\sum_{(\sigma_1',\sigma_2')\in\hat{\mathcal{E}}_{\ell,y}^k}\pi_\rho(\xi_{|B_\ell}=\sigma_1)\pi_\rho(\xi_{|B_\ell}=\sigma_2)}\\
=\frac{\gamma^{\sigma_1(1)+\sigma_1(\ell)+\sigma_2(y+1)+\sigma_2(y+\ell)}}{\sum_{(\sigma_1',\sigma_2')\in\hat{\mathcal{E}}_{\ell,y}^k}\gamma^{\sigma'_1(1)+\sigma'_1(\ell)+\sigma'_2(y+1)+\sigma'_2(y+\ell)}}\leq \frac{C(\rho)}{\#\{(\sigma_1',\sigma_2')\in\hat{\mathcal{E}}_{\ell,y}^k\}} 
\end{multline}
to $ \nu_{\rho,N}^{y,\ell,k}(\sigma_1, \sigma_2)$ in equation \eqref{eq:TBE2}. 
For $k\leq 2|B_\ell|$, let us introduce the following measure  on $\{0,1,...,k\}$:
\[\nu_{k,\ell}(j)=\frac{\#\{(\sigma_1',\sigma_2')\in\hat{\mathcal{E}}_{\ell,y}^k, \quad \sum_{x\in B_\ell}\sigma_1(x)=j\}}{\#\{(\sigma_1',\sigma_2')\in\hat{\mathcal{E}}_{\ell,y}^k\}},\]
which does not depend on $y\geq \sqrt{N}$. All quantities are now independent from $N$,  $\varepsilon$ and $y$, therefore to prove \eqref{eq:TBE2} it is sufficient to show that 
\begin{equation}
\label{eq:TBE3}
\limsup_{\ell \to\infty} \sup_{k }\;\\  
\sum_{j\leq k} \frac{|2j-k|}{|B_\ell|}\nu_{k,\ell}(j)=0.
\end{equation}
It is now straightforward to project on the possible values at the border of the two configurations, and use Lemma \ref{lem:concentration2} to prove that the measure $\nu_{k,\ell}$ is concentrated as $\ell\to\infty$ on $j$ which is of order $k/2,$ in the sense that for any $\delta>1/2$, 
\[\inf_{k}\nu_{k,\ell}\big(\big|j-k/2\big|\leq\ell^{\delta}\big)\xrightarrow[\ell\to\infty]{} 1, \]
which proves \eqref{eq:TBE3} and concludes the proof of the two-blocks estimate.
\end{proof}

\subsection{Dirichlet estimates} We prove here an estimate for the Dirichlet form. 
Recall that we denote by $\mc D_N$ the Dirichlet form 
\begin{align*}
\mc D_N(g)&=-\nu_{\rho,N}\big(\sqrt{g} \mathcal{L}_N \sqrt{g}\big)\\
&=\frac{1}{2}\nu_{\rho,N}\bigg(\sum_{x\in\T_N}{\eta(x+2)\eta(x-1)}\big(\sqrt g(\eta^{x,x+1})-\sqrt g(\eta)\big)^2\bigg).\end{align*}

\begin{proposition}[Dirichlet form estimate]
\label{prop:DirEstimate}
For any time $T$, there exists a constant $C=C(T)$ such that
\[\mc D_N\pa{\frac{1}{T}\int_{0}^Tf_s^N ds}\leq \frac{C}{N}.\]
\end{proposition}

\begin{proof}[Proof of Proposition \ref{prop:DirEstimate}]
Let us define the usual relative entropy with respect to $\nu_{\rho,N}$ as 
\[\mc H_\rho^N(f)=\nu_{\rho,N}(f\log f),\] 
we have for any $s\in[0,T]$
\[\partial_s\mc H^N_\rho(f_s^N)=N^2\nu_{\rho,N}(f_s^N \mathcal{L}_N \log f_s^N)\leq -N^2\mc D_N(f_s^N),\]
so that 
\[0\leq\mc H_\rho^N(f_{T}^N)+N^2\int_{0}^T\mc D_N(f_s^N)ds\leq \mc H^N_\rho(f_{0}^N).\]
Thanks to \eqref{eq:nuNmax}, one obtains that for any fixed configuration $\eta$ in the ergodic component $\mathcal{E}_N$, we have $-\log\nu_{\rho,N}(\eta)\leq C_{\rho}N$, so that by convexity of the entropy \[\mc H^N_\rho(f_{0}^N)\leq C_{\rho}N,\]
and therefore
\[0\leq \int_{0}^T\mc D_N(f_s^N) ds\leq \frac CN.\]
The convexity of the Dirichlet form $\mc D_N$ concludes the proof. \end{proof}

\appendix 

\section{Exit times for random walkers} \label{app:exit}

In this section we \ccl{state} technical estimates related to exit times for random walkers, which we used in the proof of Lemma \ref{lemma:Tsigma} and Lemma \ref{lemma:ExitProba}.

\begin{lemma}\label{lemma:ExitTimeRW}
Let us denote by $(X_t)_{t\in \R_+}$ (resp.\ $(Y_k)_{k\in \N}$) the trajectory of a continuous (resp.\ discrete) time symmetric random walker on $\Z$, jumping left and right at the same rate $1/{(1+\delta)\ell}$ (resp.\ with the same probability $\frac12$). We denote by $\P^x$ the probability laws of both these processes starting from  site $x$. For any integer $\ell$, we denote by $T_\ell^{X}$ (resp.\ $T_\ell^Y$) the exit time of $B_\ell=\{-\ell,...,\ell\}$.

Then, there exist two constants $C:=C(\delta)>0$ and  $\bar\ell:=\bar \ell( \delta)>0$ such that for any $\ell>\bar \ell$, 
\begin{equation}
\label{eq:ExitTimeCont}
\P^0\big(T_\ell^X>\ell^4\big)\leq \exp\pa{-C \ell}.\end{equation}
Furthermore, for any $\varepsilon >0$, there exist two universal constants $C'>0$ and $\bar \ell'$ such that for any $\ell\geq\bar \ell'$
 \begin{equation}
 \label{eq:ExitTimeDisc}
 \P^0\big(T_\ell^Y>\ell^{2+\varepsilon}\big)\leq\exp\pa{-C' \ell^{\varepsilon}}.\end{equation}
\end{lemma}
\ccl{This is  a standard random walks estimate, we omit its proof and refer the interested reader for example to \cite{VarT}, p.173}.


\begin{thebibliography}{}
\bibitem{Andjel}
E. Andjel: \emph{Invariant measures for the zero range processes},
Ann. Probab. 10 (1982), no. 3, 525–547. 

\bibitem{BBCS}
J. Baik, G. Barraquand, I. Corwin, T. Suidan: \emph{Facilitated exclusion process}, arXiv:1707.01923

\bibitem{BM} U. Basu and P. K. Mohanty: \emph{Active-absorbing-state phase transition beyond directed percolation:
A class of exactly solvable models}, Phys. Rev. E 79, 041143 (2009).

\bibitem{BCSS} O. Blondel, C. Canc\`es, M. Sasada, M. Simon: \emph{Convergence of a degenerate microscopic dynamics to the porous medium equation}, arXiv:1802.05912.

\bibitem{Fu} T. Funaki:
\emph{Free boundary problem from stochastic lattice gas model}, Annales de l'I.H.P. Probabilit\'es et statistiques, Volume 35 (1999) no. 5 , p. 573--603.

\bibitem{FHU} T. Funaki, K. Handa, K. Uchiyama: \emph{ Hydrodynamic Limit of One-Dimensional Exclusion Processes with Speed Change}, 
 Ann. Probab. {\bf 19}  no. 1, 245--265 (1991).

\bibitem{FS} T. Funaki and M. Sasada: \emph{Hydrodynamic Limit for an Evolutional Model of Two-Dimensional Young Diagrams}, Comm. Math. Phys. {\bf 299}  335--363 (2010).

\bibitem{GKR} A. Gabel, P. L. Krapivsky, and S. Redner: \emph{Facilitated Asymmetric Exclusion},
Phys. Rev. Lett. 105, 210603 (2010).

\bibitem{GLT}
P. Gon\c calves, C. Landim, and C.   Toninelli: \emph{Hydrodynamic limit for a particle system with degenerate rates}, {Ann. IHP Probab. Stat.} {\bf 45} 887-909 (2009).

\bibitem{GQ} J. Gravner, J. Quastel:
\emph{Internal DLA and the Stefan problem}, 
Ann. Probab. 28 (2000), no. 4, 1528–1562.

\bibitem{HJV} F. Hern\'andez, M. Jara, and F. Valentim: \emph{Lattice model for fast diffusion equation}, arXiv:1706.06244. 

\bibitem{J}K. Jain: \emph{Simple sandpile model of active-absorbing state transitions}, Phys. Rev. E {\bf 72}, 017105 (2005).

\bibitem{KL} C. Kipnis and C. Landim: \textit{Scaling limits of interacting
 particle systems}. Grundlehren der Mathematischen Wissenschaften
  [Fundamental Principles of Mathematical Sciences], Volume 320.  Springer-Verlag, Berlin (1999).

\bibitem{LV}  C. Landim, G. Valle: \emph{A microscopic model for Stefan's melting and freezing problem}, Ann. Probab. 34 (2006), no. 2, 779–803.

\bibitem{L} S. L\"ubeck: \emph{Scaling behavior of the absorbing phase transition in a conserved lattice gas around the upper critical dimension},
Phys. Rev. E {\bf 64}, 016123 (2001).

\bibitem{russell} R. Lyons: \emph{Strong laws of large numbers for weakly correlated random variables}, Michigan Math. J. {\bf 35}(3), 353--359 (1988).

\bibitem{O} M. J. de Oliveira: \emph{Conserved lattice gas model with infinitely many absorbing states in one dimension
},  Phys. Rev. E {\bf 71}, 016112 (2005)

\bibitem{RPV} M. Rossi, R. Pastor-Satorras, and A. Vespignani: \emph{Universality Class of Absorbing Phase Transitions with a Conserved Field
}, Phys. Rev. Lett. {\bf 85}, 1803 (2000).

\bibitem{S} M. Sasada: \emph{Hydrodynamic limit for particle systems with
degenerate rates without exclusive constraints}, Alea {\bf 7}, 277--292 (2010).

\bibitem{Sondow} J. Sondow: \emph{Problem 11132}, American Mathematics Monthly 112, 180 (2005). 

\bibitem{VarT}S.R.S.Varadhan: \emph{Probability Theory}, Courant Institute of Mathematical Sciences, Lecture Notes, New York University (2000).

\bibitem{VY}
S. R. S. Varadhan and H.T. Yau: \emph{Diffusive limit of lattice gas with mixing conditions}, Asian J. Math., Vol.1, No. 4, 623--678 (1997).

\bibitem{vaz}
J.L. Vazquez: The Porous Medium Equation, Mathematical Theory. Oxford Mathematical Monographs, Clarendon Press, Oxford, 2007.

\end{thebibliography}
\end{document}